\RequirePackage{fix-cm}

\documentclass[smallextended]{svjour3}       % onecolumn (second format)
\smartqed  % flush right qed marks, e.g. at end of proof

\usepackage[algo2e,linesnumbered,vlined,ruled]{algorithm2e}
\usepackage{graphicx,graphics,subfigure}
\usepackage{hyperref,url}
\usepackage{cancel}
\usepackage{enumerate}
\usepackage{bbm}
\usepackage{comment}
\usepackage{times,hyperref}
\usepackage{amsmath,amsxtra,amsfonts,amscd,amssymb,bm}
\usepackage[normalem]{ulem}
\usepackage{exscale}
\usepackage{tikz}
\usepackage{float}
\usepackage{algorithm}
\usepackage{algorithmic}
\usepackage{multirow}
\usepackage{multicol}
\usepackage{makecell}
\usepackage{tabularx}

\newcommand{\be}{\begin{equation}}
\newcommand{\ee}{\end{equation}}
\newcommand{\bee}{\begin{equation*}}
\newcommand{\eee}{\end{equation*}}
\newcommand{\bea}{\begin{eqnarray}}
\newcommand{\eea}{\end{eqnarray}}
\newcommand{\beaa}{\begin{eqnarray*}}
\newcommand{\eeaa}{\end{eqnarray*}}

\newcommand{\st}{\,\textrm{s.t.}\,}  
\newcommand{\T}{\mathrm{T}}  
\newcommand{\KK}{\mathcal{K}} 
\newcommand{\KKc}{\mathcal{K}^\circ} 
 
\newcommand{\RR}{\mathbb{R}}
\newcommand{\BB}{\mathbb{B}}
\newcommand{\PP}{\mathcal{P}}

\newtheorem{assumption}[theorem]{Assumption}

\newcounter{cnt}
\foreach \num in {1,2,...,26}{%
  \stepcounter{cnt}%
  \expandafter\xdef \csname c\Alph{cnt}\endcsname {\noexpand\mathcal{\Alph{cnt}}}%
  \expandafter\xdef \csname b\Alph{cnt}\endcsname {\noexpand\mathbb{\Alph{cnt}}}%
}

\newcommand{\ox}{\overline{x}}
\newcommand{\oy}{\overline{y}}

\newcommand{\prox}{\mathbf{prox}} 
\newcommand{\iprod}[2]{\left\langle{#1},{#2}\right\rangle}
\newcommand{\nrm}[1]{\left\|#1\right\|}

\newcommand{\dom}{\mathrm{dom}} 
\newcommand{\dist}[1]{\mathrm{dist}\left(#1\right)} 
\newcommand{\distw}[2]{\mathrm{dist}_{#1}\left(#2\right)} 

\newcommand{\bigO}[1]{\mathcal{O}\left(#1\right)}

\newcommand{\abs}[1]{\left|#1\right|}
\renewcommand{\neg}[1]{\left[#1\right]_{-}}
\newcommand{\pos}[1]{\left[#1\right]_{+}}

\newcommand{\diag}{\operatorname{diag}}
\newcommand{\maxop}[1]{\max\left\{#1\right\}}

\DeclareMathOperator*{\argmin}{arg\,min}

\newcommand{\cpos}[1]{\mathcal{P}_{+}\left(#1\right)}
\newcommand{\cneg}[1]{\mathcal{P}_{-}\left(#1\right)}

\newcommand{\negp}[1]{\left[#1\right]_-}

\graphicspath{ {./figure} }
\allowdisplaybreaks

\begin{document}

\title{A Unified Primal-Dual Algorithm Framework for Inequality Constrained Problems
}

\author{Zhenyuan Zhu \and Fan Chen \and Junyu Zhang \and Zaiwen Wen  
}

\institute{Zhenyuan Zhu \at
              Beijing International Center for Mathematical Research, Peking University, Beijing, China\\
              \email{zhenyuanzhu@pku.edu.cn} 
           \and
           Fan Chen \at 
              School of Mathematical Science, Peking University, Beijing, China\\
              \email{chern@pku.edu.cn}
          \and
             Corresponding Author: Junyu Zhang \at
             Department of Industrial Systems Engineering and Management, National University of Singapore, Singapore, Singapore\\
             \email{junyuz@nus.edu.sg}
          \and
          Zaiwen Wen \at 
              Beijing International Center for Mathematical Research, Peking University, Beijing, China\\
              \email{wenzw@pku.edu.cn}\vspace{0.2cm}\\
          Zhenyuan Zhu and Fan Chen contribute equally to this paper. 
}

\date{}

\maketitle

\begin{abstract}
  In this paper, we propose a unified primal-dual algorithm framework based on the augmented Lagrangian function for composite convex problems with conic inequality constraints. The new framework is highly versatile. First, it not only covers many existing algorithms such as PDHG, Chambolle-Pock (CP), GDA, OGDA and linearized ALM, but also guides us to design a new efficient algorithm called Simi-OGDA (SOGDA). Second, it enables us to study the role of the augmented penalty term in the convergence analysis. Interestingly, a properly selected penalty  not only improves the numerical performance of the above methods, but also theoretically enables the convergence of algorithms like PDHG and SOGDA. 
  Under properly designed step sizes and penalty term, our unified framework preserves the $\mathcal{O}(1/N)$ ergodic convergence while not requiring any prior knowledge about the magnitude of the optimal Lagrangian multiplier. Linear convergence rate for affine equality constrained problem is also obtained given appropriate conditions. Finally, numerical experiments on linear programming, $\ell_1$ minimization problem,  and multi-block basis pursuit problem demonstrate the efficiency of our methods.
\end{abstract}

\section{Introduction}
In this work, we consider the following convex composite optimization problem:
\begin{equation}\label{prob}
    \begin{split}
    \min_x\ & \Phi(x):=f(x)+h(x),\\
    \st\ & Ax- b\in\KK.
    \end{split}
\end{equation}
In this problem, $f(x): \mathbb{R}^n\mapsto\mathbb{R}$ is a differentiable convex function whose gradient $\nabla f(\cdot)$ is $L_f$-Lipschitz continuous, $h(x): \mathbb{R}^n\mapsto\mathbb{R}$ is a simple convex function whose proximal operator can be efficiently evaluated, the set $\KK$ is either $\{0\}$ or a proper cone, and $A\in\mathbb{R}^{m\times n}$, $b\in \mathbb{R}^m$ are some matrix and vector. In particular,  \eqref{prob} becomes the standard affine equality constrained problem when $\KK=\{0\}$.

For problem \eqref{prob} with $\mathcal{K}=\{0\}$, the augmented Lagrangian method (ALM) and the alternating direction method of multipliers (ADMM) are the most popular algorithms. ALM  minimizes the augmented Lagrangian function with respect to primal variables and applies a gradient ascent step on the dual variable. Due to Danskin's theorem, ALM can thus be interpreted as applying gradient ascent method to the dual function. For multi-block problems with separable structure, ADMM replaces the exact primal minimization of ALM with one cycle of block coordinate minimization, which can be viewed as an approximate dual gradient ascent. However, the approximation in the primal minimization step makes the convergence of ADMM subtle. When the number of blocks equals two, the convergence of ADMM was established in the context of operator splitting \cite{eckstein1992douglas}. When the number of blocks is greater than two, a counterexample was constructed in \cite{chen2016direct}, indicating that the direct extension of ADMM may not necessarily converge for general multi-block problems. 
To ensure convergence, one must make further modification to the algorithm \cite{deng2017parallel,gao2019randomized}, or assuming certain extra conditions on the problem \cite{cai2017convergence,chen2013convergence,li2015convergent,lin2015global,lin2015sublinear}. 

In contrast to the (approximate) dual gradient methods that require (approximately) solving the primal problem in each iteration, the primal-dual methods like Chambolle-Pock \cite{chambolle2011first,chambolle2016ergodic}  update the primal and dual variables equally with the current gradient information. Thus they experience no trouble from the inexact primal minimization. 
So far, the primal-dual methods have not been extensively studied in the context of affinely constrained problems. Existing results mainly focus the saddle point problem:
\begin{equation}\label{generalminmax}
\begin{split}
\min_{x\in\mathcal{X}}\max_{y\in\mathcal{Y}}\ h(x)+\Psi(x,y)-s(y),
\end{split}
\end{equation}
where efficient projections onto $\cX$ and $\cY$ are available. One early algorithm for solving \eqref{generalminmax} is the Arrow-Hurwicz method \cite{uzawa1958iterative}, which is also known as the primal-dual hybrid gradient (PDHG) method \cite{zhu2008efficient}. The convergence of PDHG has been established when one of $h(\cdot)$ and $s(\cdot)$ is strongly convex, while a non-convergent example exists if strong convexity is not available \cite{he2014convergence}. %the convergence of PDHG can be guaranteed if we further assume one of the functions of the saddle point problem is strongly convex.
The Chambolle-Pock method (CP) \cite{chambolle2011first} is another famous primal-dual algorithm. % which adopts an extrapolation step for the primal variable before updating the dual variable. 
In terms of the duality gap, the CP method preserves an $\mathcal{O}(1/N)$ ergodic convergence rate, and is shown to be optimal \cite{zhang2021lower}.
It is worth mentioning that both PDHG and CP focus on the case where $\Psi(x,y)$ is bilinear, that is, $\Psi(x,y)=x^\T Ky$.
Recently, the optimistic gradient descent ascent (OGDA) method has gained popularity due to its superior empirical performance in GANs training \cite{daskalakis2017training}. 
For bilinear objective function, the convergence of OGDA was studied in \cite{daskalakis2017training,liang2019interaction}. By viewing OGDA as an approximate proximal point method,  the $\mathcal{O}(1/N)$ convergence was established  in \cite{mokhtari2020convergence}  for smooth convex-concave problem, as well as linear convergence for smooth strongly convex strongly concave problem.  Further attempt was made in \cite{wei2020linearOGDA} to prove the last-iterate linear convergence under metric subregularity.  To our best knowledge, the OGDA methods has yet been studied for  general non-smooth convex-concave problem.

Notice that most of the convergence results of PDHG, Chambolle-Pock and OGDA are established with respect to the duality gap of \eqref{generalminmax}. However, for the saddle point problem induced by the (augmented) Lagrangian function of some constrained optimization problem,  the duality gap is no longer a valid measure of convergence due to the unboundedness of the dual domain, see \cite{zhang2021primal}. One general way to avoid this issue is artificially adding a suitable norm constraint to the multipliers. However, properly constructing this constraint requires prior knowledge on the magnitude of the optimal Lagrangian multiplier, which is impractical in many applications. 
{For specific algorithms such as the Chambolle-Pock method \cite{hamedani2021primal} and the linearized ALM \cite{xu2021first,luo2021accelerated}, as well as their variants, the convergence without bounded dual domain has been established. However, similar guarantee is still not available for many other primal-dual methods like OGDA, our newly proposed SOGDA, and so on. }
In this work, we integrate several primal-dual methods into a unified algorithmic framework and establish a unified ergodic and nonergodic convergence for the main problem \eqref{prob}, without requiring prior knowledge on the bound of the optimal Lagrangian multiplier. The main contributions are summarized as follows:% boundedness assumption on the dual range.

\begin{itemize}
\item We propose a unified primal-dual algorithm framework for solving \eqref{prob}. It covers several well-known algorithms, such as PDHG, Chambolle-Pock, GDA, OGDA and linearized ALM. By properly selecting the parameters of our framework, we also introduce a new algorithm called Semi-OGDA, abbreviated as SOGDA, which demonstrates promising performance in the numerical experiments.\vspace{0.2cm}
\item By analyzing the constraint violation and function value gap instead of the duality gap, we establish a unified analysis for our framework to achieve an $\mathcal{O}(1/N)$ ergodic convergence for problem \eqref{prob}, without requiring any prior knowledge on the magnitude of the optimal multiplier. Furthermore, the non-ergodic convergence of our algorithm framework can also be established. Under suitable conditions, our algorithm achieves linear convergence. \vspace{0.2cm}
\item Unlike the standard primal-dual algorithms like PDHG who solve a minimax problem based on the Lagrangian function of problem \eqref{prob}, we build our framework upon the augmented Lagrangian function. On the one hand, our analysis shows that adding an augmented penalty term makes the constraint $y\in\KK^*$ optional for the Lagrangian multiplier $y$. This can often improve the flexibility of the algorithm design since one can fully remove this constraint if it causes difficulty in solving subproblems. We also prove that the augmented penalty term can  theoretically improve the convergence guarantee for PDHG and SOGDA.  On the other hand, algorithms with an augmented Lagrangian function usually converge faster than those based on the Lagrangian function in numerical experiments.\vspace{0.2cm}
\item Our analysis framework can be easily extended beyond bilinear saddle point problem for augmented Lagrangian function. As a byproduct of our analysis, we derive the convergence result of the proximal OGDA, generalization of OGDA to non-smooth problem.
\end{itemize}

\subsection{Notations}
We use $\langle \cdot, \cdot\rangle$ to denote Euclidean inner product, and we use $\|\cdot\|$ to denote the Euclidean norm for a vector or the spectral norm for a matrix. For any symmetric matrix $A$, we define the weighted inner product $\langle x, y\rangle_A:=\langle x, Ay\rangle = \langle Ax, y\rangle$, and when $A\succeq0$, we define the corresponding norm $\|x\|_A:=\sqrt{\langle x, x\rangle_A}$. For any set $\mathcal{X}$, the indicator function is defined as $\mathbbm{1}_\mathcal{X}(x)=0$ if $x\in\mathcal{X}$ and $\mathbbm{1}_\mathcal{X}(x)=+\infty$ if $x\notin \mathcal{X}$. The distance between the point $x$ and the set $\mathcal{X}$ is defined as $\dist{x,\mathcal{X}}:=\min_{y\in\mathcal{X}}\|x-y\|_2$, and more generally, $\distw{A}{x,\mathcal{X}}:=\min_{y\in\mathcal{X}}\|x-y\|_A$. For any cone $\KK$, we denote its dual cone by $\KK^*$ and define the polar cone as $\KK^\circ:=-\KK^*$. Let $\mathcal{S}$ be a closed set, we use $\mathcal{P}_{\mathcal{S}}(\cdot)$ to denote the projection operator to $\mathcal{S}$. For simplicity, we will write $\cpos{\cdot}:=\PP_{\KK^\circ}(\cdot)$ and $\cneg{\cdot}:=-\mathcal{P}_{\KK}\left(\cdot\right)$ throughout the paper. %It is worth noting that when $\KK=\RR^{n}_{\leq 0}$, we have $\cpos{\cdot}=\pos{\cdot}$ being the standard entry-wise positive part, and $\cneg{\cdot}=\neg{\cdot}$ also.

\subsection{Basic assumptions}
\begin{assumption}\label{assumption1}
    $f(x)$ is a convex differentiable function and $h(x)$ is a lower semi-continuous convex function.
\end{assumption}
\begin{assumption}\label{assumption2}
	The gradient of $f(x)$ is $L_f$-Lipschitz continuous, that is,
	$$\|\nabla f(x)-\nabla f(x')\|\leq L_f\|x-x'\|, \quad \forall x,x'.$$ 
\end{assumption}
For the convenience of notation, we define $f_\rho(x):=f(x)+\frac{\rho}{2}\|Ax-b\|^2$ and let $L_{f_\rho}:=L_f+\rho\|A\|^2$ be the Lipschitz constant of $\nabla f_\rho(x)$.
\begin{assumption}\label{assumption3}
    The optimal solution of \eqref{prob} is attainable, that is, there exists $x^*\in\RR^n$ such that $Ax^*-b\in\KK$ and $\Phi(x^*)$ equals to the optimal value $\Phi^*$.
\end{assumption}
\begin{assumption}[Slater's condition]\label{Slater's}
    For the convex problem \eqref{prob},
    let $\mathcal{D}=\dom\ \Phi$ and $\mathrm{relint}\ \mathcal{D}$ denote the relative interior of the set $\mathcal{D}$. 
    There exists $x\in \mathrm{relint}\ \mathcal{D}$ such that $Ax-b\in\mathrm{int}\ \KK$, where $\mathrm{int}\ \KK$ denotes the interior of the cone $\KK$.
    When the cone $\KK$ is polyhedral (including the case of $\KK=\{0\}$), the condition can be relaxed to the existence of $x\in \mathrm{relint}\ \mathcal{D}$ such that $Ax-b\in\KK$.
\end{assumption}

\section{Primal-dual methods for conic inequality constrained problems}\label{ogdasec}

\subsection{The augmented Lagrangian duality}\label{aldual}
When $\KK=\{0\}$ or $\KK$ is a proper cone that has nonempty interior, the strong duality holds under Assumptions \ref{assumption1} and \ref{Slater's}. 
It implies that the minimization problem \eqref{prob} is equivalent to a saddle point problem:
\begin{eqnarray}
	\label{saddleprob0}
	\min_x\max_{y\in\KK^*}\ \mathcal{L}(x,y) = f(x)+h(x)-y^\T(Ax-b).
\end{eqnarray} 
In this section, we generalize the strong duality to the saddle point formulation of problem \eqref{prob} based on the augmented Lagrangian function. Importantly, by incorporating the augmented penalty term, the dual constraint $y\in\KK^*$ becomes optional. One can fully remove this constraint if it causes difficulty in solving subproblems. 
Although the augmented duality for inequality constraint has been discussed in \cite{rockafellar2009variational}, we provide a simpler proof for problem \eqref{prob}.
\begin{lemma}\label{lemma1} 
	Suppose that $\KK$ is $\{0\}$ or a proper cone. Given any penalty coefficient $\rho>0$, we define the augmented Lagrangian function as
\begin{equation}\label{lagrangian}
    \mathcal{L}_\rho(x,y):=f(x)+h(x)+\frac{\rho}{2}\nrm{ \cpos{Ax-b-\frac{y}{\rho}} }^2-\frac{\nrm{y}^2}{2\rho}.
\end{equation}
Under Assumptions \ref{assumption1} and \ref{Slater's}, the strong duality holds for $\mathcal{L}_\rho(x,y)$, that is,
\begin{equation}\label{strongdual}
    \min_{x}\max_{y}\ \mathcal{L}_\rho(x,y)=\max_{y}\min_{x}\ \mathcal{L}_\rho(x,y),
\end{equation}
where both sides of \eqref{strongdual} are equivalent to problem \eqref{prob}.
\end{lemma}
\begin{proof}
	We first consider the case where $\KK$ is a proper cone.
    For any given $\rho>0$, we introduce a slack variable $\pi$ to the conic inequality constraint and consider the following equivalent formulation of \eqref{prob}:
    \begin{equation}\label{eqprob}
        \begin{split}
        \min_{\pi\in-\KK, x}\ &f(x)+h(x) + \frac{\rho}{2}\|Ax-b+\pi\|^2,\\
        \st\ &Ax-b+\pi=0,
        \end{split}
    \end{equation}
whose Lagrangian function equals $\mathcal{L}(x,\pi,y):= f(x)+h(x)-y^\T(Ax-b+\pi)+\frac{\rho}{2}\|Ax-b+\pi\|^2$. 
By Assumption \ref{Slater's}, there exists $x\in \mathrm{relint}\ \mathcal{D}$ and $\pi\in\mathrm{int}\ \KK$ such that $Ax-b+\pi=0$, which indicates that Slater's condition holds for problem \eqref{eqprob}.
Thus, the strong duality implies
\begin{equation}
	\label{eqn:lm1-1}
	\min_{\pi\in-\KK, x}\max_y\,\,\mathcal{L}(x,\pi,y) \,\,=\,\, \max_y\min_{\pi\in-\KK, x}\,\,\mathcal{L}(x,\pi,y).
\end{equation} 
This saddle point problem \eqref{eqn:lm1-1} is equivalent to the original minimization problem \eqref{prob}. Note that for any $x$ and $y$, the minimization over $\pi$ admits a closed form solution:
$$\argmin_{\pi\in-\KK}\ \mathcal{L}(x,\pi,y) = \mathcal{P}_-\Big(Ax-b-\frac{y}{\rho}\Big).$$
Substituting this solution to $\mathcal{L}$ gives the function $\mathcal{L}_\rho(x,y)$ defined by \eqref{lagrangian}. That is, 
\begin{equation}
	\label{eqn:lm1-2}
	\max_y\min_{\pi\in-\KK, x}\,\,\mathcal{L}(x,\pi,y)\,\,=\,\,\max_y\min_{x}\mathcal{L}_\rho(x,y).
\end{equation} 
To complete the proof, it remains to show:
\begin{equation}\label{eqn:lm1-3}
	\min_x\max_{y}\ \mathcal{L}_\rho(x,y)=\min_{\pi\in-\KK,x}\max_{y}\ \mathcal{L}(x,\pi,y).
\end{equation}
First, when $Ax-b\notin\KK$, it is straightforward to see that 
\begin{eqnarray}
	\label{eqn:lm1-4}
	\max_{y} \mathcal{L}_\rho(x,y) & = & \max_{y}\min_{\pi\in-\KK} f(x)+h(x)-y^\T(Ax-b+\pi)+\frac{\rho}{2}\|Ax-b+\pi\|^2 \nonumber\\
	& \geq & \max_{y}\min_{\pi\in-\KK} f(x)+h(x) -y^\T(Ax-b+\pi)\nonumber\\
	& \geq & \max_{y\in\KK^*}\min_{\pi\in-\KK} f(x)+h(x)-y^\T(Ax-b+\pi)\\
	& = & \max_{y\in\KK^*} f(x)+h(x)-y^\T(Ax-b)\nonumber\\
	& = & +\infty = \max_y \cL(x,\pi,y).\nonumber
\end{eqnarray}
Second, when $Ax-b\in\KK$, the closed form solution to \eqref{eqn:lm1-3} can be obtained given any fixed $x$. In details, $y=0$ and $\pi = \cneg{Ax-b}$ are the optimal solution to $\min_{\pi\in-\KK}\max_{y}\ \mathcal{L}(x,\pi,y)$. Thus, combined with \eqref{eqn:lm1-4}, the right hand side of \eqref{eqn:lm1-3} is equivalent to \eqref{prob}. For the left hand side of \eqref{eqn:lm1-3}, the optimal solution of $y$ is $y=0$ when $Ax-b\in\KK$. Thus, combined with \eqref{eqn:lm1-4}, the left hand side of \eqref{eqn:lm1-3} is also equivalent to \eqref{prob}. Therefore, combining \eqref{eqn:lm1-1}-\eqref{eqn:lm1-3} yields the conclusion. 

For the case of $\KK=\{0\}$, it always holds that $\pi=0$ and the conclusion can be proved similarly.
\qed
\end{proof}

\begin{remark}\label{yink*}
    According to the KKT condition of \eqref{prob}, it can be proved that the optimal dual variable $y^*\in \KK^*$.
    Therefore, \eqref{strongdual} still holds when $y$ is restricted to $\KK^*$, that is, 
    \begin{equation}\label{strongdual''}
        \min_{x}\max_{y\in\KK^*}\ \mathcal{L}_\rho(x,y)=\max_{y\in\KK^*}\min_{x}\ \mathcal{L}_\rho(x,y).
    \end{equation}
\end{remark}
As a result of Lemma \ref{lemma1}, solving the original minimization problem \eqref{prob} is equivalent to solving the following saddle point problem:
\begin{equation} \label{saddleprob}
	\min_x\max_y\ \mathcal{L}_\rho(x,y).
\end{equation}

\subsection{A unified primal-dual algorithm framework}
In this section, we propose a unified primal-dual algorithm framework to solve problem \eqref{saddleprob}, as well as problem \eqref{saddleprob0}. For the ease of notation, we define
$$s(y)=
\begin{cases}
	0,&{\rho>0,}\\
	\mathbbm{1}_{\KK^*}(y),&{\rho=0,}
\end{cases} $$
and 
$$\Psi(x,y) =
\begin{cases}
	{f(x)+\frac{\rho}{2}\nrm{ \cpos{Ax-b-\frac{y}{\rho}} }^2-\frac{\nrm{y}^2}{2\rho},} & {\rho>0,}\\
	{f(x)-y^\T(Ax-b),} & {\rho=0.} 
\end{cases}$$
Depending on the value of $\rho$, we can rewrite both problem \eqref{saddleprob0} and \eqref{saddleprob} as 
\begin{equation}\label{minmax}
    \min_x\max_y\ \cL_\rho(x,y):= h(x)+\Psi(x,y)-s(y).
\end{equation}
Let the proximal operator of an arbitrary function $f$ be 
\[\prox_{f}(x):=\argmin_u \{f(u)+\frac12\|x-u\|^2\}.\] Given this notation, we propose the following primal-dual algorithm framework:

\begin{algorithm2e}[H]\label{alg: drs-pf}
	\caption{A unified primal-dual algorithm framework}
	Input: penalty factor $\rho\geq0$; primal and dual step sizes $\tau, \sigma>0$;
	gradient extrapolation coefficients $\alpha\in[0,1]$ and $\beta$;
	the ratio of Gauss-Seidel iteration versus Jacobian iteration $\mu\in[0,1]$. \\
	Initialize: $x^{-1} = x^0, y^{-1} = y^0 = 0$. \\
	\For{$k=0,1,2,\dots$}
	{
	  Compute $g_x^{k} = \nabla_x\Psi(x^k,y^k), g_y^{k} = \nabla_y\Psi(x^k,y^k)$.\\
	  Perform the update:
	  \begin{equation}\label{generalalgo}
		\begin{split}
		x^{k+1} &= \prox_{\tau h}\left[x^k-\tau \left((1+\alpha) g_x^{k}-\alpha g_x^{k-1}\right)\right],\\
		y^{k+1} &= \prox_{\sigma s}\left[y^k+\sigma\mu\left((1+\beta) g_y^{k}-\beta g_y^{k-1}\right)\right. \\
		&\qquad\qquad\quad\left.+\sigma(1-\mu)\left((1+\beta) g_y^{k+1}-\beta g_y^{k}\right)\right],
		\end{split}
	  \end{equation}
	}
\end{algorithm2e}
When $\rho=0$, the gradients $g_x^{k}$ and $g_y^{k}$ are easy to obtain. When $\rho>0$, we have
\begin{align*}
	g_x^k&=\nabla f(x^k)+\rho A^\T\cpos{w^k},\\
	g_y^k&=-\frac{y^k}{\rho}-\cpos{w^k}=-(Ax^k-b)-\cneg{w^k},
\end{align*}
where $w^{k}=Ax^k-b-\frac {y^k}\rho$.
Note that the gradient $g_y^{k+1}=\nabla_y\Psi(x^{k+1},y^{k+1})$ in the $y^{k+1}$ update rule involves $y^{k+1}$ itself if $\rho>0$, the algorithm defined by \eqref{generalalgo} is potentially an implicit scheme. Fortunately, thanks to the fact that we can set $s(y) = 0$ since $y\in\KK^*$ is optional when $\rho>0$, the $y^{k+1}$ update can equivalently be rewritten in an explicit form, which is described as follows.  
\begin{lemma}\label{s0cone}
	Let $s(y)=0$ and $\rho>0$, the dual update rule of \eqref{generalalgo} can be written as $$y^{k+1} = \omega - \frac{\kappa}{\kappa+1}\cdot\cneg{\nu-\omega},$$
	where $\omega = y^k+\sigma\mu\left((1+\beta) g_y^{k}-\beta g_y^{k-1}\right)-\sigma(1-\mu)\left((1+\beta) (Ax^{k+1}-b)+\beta g_y^{k}\right)$, $\kappa = \sigma(1-\mu)(1+\beta)/\rho\geq0$, and
	$\nu = \rho(Ax^{k+1}-b)$.
\end{lemma}
\begin{proof}
	According to \eqref{generalalgo}, $y^{k+1}$ is the solution to the following equation:
	\[
	y^{k+1} = \omega-\kappa\cneg{\nu-y^{k+1}}.
	\]
	Due to Moreau's decomposition in Theorem \ref{Moreau} in appendix, we have 
	$$\nu-y^{k+1} = \cpos{\nu-y^{k+1}}-\cneg{\nu-y^{k+1}}.$$
	Combining the above two equations gives 
	\[
	\cpos{\nu-y^{k+1}}-(\kappa+1)\cneg{\nu-y^{k+1}}=\nu-\omega,
	\]
	which further indicates that 
	\begin{align*}
		\cpos{\nu-y^{k+1}} = \cpos{\nu-\omega}\quad\mbox{and}\quad 		\cneg{\nu-y^{k+1}} = \frac{\cneg{\nu-\omega}}{\kappa+1}.
	\end{align*}
	Therefore, $y^{k+1} = \omega - \frac{\kappa}{\kappa+1}\cdot\cneg{\nu-\omega}$, which completes the proof.
	\qed
\end{proof}

\begin{remark}\label{rmk:explicit}
    By Remark \ref{yink*}, one can also set $s(y)=\mathbbm{1}_{\KK^*}(y)$ when $\rho>0$. 
    Numerical results show that the performance may benefit from this setting. However, the explicit dual update of \eqref{generalalgo} can only be obtained in specific cases such as $\KK=\{x|x\leq0\}$, see Lemma \ref{xleq0} in the appendix.
    Otherwise, we can utilize the following explicit variant of the framework in practice:
    \begin{equation}\label{generalalgo-explicit}
        \begin{split}
        x^{k+1} &= \prox_{\tau h}\left[x^k-\tau \left((1+\alpha) g_x^{k}-\alpha g_x^{k-1}\right)\right],\\
        y^{k+1} &= \prox_{\sigma s}\left[y^k+\sigma\mu\left((1+\beta) g_y^{k}-\beta g_y^{k-1}\right)\right. \\
        &\qquad\qquad\quad\left.+\sigma(1-\mu)\left((1+\beta) g_y^{k+1,k}-\beta g_y^{k}\right)\right],
        \end{split}
    \end{equation}
    where $g_y^{k+1,k}:= \nabla_y \Psi(x^{k+1},y^k)$.
    It is worth mentioning that the implicit scheme \eqref{generalalgo} and the explicit scheme \eqref{generalalgo-explicit} are equivalent when $\KK=\{0\}$ or $\rho=0$.
\end{remark}

\subsection{Consequences of the unified framework}
In the unified primal-dual algorithm framework \eqref{generalalgo}, one is allowed to choose a wide range of algorithmic parameters. In fact, under appropriate parameters, our framework not only covers several well-known primal-dual algorithms such as PDHG, Chambolle-Pock, GDA, OGDA and linearized ALM, but also leads to a new efficient algorithm which we call SOGDA. Next, let us discuss the detailed consequences of our framework under different parameter specifications. \vspace{0.2cm}

\noindent\textbf{SOGDA} \,\, Let us set $\mu=1, \alpha=0, \beta=1, \rho\geq0$ in \eqref{generalalgo}, the scheme becomes a new algorithm which we call Semi-OGDA (SOGDA):
\begin{align*}
	x^{k+1} &= \prox_{\tau h}\left[x^k-\tau g_x^{k}\right],\\
	y^{k+1} &= \prox_{\sigma s}\left[y^k+\sigma\left(2 g_y^{k}- g_y^{k-1}\right) \right].  
\end{align*}
When $\KK=\{0\}$ {and} $\rho=0$, this method can be interpreted as the OGDA method with only dual variables extrapolated or a Jacobian variant of the Chambolle-Pock method.\vspace{0.2cm}

\noindent\textbf{PDHG and Chembolle-Pock}\,\, Now let us consider the Gauss-Seidel iteration without primal extrapolation, which corresponds to $\mu=\alpha=0$. If we further remove the dual extrapolation by setting  $\beta=0$, then we obtain the following scheme:
\begin{equation}\label{pdhgscheme}
	\begin{split}
		x^{k+1} &= \prox_{\tau h}\left[x^{k}-\tau g_x^{k}\right],\\
		y^{k+1} &= \prox_{\sigma s}\left[y^{k}+\sigma g_y^{k+1} \right].
	\end{split}
\end{equation}
Alternatively, if we allow the dual extrapolation and set $\beta=1$, the scheme becomes
\begin{equation}\label{cpscheme}
	\begin{split}
		x^{k+1} &= \prox_{\tau h}\left[x^k-\tau g_x^{k}\right],\\
		y^{k+1} &= \prox_{\sigma s}\left[y^k+\sigma\left(2 g_y^{k+1}- g_y^{k}\right) \right].
	\end{split}
\end{equation}
When $\KK=\{0\}$ {and} $\rho=0$, the schemes \eqref{pdhgscheme} and \eqref{cpscheme} are PDHG and Chambolle-Pock, respectively. Therefore, \eqref{pdhgscheme} and \eqref{cpscheme} can be viewed as the natural extensions of PDHG and Chambolle-Pock to the conic inequality constrained problem with an augmented penalty term. \vspace{0.2cm}

\noindent\textbf{GDA and OGDA}\,\, If we consider the algorithms of Jacobian iteration, then we can set $\mu=1$. In addition, if we omit both primal and dual extrapolation by setting $\alpha=\beta=0$, then \eqref{generalalgo} becomes
\begin{equation}\label{gdascheme}
	\begin{split}
		x^{k+1} &= \prox_{\tau h}\left[x^k-\tau g_x^{k}\right],\\
		y^{k+1} &= \prox_{\sigma s}\left[y^k+\sigma g_y^{k}\right].
	\end{split}
\end{equation}
Alternatively, if we allow primal and dual extrapolation with $\alpha=\beta=1$, we obtain
\begin{equation}\label{ogdascheme}
	\begin{split}
		x^{k+1} &= \prox_{\tau h}\left[x^k-\tau \left(2 g_x^{k}-g_x^{k-1}\right)\right],\\
		y^{k+1} &= \prox_{\sigma s}\left[y^k+\sigma\left(2 g_y^{k}- g_y^{k-1}\right) \right].
	\end{split}
\end{equation}
The scheme \eqref{gdascheme} is GDA and the scheme \eqref{ogdascheme} is a proximal variant of OGDA. Since the existing works of OGDA mainly focus on differentiable problems, \eqref{ogdascheme} is a direct extension of OGDA on the non-differentiable saddle point problems. \vspace{0.2cm}

\noindent\textbf{Linearized ALM and Jacobian linearized ADMM}\,\,  Lastly, if we further require $\rho>0$ and $\KK=\{0\}$ in the scheme \eqref{pdhgscheme}, it becomes the linearized ALM \cite{yang2013linearized}. If the primal variable can be split into multiple blocks in addition, then algorithm \eqref{pdhgscheme} can also be viewed as Jacobian linearized ADMM proposed in \cite{deng2017parallel}.

\section{The ergodic convergence}\label{sec:conv}
\subsection{Preliminary result}

Note that the general scheme \eqref{generalalgo} covers a wide range of primal-dual algorithms. In this section, we establish an $\mathcal{O}(1/N)$ ergodic convergence for \eqref{generalalgo} with a unified analysis. %The convergence of the implicit version for the affinely constrained problems can be derived equivalently.
Note that for any $(\bar x,\bar y)$, due to \eqref{eqn:lm1-4}, the duality gap for problem \eqref{saddleprob} remains $+\infty$ whenever  $A\bar{x}-b\notin\KK$:
\begin{equation}\label{dualgap}
	\max_{y\in\mathcal{Y}}\ \cL_\rho(\bar{x},y)-\min_{x\in\mathcal{X}}\ \cL_\rho(x,\bar{y})=+\infty.
\end{equation} 
Therefore, the duality gap is not a reasonable measure of the convergence for the conic constrained problems, which invalidates most existing analysis for general saddle point algorithms that establish their convergence in terms of the duality gap. In this paper, we carefully utilize the structure of the original minimization problem and give an error bound in terms of objective function gap and the constraint violation. First of all, let us provide a lemma that controls the objective function gap and the constraint violation separately by bounding their weighted summation.

 \begin{lemma}\label{lemma:sep-gap}
    {Suppose that Assumptions \ref{assumption1}, \ref{assumption3} and \ref{Slater's} are satisfied.}
    Let $(x^*,y^*)$ be a pair of primal and dual optimal solution of problem \eqref{prob}. For any positive constant $\gamma>\nrm{y^*}$, if the following inequality holds
    \begin{align*}
        \Phi(x)-\Phi(x^*)+\gamma\nrm{\cpos{Ax-b}}\leq\delta,
    \end{align*}
    then it holds that
    \[\begin{aligned}
    -\frac{\nrm{y^*}}{\gamma-\nrm{y^*}}\delta\leq \Phi(x)-\Phi(x^*)\leq \delta, \qquad
    \nrm{\cpos{Ax-b}}\leq \frac{\delta}{\gamma-\nrm{y^*}}.
    \end{aligned}\]
\end{lemma}
\begin{proof} 
First, by direct computation, we have
\begin{align*}
&\Phi(x)-\Phi(x^*)+\gamma\nrm{\cpos{Ax-b}}\\
\stackrel{(i)}{\geq} & \left(\gamma-\nrm{y^*}\right)\nrm{\cpos{Ax-b}}+\Phi(x)-\Phi(x^*)-\iprod{y^*}{\cpos{Ax-b}}\\
=& \left(\gamma-\nrm{y^*}\right)\nrm{\cpos{Ax-b}}+\Phi(x)-\Phi(x^*)-\iprod{y^*}{Ax-b}\\
&-\iprod{y^*}{\cpos{Ax-b}-(Ax-b)}\\
\stackrel{(ii)}{\geq} & \left(\gamma-\nrm{y^*}\right)\nrm{\cpos{Ax-b}}+\Phi(x)-\Phi(x^*)-\iprod{A^\T y^*}{x-x^*}\\
\stackrel{(iii)}{\geq}& \left(\gamma-\nrm{y^*}\right)\nrm{\cpos{Ax-b}},
\end{align*}
where (i) is due to the Cauchy-Schwarz inequality, 
(ii) uses the fact that $\cpos{u}-u\in-\KK$, $y^*\in\KK^*$ and $\iprod{y^*}{Ax^*-b}=0$,
(iii) comes from the fact that $A^\T y^*\in \partial \Phi(x^*)$ and the convexity of $\Phi$, that is,  $\Phi(x)\geq \Phi(x^*)+\iprod{\partial \Phi(x^*)}{x-x^*}.$
Therefore, as long as $\gamma>\nrm{y^*}$, it holds that $\nrm{\cpos{Ax-b}}\leq \frac{\delta}{\gamma-\nrm{y^*}}$, and hence
\begin{align*}
    \Phi(x)-\Phi(x^*)\geq \iprod{y^*}{Ax-b} \geq -\nrm{y^*}\nrm{\cpos{Ax-b}}
    \geq -\frac{\nrm{y^*}}{\gamma-\nrm{y^*}}\delta,
\end{align*}
which completes the proof of this lemma. \qed
\end{proof}

Let $z^k=[x^k;y^k]\in\RR^{m+n}$ be the iterates generated by the general primal-dual method given in \eqref{generalalgo}.
Define the operator $F:\mathbb{R}^{m+n}\rightarrow\mathbb{R}^{m+n}$ as
\begin{equation}\label{def:F}
    F(z)=[\nabla_x\Psi(z);-\nabla_y\Psi(z)].
\end{equation}
Define the function $R(z):=h(x)+s(y)$ and the scaling matrices 
\[
\Lambda=\left[\begin{array}{ll}
\tau I_n &  \\
& \sigma I_m\\
\end{array}\right],
\qquad
\Theta=\left[\begin{array}{ll}
\alpha I_n &  \\
& \beta I_m\\
\end{array}\right],
\qquad
\Xi=\left[\begin{array}{ll}
I_n &  \\
& \mu I_m\\
\end{array}\right].
\]
Then our unified primal-dual algorithm framework \eqref{generalalgo} can be rewritten as
\begin{align*}
    z^{k+1}=\prox_{\Lambda R}\big(
        z^{k}&-\Lambda\Xi\left[\left(I+\Theta\right)F(z^k)-\Theta F(z^{k-1})\right]\\
        &-\Lambda(I-\Xi)\left[\left(I+\Theta\right)F(z^{k+1})-\Theta F(z^{k})\right]
    \big).
\end{align*}
In the next lemma, we characterize the one-step behavior of this algorithm.
\begin{lemma}\label{lemma:one-step}
    Under Assumption \ref{assumption1}, for any $z$, it holds that  
    \begin{eqnarray}\label{onehand**}
    	&&R(z^{k+1})-R(z)+\left\langle F(z^{k+1}), z^{k+1}-z\right\rangle\nonumber\\
    	&\leq\ & \frac1{2}\|z^k-z\|^2_{\Lambda^{-1}}-\frac1{2}\|z^{k+1}-z\|^2_{\Lambda^{-1}}-\frac1{2}\nrm{ z^{k+1}-z^{k} }^2_{\Lambda^{-1}}\\
    	&&+\! \iprod{ F(z^k)-F(z^{k-1}) }{ z-z^{k} }_{\Xi\Theta} - \iprod{ F(z^{k+1})-F(z^k) }{ z-z^{k+1} }_{\Xi\Theta} \nonumber\\
    	&&+\! \iprod{F(z^{k})\!-\!F(z^{k-1})}{z^k\!-\!z^{k+1} }_{\!\Xi\Theta}+\langle F(z^{k+1})-F(z^k), z-z^{k+1}\rangle_{\Theta-\Xi}.\nonumber
    \end{eqnarray}  
\end{lemma}
\begin{proof}
    By the optimality condition of the proximal subproblem in \eqref{generalalgo}, there exists a vector $ v^{k+1}\in\partial R(z^{k+1})$  such that
    \begin{equation}\label{subgrad*}
        \begin{split}
            0=v^{k+1}+\Lambda^{-1}(z^{k+1}-z^k)
            &+\Xi\left[\left(I+\Theta\right)F(z^k)-\Theta F(z^{k-1})\right]\\
            &+(I-\Xi)\left[\left(I+\Theta\right)F(z^{k+1})-\Theta F(z^{k})\right].
        \end{split}
    \end{equation}
Taking inner product between \eqref{subgrad*} and $z^{k+1}-z$ and then rearranging the terms, we get
\begin{eqnarray}
	\label{onehand*}
	&&\langle v^{k+1}, z^{k+1}-z\rangle+\langle F(z^{k+1}), z^{k+1}-z\rangle\\
	&=& \langle z^{k+1}-z^k,z-z^{k+1}\rangle_{\Lambda^{-1}}+\langle F(z^k)-F(z^{k-1}), z-z^{k+1}\rangle_{\Xi\Theta}\nonumber\\
	& & -\langle F(z^{k+1})-F(z^k), z-z^{k+1}\rangle_{\Xi\Theta}+\langle F(z^{k+1})-F(z^k), z-z^{k+1}\rangle_{\Theta-\Xi}\nonumber\\
	& = & \frac1{2}\|z^k-z\|^2_{\Lambda^{-1}}-\frac1{2}\|z^{k+1}-z\|^2_{\Lambda^{-1}}-\frac1{2}\nrm{ z^{k+1}-z^{k} }^2_{\Lambda^{-1}}\nonumber\\
	& & +\langle F(z^k)-F(z^{k-1}), z-z^{k+1}\rangle_{\Xi\Theta} -\langle F(z^{k+1})-F(z^k), z-z^{k+1}\rangle_{\Xi\Theta}\nonumber\\
	& & +\langle F(z^{k+1})-F(z^k), z-z^{k+1}\rangle_{\Theta-\Xi}\nonumber,
\end{eqnarray} 
where the last equality is because
\begin{equation*}
    \langle z^{k+1}-z^k,z-z^{k+1}\rangle_{\Lambda^{-1}} = \frac1{2}\|z^k-z\|^2_{\Lambda^{-1}}-\frac1{2}\|z^{k+1}-z\|^2_{\Lambda^{-1}}-\frac1{2}\nrm{ z^{k+1}-z^{k} }^2_{\Lambda^{-1}}.
\end{equation*}
Substituting  $R(z^{k+1})-R(z)\leq\langle v^{k+1}, z^{k+1}-z\rangle$ into \eqref{onehand*} proves the lemma. \qed
\end{proof}
Inspired by the above ``one-step descent'', we consider the following potential function constructed for an arbitrary reference point $z$:
\begin{equation}\label{def:Delta}
    \begin{split}
        \Delta_k(z) :=& \frac1{2}\|z^k-z\|^2_{\Lambda^{-1}}+\frac{c}{2}\|z^{k}-z^{k-1}\|^2_{\Lambda^{-1}}\\
        &+\langle F(z^k)-F(z^{k-1}), z-z^{k}\rangle_{\Xi\Theta}+(\mu-\beta)\iprod{ \nabla_y\Psi(z^k)}{ y^{k}-y },
    \end{split}
\end{equation}
where 
$c$ is a positive constant depending on $\alpha, \beta, \mu, \tau, \sigma, \rho$, which shall be specified later.

\subsection{Ergodic convergence for the affine equality constrained problem}\label{sec:equality}
In the following analysis, we start with the case $\KK=\{0\}$, whose analysis is cleaner and more insightful compared to the general case.
In this case, we specify the weight $c$ in the potential as
$$
c=C^{\mathrm{affine}}_{\alpha,\beta,\mu}(\tau, \sigma, \rho):=\alpha\tau  L_{f_{\rho}}+\max\{\abs{\mu\beta}\sqrt{\sigma\tau}\nrm{A}, \alpha \sqrt{\sigma\tau}\nrm{A}\}.
$$
Actually any $c\geq C^{\mathrm{affine}}_{\alpha,\beta,\mu}(\tau, \sigma, \rho)$ suffices.
Given any coefficient $c$, we define the matrix $P_c$ as% Given any coefficient $c$ and the algorithmic parameters, we define the matrix $P_c$ as 
\[\begin{aligned}
	P_c:=\begin{bmatrix}
		\rho I_m  & 0_{m\times n} & \frac{1-\alpha-\beta+\mu}{2}I_m \\
		0_{n\times m} & \left(\frac{1-2c}{2\tau}-\frac{(1-\alpha)L_{f_{\rho}}}{2}\right)I_n & \frac{\beta-\mu}{2}A^\T \\
		\frac{1-\alpha-\beta+\mu}{2}I_m & \frac{\beta-\mu}{2}A & \frac{1-2c}{2\sigma}I_m
	\end{bmatrix}.
\end{aligned}\]
As long as $P_c\succeq 0$, we derive the following lemma to characterize the monotonicity of the potential function.

\begin{lemma}\label{lemma:eq}
	Suppose Assumptions \ref{assumption1}-\ref{Slater's} hold and let $\KK = \{0\}$ so that \eqref{prob} is an affine equality constrained problem. 
    Define $\tilde{z} = (x^*,y)$, $\forall y\in\mathbbm{R}^m$, where $x^*$ be a primal optimal solution.
    Then as long as $0\leq\alpha\leq1$ and $P_c\succeq 0$, we have
	\begin{align}\label{onestep}
		\Delta_{k}(\tilde{z})-\Delta_{k+1}(\tilde{z})\geq\ &\Phi(x^{k+1})-\Phi(x^*)-\langle Ax^{k+1}-b,y\rangle
		+\delta^{k}, 
	\end{align}
where $\delta^k\geq0$ is some non-negative term.  
\end{lemma}
\begin{proof}
Substituting $\tilde{z}=(x^*,y)$ into Lemma \ref{lemma:one-step} and rearranging the terms yields
\begin{eqnarray}\label{onestepxstar}
	&&h(x^{k+1})-h(x^*)+\iprod{ F(z^{k+1}) }{ z^{k+1}-\tilde{z} }\nonumber\\
	& & +(\alpha-1)\iprod{ \nabla_x\Psi(z^{k+1})-\nabla_x\Psi(z^k) }{ x^{k+1}-x^* }\nonumber\\
	&\!\!\leq\ & \frac1{2}\|z^k-\tilde z\|^2_{\Lambda^{-1}}-\frac1{2}\|z^{k+1}-\tilde z\|^2_{\Lambda^{-1}}-\frac1{2}\nrm{ z^{k+1}-z^{k} }^2_{\Lambda^{-1}}\\
	&&+ \iprod{ F(z^k)-F(z^{k-1}) }{ \tilde z-z^{k} }_{\Xi\Theta} - \iprod{ F(z^{k+1})-F(z^k) }{ \tilde z-z^{k+1} }_{\Xi\Theta} \nonumber\\
	&&+ (\mu-\beta)\iprod{  \nabla_y\Psi(z^k) }{ y^{k}-y } - (\mu-\beta)\iprod{  \nabla_y\Psi(z^{k+1}) }{ y^{k+1}-y }\nonumber\\
	&&+ \iprod{  F(z^{k})-F(z^{k-1}) }{ z^k-z^{k+1} }_{\Xi\Theta}
	+(\mu-\beta)\iprod{ \nabla_y\Psi(z^k) }{ y^{k+1}-y^k }.\nonumber
\end{eqnarray}
Plugging in the definition of $\Delta_k(z)$ yields
\begin{eqnarray}
    \Delta_{k}(\tilde{z})-\Delta_{k+1}(\tilde{z})
    \geq h(x^{k+1})-h(x^*)+\Gamma_1+\Gamma_2+\Gamma_3,
\end{eqnarray}
where
\begin{align}\label{def:Gamma}
    \Gamma_1&:=\iprod{ F(z^{k+1}) }{ z^{k+1}-\tilde{z} }
    +(\alpha-1)\iprod{ \nabla_x\Psi(z^{k+1})-\nabla_x\Psi(z^k) }{ x^{k+1}-x^* },\nonumber\\
    \Gamma_2&:={\frac{c}{2}\nrm{ z^{k}\!\!-\!z^{k-1} }^2_{\!\Lambda^{-1}}\!+\!\frac{1\!-\!c}{2}\|z^{k+1}\!\!-\!z^{k}\|^2_{\!\Lambda^{-1}}}
    \!-\!\iprod{ F(z^{k})\!-\!F(z^{k-1}) }{ z^k\!\!-\!z^{k+1} }_{\!\Xi\Theta}\!,\nonumber\\
    \Gamma_3&:=-(\mu\!-\!\beta)\!\iprod{ \nabla_y\Psi(z^k) }{ y^{k+1}\!-\!y^{k} }\!.%\!=\!(\mu\!-\!\beta)\!\iprod{ Ax^{k}\!-\!b }{ y^{k+1}\!-\!y^{k} }\!.
\end{align}
We first deal with the term $\Gamma_1$.
Through a direct computation, we have
\begin{align*}
    \Gamma_1
    =& \,\alpha \iprod{ \nabla f_{\rho}(x^{k+1})-A^\top y^{k+1} }{ x^{k+1}-x^* }
    + \iprod{ Ax^{k+1}-b }{ y^{k+1}-y }\\
    &+(1-\alpha) \iprod{ \nabla f_{\rho}(x^{k})-A^\top y^{k} }{ x^{k+1}-x^* }\\
    =& \iprod{ \alpha\nabla f_{\rho}(x^{k+1}) + (1-\alpha)\nabla f_{\rho}(x^{k}) }{ x^{k+1}-x^* } - \iprod{ Ax^{k+1}-b }{ y }\\
    &+ (1-\alpha)\iprod{ A^\top(y^{k+1}-y^{k}) }{ x^{k+1}-x^* }.
\end{align*}
Due to the convexity of $f_{\rho}(x)$ and the smoothness of $f_{\rho}(x)$, it holds that
\begin{equation}\label{convex-smooth}
    \begin{split}
        &\langle \nabla f_{\rho}(x^{k+1}), x^{k+1}-x^*\rangle \geq f_{\rho}(x^{k+1})-f_{\rho}(x^*),\\
        &\langle \nabla f_{\rho}(x^k), x^{k+1}-x^*\rangle\geq f_{\rho}(x^{k+1})-f_{\rho}(x^*)-\frac{L_{f_{\rho}}}{2}\nrm{x^{k+1}-x^k}^2.
    \end{split}
\end{equation}
Note that $Ax^*=b$, we then have
\begin{align}\label{eqn:desc}
\Gamma_1
\geq &\,\,
f(x^{k+1})-f(x^*) - \iprod{ Ax^{k+1}-b }{ y } + \rho\nrm{Ax^{k+1}-b}^2\\
&-\frac{(1-\alpha)L_{f_{\rho}}}{2}\nrm{x^{k+1}-x^k}^2
+ (1-\alpha)\iprod{ y^{k+1}-y^{k} }{ Ax^{k+1}-b }.\nonumber
\end{align}
Next, we bound the term $\Gamma_2$. By the smoothness of $F$, for any $z=(x,y)$ we have
\begin{equation}\label{Fz}
    \begin{split}
        &\abs{ \iprod{  F(z^{k})-F(z^{k-1}) }{ z^k-z }_{\Xi\Theta}}\\
        \leq \ & \alpha\abs{\langle \nabla f_\rho(x^k)-\nabla f_\rho(x^{k-1}),x^k-x\rangle}
        +\alpha\abs{\langle A^\T(y^k-y^{k-1}),x^k-x \rangle}\\
        &+\abs{\mu\beta} \abs{\iprod{  A(x^k-x^{k-1}) }{ y^k-y }}\\
        \leq\ &\frac{\alpha L_{f_\rho}}2\Big(\!\nrm{x-x^k}^2\!\!+\nrm{x^k-x^{k-1}}^2\!\Big) \\
        &+ \frac{\alpha\|A\|}{2}\Big(\sqrt{\frac{\sigma}{\tau}}\|x^k-x\|^2 + \sqrt{\frac{\tau}{\sigma}}\|y^k-y^{k-1}\|^2\Big) \\
        &+\frac{\abs{\mu\beta}\!\nrm{A}}{2}\Big(\sqrt{\frac{\sigma}{\tau}}\nrm{x^k-x^{k-1}}^2 + \sqrt{\frac{\tau}{\sigma}}\nrm{y-y^{k}}^2\Big)\\
        \leq\ &\frac{c}{2}\nrm{z^k-z^{k-1}}_{\Lambda^{-1}}^2+\frac{c}{2}\nrm{z^k-z}_{\Lambda^{-1}}^2,
    \end{split}
\end{equation}
where the last inequality is due to definition of $c$.
Then we can take
\begin{equation*}
        \delta_1^k\!:=\!
        \frac{c}{2}\nrm{ z^{k+1}\!-\!z^{k} }^2_{\Lambda^{-1}}\!+\!\frac{c}{2}\|z^{k}\!-\!z^{k-1}\|^2_{\Lambda^{-1}}
        \!-\!\iprod{ F(z^{k})\!-\!F(z^{k-1}) }{ z^k\!-\!z^{k+1} }_{\Xi\Theta},
\end{equation*}
and $\delta_1^k\geq 0$ by \eqref{Fz}. Combining \eqref{eqn:desc} with $\Gamma_2=\delta_1^k+\frac{1-2c}{2}\nrm{ z^{k+1}-z^{k} }^2_{\Lambda^{-1}}$, we see
\begin{align*}
    \Delta_{k}(\tilde{z})-\Delta_{k+1}(\tilde{z})\geq\ &\Phi(x^{k+1})-\Phi(x^*)-\langle Ax^{k+1}-b,y\rangle
    +\delta_1^{k}+\delta_2^k, 
\end{align*}
where
\begin{align*}\label{eqn:def-delta-2}
    \delta_2^k:=
    &\frac{1-2c}{2}\nrm{ z^{k+1}-z^{k} }^2_{\Lambda^{-1}}+\rho\|Ax^{k+1}-b\|^2
    -\frac{(1-\alpha)L_{f_{\rho}}}{2}\nrm{x^{k+1}-x^k}^2\\
    &+(\mu-\beta)\iprod{ Ax^{k}-b }{ y^{k+1}-y^{k} }
    +(1-\alpha)\iprod{ Ax^{k+1}-b }{ y^{k+1}-y^{k} }\\
    =&\frac{1-2c}{2}\nrm{ z^{k+1}-z^{k} }^2_{\Lambda^{-1}}+\rho\|Ax^{k+1}-b\|^2
    -\frac{(1-\alpha)L_{f_{\rho}}}{2}\nrm{x^{k+1}-x^k}^2\\
    &\!-\!(\mu\!-\!\beta)\!\iprod{ A(x^{k+1}\!-\!x^k) }{ y^{k+1}\!-\!y^{k} }
    \!+\!(1\!-\!\alpha\!-\!\beta\!+\!\mu)\!\iprod{ Ax^{k+1}\!-\!b}{ y^{k+1}\!-\!y^{k} }\!.
\end{align*}
The fact $\delta_2^k\geq 0$ directly follows from $P_c\succeq 0$, because $\delta_2^k$ is a quadratic form of $Ax^{k+1}-b$, $x^{k+1}-x^{k}$ and $y^{k+1}-y^{k}$, with matrix $P_c$. We finalize the proof by setting $\delta^k:=\delta_1^k+\delta_2^k$.
\qed
\end{proof}

\begin{corollary}\label{cor:pos-delta}
    Under the settings of Lemma \ref{lemma:eq} with $\tilde{z}=(x^*,y)$, it holds that
    \begin{equation}\label{eqn:pos-delta}
		0
        \leq \Delta_k(\tilde{z})
		\leq \nrm{z^k-\tilde{z}}^2_{\Lambda^{-1}}+c\nrm{z^k-z^{k-1}}_{\Lambda^{-1}}^2.
	\end{equation}
\end{corollary}
\begin{proof}
Notice that due to $Ax^*=b$, we have
\begin{align*}
    \begin{split}
        \Delta_k(\tilde{z}) =& \frac1{2}\|z^k-z\|^2_{\Lambda^{-1}}+\frac{c}{2}\|z^{k}-z^{k-1}\|^2_{\Lambda^{-1}}+(\beta-\mu)\iprod{ A(x^k-x^*) }{ y^{k}-y }\\
        &+\langle F(z^k)-F(z^{k-1}), \tilde{z}-z^{k}\rangle_{\Xi\Theta}.
    \end{split}
\end{align*}
From \eqref{Fz}, we know
\[\begin{aligned}
    \abs{ \iprod{  F(z^{k})-F(z^{k-1}) }{ z^k-\tilde{z} }_{\Xi\Theta}}
    \leq\ \frac{c}{2}\nrm{z^k-z^{k-1}}_{\Lambda^{-1}}^2+\frac{c}{2}\nrm{z^k-\tilde{z}}_{\Lambda^{-1}}^2.
\end{aligned}\]
Combining the above two relations, we obtain
\[\begin{aligned}
    \nrm{\tilde{z}-z^k}_{Q_c}^2
    \leq \Delta_k(\tilde{z})
    \leq \nrm{\tilde{z}-z^k}_{c\Lambda^{-1}+Q_c}^2+c\nrm{z^k-z^{k-1}}_{\Lambda^{-1}}^2,
\end{aligned}\]
where we denote
\begin{equation}\label{def:Qc}
    Q_c:=\begin{bmatrix}
        \frac{1-c}{2\tau}I_n & \frac{\beta-\mu}{2}A^\T \\
        \frac{\beta-\mu}{2}A & \frac{1-c}{2\sigma}I_m
    \end{bmatrix}.
\end{equation}
Here $(1-c)\Lambda^{-1}\succeq Q_c\succeq 0$ due to $P_c\succeq 0$. This observation completes the proof.
\qed
\end{proof}
As a direct result of Lemma \ref{lemma:eq}, with initialization $x^{-1} = x^0, y^{-1} = y^0 = 0$, we have the following ergodic convergence result. 

\begin{theorem}\label{thm:affine}
	For the affine equality constrained problems, suppose that Assumptions \ref{assumption1}-\ref{Slater's} are satisfied and $0\leq\alpha\leq 1$. If the step sizes $\tau, \sigma$, the penalty factor $\rho$, and the extrapolation coefficients $\alpha,\beta,\mu$ are properly chosen so that $P_c\succeq0$, 	then for $\forall N\geq1$ and  $\forall\gamma\geq 0$, we have
	\begin{equation}\label{thm1}
		\Phi(\bar{x}_N)-\Phi(x^*)+\gamma\|A\bar{x}_N-b\|
		\leq \frac1{N}\left(\frac{\nrm{x^0-x^*}^2}{\tau}+\frac{\gamma^2}{\sigma}\right),
	\end{equation}
    where $\bar x_N = \frac{1}{N}\sum_{k=1}^{N}x^k$. Moreover, it holds that
	\begin{align*}
		\abs{ \Phi(\bar{x}_N)-\Phi(x^*) }&\leq \frac4{N}\left(\frac{\nrm{x^0-x^*}^2}{\tau}+\frac{\nrm{y^*}^2}{\sigma}\right),\\
		\nrm{A\bar{x}_N-b} &\leq \frac{3}{N}\left(\frac{\nrm{x^0-x^*}}{\sqrt{\tau\sigma}}+\frac{\nrm{y^*}}{\sigma}\right).
	\end{align*}
\end{theorem}

\begin{proof}
	By Lemma \ref{lemma:eq}, for $k=0,\cdots,N-1$ and $y$, it holds that
	\begin{equation*}
		\begin{split}
			\Delta_{k}(\tilde{z})-\Delta_{k+1}(\tilde{z})\geq\ &\Phi(x^{k+1})-\Phi(x^*)-\langle Ax^{k+1}-b,y\rangle.
		\end{split}
	\end{equation*}
	Taking the sum of the above inequality from $0$ to $N - 1$, we get
	\begin{equation}
		\begin{split}
			\sum_{k=0}^{N-1} \left(\Phi(x^{k+1})-\Phi(x^*)-\langle Ax^{k+1}-b,y\rangle\right)
			\leq \Delta_0(\tilde{z})-\Delta_N(\tilde{z}).
		\end{split}
	\end{equation}
    Let $\bar{x}_N:=\frac1N\sum_{k=1}^{N}x^k$ denote the averaged iterates.
	Then, for any $\gamma\geq 0$,
	let $\hat{y}=-\frac{\gamma (A\bar x_N-b)}{\|A\bar x_N-b\|}$ and $\hat z = (x^*,\hat y)$, we have 
	\begin{equation}
		\begin{split}
			\Phi(\bar{x}_N)-\Phi(x^*)+\gamma\|A\bar x_N-b\|
			&=\  \Phi(\bar{x}_N)-\Phi(x^*)-\iprod{A\bar x_N-b}{ \hat{y} }\\
			\leq\  \frac1N&\sum_{k=0}^{N-1} \left(\Phi(x^{k+1})-\Phi(x^*)-\langle A\bar x^{k+1}-b,\hat{y}\rangle\right)\\
			&\!\!\!\!\!\!\!\!\!\!\!\!\!\!\!\!\leq\  \frac{\Delta_0(\hat z)-\Delta_N(\hat{z})}{N} \leq \ \frac1{N}\bigg(\frac{\nrm{x^0-x^*}^2}{\tau}+\frac{\gamma^2}{\sigma}\bigg),
		\end{split}
	\end{equation} 
    where the last inequality is due to Corollary \ref{cor:pos-delta}. 
	The proof is completed by setting $\gamma=\nrm{y^*}+\sqrt{\frac{\sigma}{\tau}\nrm{x^*-x^0}^2+\nrm{y^*}^2}$ and applying Lemma \ref{lemma:sep-gap}.
	\qed
\end{proof}

Next, we substitute the parameters $\mu, \alpha, \beta$ into the step size conditions $P_c\succeq0$ and analyze the sufficient conditions on step sizes specifically.

\vspace{0.2cm}
\noindent\textbf{SOGDA}\,\, Set $\mu=1, \alpha=0, \beta=1$. For any $\rho>0$, $P_c\succeq0$ is guaranteed if
\[
    \quad 2\sqrt{\sigma\tau}\nrm{A}+\max\left(\frac{\sigma}{2\rho}, \tau L_{f_{\rho}}\right)\leq 1.
\]
Note that the above step size condition will never be satisfied if we set $\rho=0$. This indicates the potential non-convergence of SOGDA when $\rho=0$, which is also numerically observed in our experiments.\vspace{0.1cm}\\
\noindent\textbf{PDHG}\,\,  Set $\mu=0, \alpha=0, \beta=0$. For any $\rho>0$, $P_c\succeq0$ is guaranteed if
\[
    \sigma\leq 2\rho,\quad \frac1\tau\geq L_f+\rho\nrm{A}^2.
\]
When $\rho=0$, the above step size condition can be satisfied only if $\sigma = 0$, which makes the right hand side of \eqref{thm1} unbounded. This case corresponds to the potential non-convergence of PDHG pointed out by \cite{he2014convergence}.
When $\rho>0$, the algorithm can been viewed as the linearized ALM. Hence the convergence of the linearized ALM is also obtained.\vspace{0.1cm}\\
\noindent\textbf{CP}\,\,  Set $\mu=0, \alpha=0, \beta=1$. For any $\rho\geq0$, $P_c\succeq0$ is guaranteed if
\[\frac1{\tau}\geq L_f+(\rho+\sigma)\nrm{A}^2.\]

\noindent\textbf{GDA}\,\,  Set $\mu=1, \alpha=0, \beta=0$. For any $\rho>0$, $P_c\succeq0$ is guaranteed if
\[
    \sigma< \frac{\rho}{2}, \quad \frac{1}{\tau} \geq L_f+\rho\nrm{A}^2 \frac{\rho-\sigma}{\rho-2\sigma}.
\]

\noindent\textbf{OGDA}\,\,  Set $\mu=1, \alpha=1, \beta=1$. For any $\rho \geq0$, $P_c\succeq0$ is guaranteed if
\[
    \tau L_{f_{\rho}}+\sqrt{\sigma\tau}\nrm{A}\leq \frac{1}{2}.
\]

It is worth mentioning that the analysis in this section is also feasible to the case of general cone with $\rho\!=\!0$. The only difference is that we need to set $s(y)=\mathbbm{1}_{\KK^*}(y)$, which does not interrupt the proof. Notice that CP and OGDA are guaranteed to converge while $\rho\!=\!0$ according to the above conditions. Therefore, for the case of general cone with $\rho\!=\!0$, CP and OGDA converge under the same step size conditions.

\subsection{Ergodic convergence for conic inequality constrained problem}

We next consider the general conic affinely constrained problems with $\rho>0$. 
For such problems, we specify the weight $c$ in the potential function as
$$
c=C^{\mathrm{conic}}_{\alpha,\beta,\mu}(\tau, \sigma, \rho):=\max \big\{\alpha \tau L_{f_\rho}, \abs{\mu\beta}\frac{\sigma}{\rho}\big\} + \max \big\{\alpha, \abs{\mu\beta}\big\}\nrm{A}\sqrt{\sigma\tau}.
$$
Define   $\gamma_y\!=\!(\mu\!-\!\beta)^2\!+\!(1\!+\!\alpha)\!\abs{\mu\!-\!\beta} \!+\! 4(\mu\!-\!\beta)$, $\gamma_w\!=\!t\big(2\!-\!2\alpha,(1\!-\!\alpha)^2\!+\!(1\!+\!\alpha)\!\abs{\mu\!-\!\beta}\!\big)$ where the function $t(\cdot,\cdot)$ is given by 
$$t(a,b):=\begin{cases}
	b+\frac{(a-b)^2}{2a-b}, & a>b,\\
	b, & a\leq b.\\
\end{cases}$$
Then we define the matrix $P_c^\prime$ for any given $c>0$ as 
\[\begin{aligned}
	P_c^\prime=
	\begin{bmatrix}
		\left(\frac{1-2c}{2\tau}-\frac{(1-\alpha)L_{f}}{2}\right)I_n-\frac{\rho\gamma_w}{4}A^\T A & \frac{\gamma_w}{4}A^\T\\
		\frac{\gamma_w}{4}A & \left(\frac{1-2c}{2\sigma}-\frac{\gamma_w+\gamma_y}{4\rho}\right)I_m
	\end{bmatrix}.
\end{aligned}\] 
Before presenting the analysis of the general cone, we introduce a supporting lemma. 
\begin{lemma}\label{lemma:algebra}
	For $w, w'\in\RR^m$ and $a,b\geq 0$, it holds that
	\[\begin{aligned}
		2a\iprod{ \cpos{w} }{ \cneg{w'} }+b\nrm{ \cpos{w}-\cpos{w'} }^2
		\leq t(a,b)\nrm{ w-w' }^2.
	\end{aligned}\] 
	Furthermore, for $v, v'$, we consider $r=\cpos{w}+v, r'=\cpos{w'}+v', u=w+v, u'=w'+v'$, then
	\[\begin{aligned}
		&\nrm{r-r'}^2\leq \nrm{u-u'}^2+\nrm{v-v'}^2.%-2\iprod{\cpos{w}}{\cneg{w'}}-2\iprod{\cpos{w'}}{\cneg{w}}.
	\end{aligned}\]
	Especially, it holds that $\nrm{r^{k}-r^{k-1}}\leq \nrm{A}\nrm{x^{k}-x^{k-1}}+\frac{1}{\rho}\nrm{y^{k}-y^{k-1}}$.
\end{lemma}
The proof of the above lemma is deferred to the appendix. Armed with this lemma, we are ready to prove the monotonicity of the potential function.  

\begin{lemma}\label{lemma:ineq}
	Consider the case where $\KK$ is a general proper cone, $\rho>0$, and $s(\cdot) = 0$. 
    Suppose Assumptions \ref{assumption1}-\ref{Slater's} hold and $0\leq\alpha\leq1$.
	Let the reference point be $\tilde{z} = (x^*,y)$, where $x^*$ is a primal optimal solution and $y\in\dom\ s$ is arbitrary.  %Set  {$c\!=\!C^{\mathrm{conic}}_{\alpha,\beta,\mu}(\tau, \sigma, \rho)\!:=\!\max \!\big\{\!\alpha \tau L_{\!f_\rho}, \abs{\mu\beta}\!\frac{\sigma}{\rho}\big\} \!\!+\! \max\! \big\{\!\alpha, \abs{\mu\beta}\!\big\}\!\nrm{A}\!\sqrt{\sigma\tau}$}
  Then as long as $P_c^\prime\succeq0$, we have 
	\begin{equation}
		\label{onestep1}
		\Delta_{k}(\tilde{z})-\Delta_{k+1}(\tilde{z})\geq \Phi(x^{k+1})-\Phi(x^*)-\langle r^{k+1},y\rangle+\delta^k,
	\end{equation}
	where $\delta^k\geq0$ is some non-negative term.  
\end{lemma}

\begin{proof}
Similar to the derivation of \eqref{onestepxstar}, substituting $\tilde{z}=(x^*,y)$ into \eqref{onehand**} and rearranging the terms gives
\begin{equation}\label{eqn:lmcone-0}
    \Delta_{k}(\tilde{z})-\Delta_{k+1}(\tilde{z})\geq h(x^{k+1})-h(x^*)+s(y^{k+1})-s(y) +\Gamma_1+\Gamma_2+\Gamma_3,
\end{equation}
where $\Gamma_1, \Gamma_2$ and $\Gamma_3$ are defined in \eqref{def:Gamma}. For our choice of $s$ and $y$, it always holds that $s(y)=0$, $s(y^{k+1})=0$.
Now, we first bound the term $\Gamma_1$. For simplicity, we write $r^{k}\!=\!-\!\nabla_{\!y}\Psi(z^k)\!=\!\frac{y^k}{\rho}\!+\!\mathcal{P}_{\!+}(w^k)\!=\!Ax^k\!-\!b\!+\!\mathcal{P}_{\!-}(w^{k})$ and $w^k\!=\!Ax^k\!-\!b\!-\!\frac{y^k}{\rho}$.
Now, by basic algebra we have
\begin{align}
	\label{eqn:lmcone-1}
    \Gamma_1
    =&  \iprod{ \alpha\nabla_x \Psi(x^{k+1}) + (1-\alpha)\nabla_x \Psi(x^{k}) }{ x^{k+1}-x^* }
    + \iprod{ r^{k+1} }{ y^{k+1}-y }\\
    =& \underbrace{
        \iprod{ \alpha\nabla f(x^{k+1}) + (1-\alpha)\nabla f(x^{k}) }{ x^{k+1}-x^* } - \iprod{ r^{k+1} }{ y }
    }_{\Gamma_{1,a}}\nonumber\\
    &+  \underbrace{
        \rho\iprod{ \alpha\cpos{w^{k+1}} + (1-\alpha)\cpos{w^{k}} }{ Ax^{k+1}-Ax^* } + \iprod{ r^{k+1} }{ y^{k+1} }
    }_{\Gamma_{1,b}}.\nonumber
\end{align}
Similar to \eqref{convex-smooth}, we utilize the convexity and smoothness of $f(x)$ and obtain
\begin{equation}
	\label{eqn:lmcone-1a}
	\Gamma_{1,a}\geq f(x^{k+1})-f(x^*) - \iprod{ r^{k+1} }{ y } - \frac{(1-\alpha)L_{f}}{2}\nrm{x^{k+1}-x^k}^2.
\end{equation} 
Notice that $Ax^*-b\in \KK$, we also have
\begin{align}
	\label{eqn:lmcone-1b}
	& \frac{1}{\rho}\Gamma_{1,b}
	\geq 
	\iprod{ \alpha\cpos{w^{k+1}} + (1-\alpha)\cpos{w^{k}} }{ Ax^{k+1}-b } + \iprod{ r^{k+1} }{ \frac{y^{k+1}}{\rho} }\\
	& \!\!=
	\iprod{ \alpha\mathcal{P}_{\!\!+}(w^{k+1}) + (1-\alpha)\mathcal{P}_{\!\!+}(w^{k}) }{ r^{k+1} } - (1-\alpha) \iprod{ \mathcal{P}_{\!\!+}(w^{k}) }{ \mathcal{P}_{\!\!-}(w^{k+1}) \!}\nonumber\\
	&\quad\,+ \iprod{ r^{k+1} }{ r^{k+1}-\cpos{w^{k+1}}\!}\nonumber\\
	&\!\!=
	\nrm{r^{k+1}}^{\!2} \!\!\!+\! (1\!-\!\alpha)\!\iprod{ \mathcal{P}_{\!\!+}(w^{k})\!-\!\mathcal{P}_{\!\!+}(w^{k+1}) }{ r^{k+1} } \!-\! (1\!-\!\alpha)\! \iprod{ \mathcal{P}_{\!\!+}(w^{k}) }{ \mathcal{P}_{\!\!-}(w^{k+1})\!}\!, \nonumber
\end{align}
where the last second equality is because the fact that $\langle\cpos{u},\cneg{u}\rangle\equiv0$ for $\forall u$ and the definition of the vector $r^k$.  
Next, we bound the term $\Gamma_2$ similarly to \eqref{Fz}. 
It holds that
\begin{align}
    \label{eqn:lmcone-2}
    &\abs{ \iprod{  F(z^{k})-F(z^{k-1}) }{ z^k-z }_{\Xi\Theta}}\\ 
    \leq \ & \alpha\abs{\langle\nabla_x\Psi(z^k)-\nabla_x\Psi(z^{k-1}),x^k-x\rangle}+\abs{\mu\beta}\abs{\langle r^k-r^{k-1},y^k-y\rangle}\nonumber\\
    \stackrel{(i)}{\leq} \ & \alpha L_f\nrm{ x^k-x^{k-1} }\nrm{ x^k-x }
    +\alpha\rho\nrm{A}\nrm{ \cpos{w^k}-\cpos{w^{k-1}} }\nrm{ x^k-x }\nonumber\\
    &+\abs{\mu\beta} \nrm{  r^{k}-r^{k-1} }\nrm{ y^k-y }\nonumber\\
    \stackrel{(ii)}{\leq} \ &\frac{\alpha L_{f}}2\nrm{x-x^k}^2+\frac{\alpha L_{f}}2\nrm{x^k-x^{k-1}}^2\nonumber\\
    &+\alpha\rho \nrm{A}\nrm{x-x^k}\left(\nrm{A}\nrm{x^{k}-x^{k-1}}+\frac{1}{\rho}\nrm{y^k-y^{k-1}}\right)\nonumber\\
    &+\abs{\mu\beta} \nrm{y-y^k}\left(\nrm{A}\nrm{x^{k}-x^{k-1}}+\frac{1}{\rho}\nrm{y^k-y^{k-1}}\right)\nonumber\\
    \leq \ &\frac{\alpha L_{f}}2\nrm{x-x^k}^2+\frac{\alpha L_{f}}2\nrm{x^k-x^{k-1}}^2\nonumber\\
    &+\alpha\rho \nrm{A}^2\frac{\nrm{x-x^k}^2+\nrm{x^{k}-x^{k-1}}^2}{2}\nonumber\\
    &+\alpha\sqrt{\sigma\tau}\nrm{A}\left(\frac{\nrm{x-x^k}^2}{2\tau}+\frac{\nrm{y^k-y^{k-1}}^2}{2\sigma}\right)\nonumber\\
    &+\abs{\mu\beta}\sqrt{\sigma\tau}\nrm{A}\left(\frac{\nrm{x^{k}-x^{k-1}}^2}{2\tau}+\frac{\nrm{y-y^k}^2}{2\sigma}\right)\nonumber\\
    &+\frac{\abs{\mu\beta}}{\rho}\frac{\nrm{y-y^k}^2+\nrm{y^{k}-y^{k-1}}^2}{2}\nonumber\\
    \stackrel{(iii)}{\leq}\ &\frac{c}{2}\nrm{z^k-z^{k-1}}_{\Lambda^{-1}}^2+\frac{c}{2}\nrm{z-z^k}_{\Lambda^{-1}}^2,\nonumber
\end{align}
where (i) is due to the fact that $\nabla_x\Psi(z^k)=\nabla f(x)+\rho A^\T\cpos{w^k}$ and $f$ is $L_f$-smooth, 
(ii) utilizes the nonexpansiveness of the projection operator and Lemma \ref{lemma:algebra}, 
(iii) is because of the definition of $c$. Substitute $z=z^{k+1}$ into \eqref{eqn:lmcone-2} and apply the inequality, we get
\begin{equation}
    \Gamma_2\geq \frac{1-2c}{2}\nrm{z^{k+1}-z^k}^2_{\Lambda^{-1}}.
\end{equation}
Finally, we note that $\Gamma_3$ can be rewritten as 
\begin{align}
    \label{eqn:lmcone-3}
    \langle r^k, y^{k+1}&\!-\!y^k\rangle
    =\iprod{ r^k-r^{k+1} }{ y^{k+1}-y^{k} }
    +\iprod{ r^{k+1} }{ y^{k+1}-y^{k} }\\
    =&\iprod{ \mathcal{P}_{\!+}(w^{k})\!-\!\mathcal{P}_{\!+}(w^{k+1}) }{ y^{k+1}\!-\!y^{k} }\!-\!\frac{1}{\rho}\nrm{y^{k+1}\!-\!y^{k}}^2\!\!+\!\iprod{ r^{k+1} }{ y^{k+1}\!-\!y^{k} }.\nonumber
\end{align}
Combining the inequalities \eqref{eqn:lmcone-0} - \eqref{eqn:lmcone-3}, we obtain 
\begin{eqnarray}
    \label{eqn:lmcone-4}
    &&\Delta_{k}(\tilde{z})-\Delta_{k+1}(\tilde{z})\\
    &\geq& \Phi(x^{k+1})-\Phi(x^*)  -\langle r^{k+1},y\rangle + \Gamma_4\nonumber\\
    &&+\frac{1-2c}{2}\nrm{ z^{k+1}-z^{k} }^2_{\Lambda^{-1}}-\frac{(1-\alpha)L_{f}}{2}\nrm{x^{k+1}-x^k}^2-\frac{\mu-\beta}{\rho}\nrm{ y^{k+1}-y^k }^2 \nonumber,
\end{eqnarray}
where
\begin{eqnarray*}
	&& \Gamma_4= \rho\nrm{r^{k+1}}^2 - (1\!-\!\alpha) \rho \iprod{ \mathcal{P}_{\!+}(w^{k}) }{ \mathcal{P}_{\!-}(w^{k+1}) }\\
	& & \qquad\,\, -  (1\!-\!\alpha)\rho\iprod{ \mathcal{P}_{\!+}(w^{k+1})\!-\!\mathcal{P}_{\!+}(w^{k}) }{ r^{k+1} } \!+\! (\mu\!-\!\beta)\iprod{ r^{k+1} }{ y^{k+1}-y^{k} }\\
    && \qquad\,\, -(\mu-\beta)\iprod{ \mathcal{P}_{\!+}(w^{k+1})\!-\!\mathcal{P}_{\!+}(w^{k}) }{ y^{k+1}\!-\!y^{k} }.
\end{eqnarray*}
Consequently, we have 
\begin{eqnarray}
	& & \Gamma_4 + (1\!-\!\alpha) \rho \iprod{ \mathcal{P}_{\!+}(w^{k}) }{ \mathcal{P}_{\!-}(w^{k+1})}\nonumber\\
	& = & \rho\nrm{r^{k+1}}^2\!\!-\!2\rho\cdot\iprod{r^{k+1}}{ \frac{1\!-\!\alpha}{2}\left(\cpos{w^{k+1}}\!-\!\cpos{w^{k}}\right) - \frac{\mu\!-\!\beta}{2\rho}\left(y^{k+1}-y^k\right)}\nonumber\\
	& &  -(\mu-\beta)\iprod{ \cpos{w^{k+1}}-\cpos{w^{k}} }{ y^{k+1}-y^{k} } \nonumber\\
	&\geq& -\frac{\rho}{4}\nrm{(1-\alpha)\left(\cpos{w^{k+1}}-\cpos{w^{k}}\right) - \frac{\mu-\beta}{\rho}\left(y^{k+1}-y^{k}\right)
	}^2\nonumber\\
	&&-(\mu-\beta)\iprod{ \cpos{w^{k+1}}-\cpos{w^{k}} }{ y^{k+1}-y^{k} }\nonumber\\
	&= &
	-\frac{\rho(1-\alpha)^2}{4}\nrm{ \cpos{w^{k+1}}-\cpos{w^{k}} }^2
	-\frac{(\mu\!-\!\beta)^2}{4\rho}\nrm{ y^{k+1}-y^{k} }^2\nonumber\\
	&&-\frac{(1+\alpha)(\mu-\beta)}{2}\iprod{ \cpos{w^{k+1}}-\cpos{w^{k}} }{ y^{k+1}-y^{k} }\nonumber\\
	&\geq&
	-\frac{\rho\gamma_0}{4}\nrm{ \cpos{w^{k+1}}-\cpos{w^{k}} }^2-\frac{\gamma_1}{4\rho}\nrm{ y^{k+1}-y^{k} }^2,\nonumber
\end{eqnarray}
where the last inequality is due to the choice $\gamma_0=(1-\alpha)^2+(1+\alpha)\abs{\mu-\beta}$, $\gamma_1=(\mu-\beta)^2+(1+\alpha)\abs{\mu-\beta}$. Consequently, with $\gamma_w=t(2(1-\alpha),\gamma_0)$, Lemma \ref{lemma:algebra} indicates that 
\begin{eqnarray} 
	\Gamma_4 &\!\geq& \!-\frac{\gamma_1}{4\rho}\!\nrm{ y^{k+1}\!\!-\!y^{k} }^2 \!\!-\! \frac{\rho\gamma_0}{4}\!\nrm{ \mathcal{P}_{\!+}(w^{k+1})\!-\!\mathcal{P}_{\!+}(w^{k}) }^2 \!\!-\! (1\!-\!\alpha) \rho \!\iprod{ \mathcal{P}_{\!+}(w^{k}) }{ \mathcal{P}_{\!-}(w^{k+1}) \!}\nonumber\\
	& \geq & \!-\frac{\gamma_1}{4\rho}\!\nrm{ y^{k+1}\!\!-\!y^{k} }^2 \!\!-\! \frac{\rho\gamma_w}{4}\!\nrm{ w^{k+1}\!-\!w^{k} }^2.\nonumber
\end{eqnarray}
Substituting the above inequality into \eqref{eqn:lmcone-4} yields 
$$\Delta_{k}(\tilde{z})\!-\!\Delta_{k+1}(\tilde{z})\geq \Phi(x^{k+1}) \!-\! \Phi(x^*)  \!-\! \langle r^{k+1},y\rangle \!+\! \delta^k,$$
where we define $\gamma_y=\gamma_1+4(\mu-\beta)$ and
\begin{align*}
    \delta^k:=&\frac{1-2c}{2}\nrm{ z^{k+1}-z^{k} }^2_{\Lambda^{-1}}-\frac{(1-\alpha)L_{f}}{2}\nrm{x^{k+1}-x^k}^2 \\
    & -\frac{\rho \gamma_w}{4}\nrm{ w^{k+1}-w^{k} }^2
    -\frac{\gamma_y}{4\rho}\nrm{ y^{k+1}-y^{k} }^2,\\
    =& \left(\frac{1-2c}{2\tau}-\frac{(1-\alpha)L_{f}}{2}-\frac{\rho\gamma_w}{4}A^\T A \right)\|x^{k+1}-x^k\|^2 \\
    & +\frac{\gamma_w}{2}\langle A(x^{k+1}-x^k),y^{k+1}-y^k\rangle
    +\left(\frac{1-2c}{2\sigma}-\frac{\gamma_w+\gamma_y}{4\rho}\right)\|y^{k+1}-y^k\|^2.
\end{align*}
Hence $\delta^k$ is a quadratic form on $x^{k+1}-x^{k}$ and $y^{k+1}-y^{k}$ associated with the matrix $P_c^\prime$. Since $P_c^\prime\succeq 0$, the quadratic form is positive semi-definite, and $\delta^k\geq0$ always holds. This completes the proof. 
\qed
\end{proof}
According to Remark \ref{rmk:explicit}, $s(\cdot)$ can also be set to $\mathbbm{1}_{\KK^*}(\cdot)$ when $\rho>0$. 
Since the above proof is feasible for any convex and lower semi-continuous $s(\cdot)$ and any $y\in\dom\ s$, Lemma \ref{lemma:ineq} also holds for the case where $s(\cdot)=\mathbbm{1}_{\KK^*}(\cdot)$.

\begin{lemma}\label{lemma:deltalow}
    Under the settings of Lemma \ref{lemma:ineq}, we fix a $y^*\in\KK^*$ such that $z^*=(x^*,y^*)$ is the KKT pair. Then it holds that
    $$
    \Delta_k(z^*)\geq \nrm{z^k-z^*}_{\frac{c}{2}\Lambda^{-1}+P_c'}^2.
    $$
    Furthermore, assuming that $P_c'\succeq \diag\left(0,\frac{a}{2\sigma}\right)$ for some $a\geq0$ such that $a+c>0$, we then have
    $$
    \Delta_k(\tilde{z})\geq -\frac{\abs{\mu-\beta}^2\sigma}{2(a+c)\rho^2}\nrm{y-y^*}^2.
    $$
\end{lemma}

\begin{proof}
According to \eqref{eqn:lmcone-2}, we have
\begin{align*}
    \Delta_k(z) =& \frac1{2}\|z^k-z\|^2_{\Lambda^{-1}}+\frac{c}{2}\|z^{k}-z^{k-1}\|^2_{\Lambda^{-1}}\\
        &+\langle F(z^k)-F(z^{k-1}), z-z^{k}\rangle_{\Xi\Theta}+(\mu-\beta)\iprod{  \nabla_y\Psi(z^k) }{ y^{k}-y }\\
        \geq& \frac{1-c}{2}\|z^k-z\|^2_{\Lambda^{-1}}+(\beta-\mu)\iprod{ r^k }{ y^{k}-y }.
\end{align*}
It remains to bound the last term on the right hand side. 
Define $\tilde{r}$ and $\tilde{w}$ as the value of $r$ and $w$ at $\tilde{z}$, that is, $\tilde{r} := -\nabla_y\Psi(\tilde{z})$, $\tilde{w}:=Ax^*-b-\frac{y}{\rho}$, then we have $\tilde{r} = \frac{y}{\rho}+\cpos{\tilde{w}}$.
Utilizing the definition of $r^k$ and $\tilde{r}$ yields
\begin{align*}
    \iprod{ r^k }{ y^{k}-y }
    =&
    \iprod{ r^k-\tilde{r} }{ y^{k}-y } + \iprod{ \tilde{r} }{ y^{k}-y }\\
    =&
    \iprod{ \cpos{w^k}-\cpos{\tilde{w}} }{ y^{k}-y } + \frac{1}{\rho}\nrm{y^k-y}^2 + \iprod{ \tilde{r} }{ y^{k}-y }.
\end{align*}
The optimality condition indicates that $\nabla_y\Psi(x^*,y^*)=0$, hence we have
$\|\tilde{r}\|=\nrm{\nabla_y\Psi(x^*,y)-\nabla_y\Psi(x^*,y^*)}\leq \frac1\rho\nrm{y-y^*}$ by Lemma \ref{lemma:algebra}. Then it holds
\begin{equation}\label{eqn:conpos-1}
    \begin{split}
        \Delta_k(\tilde{z})
        \geq& \frac{1-c}{2}\|z^k-\tilde{z}\|^2_{\Lambda^{-1}}-\frac{\mu-\beta}{\rho}\nrm{y^k-y}^2\\
        &-\abs{\mu-\beta}\nrm{w^k-\tilde{w}}\nrm{y^k-y}-\frac{\abs{\mu-\beta}}{\rho}\nrm{y-y^*}\nrm{y^k-y}.
    \end{split}
\end{equation}
By the definition of $P_c^\prime$, we have
\begin{equation}\label{eqn:conpos-2}
    \begin{split}
        \nrm{z^k-z}_{P_c^\prime}^2=&\frac{1-2c}{2}\nrm{ z^{k}-\tilde{z} }^2_{\Lambda^{-1}}-\frac{(1-\alpha)L_{f}}{2}\nrm{x^{k}-x^*}^2 \\
        & -\frac{\rho \gamma_w}{4}\nrm{ w^{k}-\tilde{w} }^2
        -\frac{\gamma_y}{4\rho}\nrm{ y^{k}-y}^2.
    \end{split}
\end{equation}
Combining \eqref{eqn:conpos-1} and \eqref{eqn:conpos-2} yields
\begin{equation}\label{eqn:conpos-3}
    \begin{split}
        \Delta_k(\tilde{z})\geq&\|z^k-\tilde{z}\|^2_{\frac{c}{2}\Lambda^{-1}+P_c'}+\frac{\rho\gamma_w}{4}\nrm{w^k-\tilde{w}}^2+\frac{\gamma_y-4(\mu-\beta)}{4\rho}\nrm{y^k-y}^2\\
        &-\abs{\mu-\beta}\nrm{w^k-\tilde{w}}\nrm{y^k-y}-\frac{\abs{\mu-\beta}}{\rho}\nrm{y-y^*}\nrm{y^k-y}.
    \end{split}
\end{equation}
By definition, we have $\gamma_w\geq (1-\alpha)^2+(1+\alpha)\abs{\mu-\beta}$, $\gamma_y-4(\mu-\beta)=(\mu-\beta)^2+(1+\alpha)\abs{\mu-\beta}$, and hence by basic algebra, $\gamma_w\left(\gamma_y-4(\mu-\beta)\right)\geq 4\abs{\mu-\beta}^2$. 
We further set $y=y^*$ in \eqref{eqn:conpos-3}, then the case of $z=z^*$ is proven. Under the assumption that $P_c'\succeq \diag\left(0,\frac{a}{2\sigma}\right)$, it also holds that
\begin{align*}
    \Delta_k(\tilde{z})
    \geq&\|z^k-\tilde{z}\|^2_{\frac{c}{2}\Lambda^{-1}+P_c'}-\frac{\abs{\mu-\beta}}{\rho}\nrm{y-y^*}\nrm{y^k-y}\\
    \geq&\frac{a+c}{2\sigma}\nrm{y^k-y}^2-\frac{\abs{\mu-\beta}}{\rho}\nrm{y-y^*}\nrm{y^k-y}\\
    \geq&-\frac{\abs{\mu-\beta}^2\sigma}{2(a+c)\rho^2}\nrm{y-y^*}^2,
\end{align*}
which completes the proof.
\qed
\end{proof}
Recall that the initial conditions are given by $x^{-1}=x^0$, $y^0=y^{-1}=0$. 
We give the main theorem of the convergence analysis.
\begin{theorem}\label{thm:conic}
    For the conic inequality constrained problems, suppose that Assumptions \ref{assumption1}-\ref{Slater's} are satisfied and $0\leq\alpha\leq1$. If the step size $\tau, \sigma$ and the penalty factor $\rho>0$ and the extrapolation coefficients $\alpha,\beta,\mu$ are properly chosen such that $P_c'\succeq 0$ and $P_c'+c\Lambda^{-1}\succ 0$.
    Then for $\bar x_N = \frac{1}{N}\sum_{k=1}^{N}x^k$, it holds that
    \begin{align*}
        \abs{ \Phi(\bar{x}_N)-\Phi(x^*) }\leq \bigO{\frac{1}{N}},\qquad
        \nrm{\cpos{ A\bar{x}_N-b }} \leq \bigO{\frac{1}{N}},
    \end{align*}
    where $\bigO{\cdot}$ hides constants that depend on $\nrm{x^*-x^0}, \nrm{y^*}$ and the parameters.
\end{theorem}

\begin{proof}
By Lemma \ref{lemma:ineq}, for $k=0,\cdots,N-1$ and $y$, it holds that
\begin{equation*}
    \begin{split}
        \Delta_{k}(\tilde{z})-\Delta_{k+1}(\tilde{z})\geq\ &\Phi(x^{k+1})-\Phi(x^*)-\langle r^{k+1},y\rangle.
    \end{split}
\end{equation*}
Taking the sum of the above inequality from $0$ to $N - 1$, we get
\begin{equation}
    \begin{split}
        \sum_{k=0}^{N-1} \left(\Phi(x^{k+1})-\Phi(x^*)-\langle r^{k+1},y\rangle\right)
        \leq \Delta_0(\tilde{z})-\Delta_N(\tilde{z}).
    \end{split}
\end{equation}
Therefore, for $\gamma\geq 0$, %we take $\hat{y} := \gamma (A\bar{x}_N-b)/\|A\bar{x}_N-b\|$ with $\bar{x}_N$ denotes the averaged iterates $\frac1N\sum_{k=1}^{N}x^k$.
let $\bar{r}_N = \frac1N\sum_{k=0}^{N-1}r^k$, $\hat{y}=-\gamma\cpos{\bar{r}_N}/\|\cpos{\bar{r}_N}\|$ and $\hat z = (x^*,\hat y)$, we have
\begin{align}\label{thm2}
    &\Phi(\bar{x}_N)-\Phi(x^*)+\gamma\|\cpos{\bar{r}_N}\|\\
    =\ & \Phi(\bar{x}_N)-\Phi(x^*)-\iprod{\bar{r}_N}{ \hat{y} }\nonumber\\
    \stackrel{(i)}{\leq}\ & \frac1N\sum_{k=0}^{N-1} \left(\Phi(x^{k+1})-\Phi(x^*)-\langle r^{k+1},\hat{y}\rangle\right)\nonumber\\
    \leq\ & \frac{1}{N}\left(\Delta_0(\hat{z})-\Delta_N(\hat{z})\right)\nonumber\\
    \stackrel{(ii)}{\leq}\ & \frac1N\left(\frac{\nrm{x^0-x^*}^2}{2\tau}+\frac{\gamma^2}{2\sigma}+\gamma\abs{\mu-\beta}\nrm{r^0}+\frac{\abs{\mu-\beta}^2\sigma}{(a+c)\rho^2}\left(\gamma^2+\nrm{y^*}^2\right)\right),\nonumber
\end{align}
where (i) utilizes the convexity of $\Phi$ and $r$ and (ii) is due to $P_c'+c\Lambda^{-1}\succ 0$ and Lemma \ref{lemma:deltalow}.
Since $\bar{r}_N = A\bar{x}_N-b+\frac1N\sum_{i=0}^{N-1}\cneg{w^{k}}$ and $\frac1N\sum_{i=0}^{N-1}\cneg{w^{k}}\in-\KK$, using Lemma \ref{normlemma} in appendix, we have
\begin{equation}\label{thm3}
    \|\cpos{\bar{r}_N}\|\geq\|\cpos{A\bar{x}_N-b}\|.
\end{equation}
Combining \eqref{thm3} with \eqref{thm2} yields
\begin{equation*}
    \Phi(\bar{x}_N)-\Phi(x^*)+\gamma\|\cpos{ A\bar{x}_N-b }\|
    \leq \bigO{\frac{1}{N}},
\end{equation*}
where $\bigO{\cdot}$ hides constants that depend on $\nrm{x^*-x^0}, \nrm{y^*}$ and the parameters.
The proof is completed by setting some $\gamma>\nrm{y^*}$ and directly applying Lemma \ref{lemma:sep-gap}.
\qed
\end{proof}

Next, we substitute the parameters into the step size conditions $P_c^\prime\succeq 0$ and analyze the sufficient conditions on step sizes specifically.

\vspace{0.2cm}
\noindent\textbf{SOGDA}\,\, Set $\mu=1, \alpha=0, \beta=1$. For any $\rho>0$, $P_c^\prime\succeq0$ can be guaranteed if
\[
    \sqrt{\sigma\tau}\nrm{A}+\frac{\sigma}{\rho}\leq \frac{3}{8}, \quad
    \frac1\tau\geq 4L_f+\rho\nrm{A}^2\frac{\rho}{\frac{3}{8}\rho-\sigma}.
\]

\noindent\textbf{PDHG}\,\,  Set $\mu=0, \alpha=0, \beta=0$. For any $\rho>0$, $P_c^\prime\succeq0$ can be guaranteed if
\[
    \sigma< \frac{3}{2}\rho,\quad \frac1\tau\geq L_f+\rho\nrm{A}^2\frac{\rho}{\frac{3}{2}\rho-\sigma}.
\]

\noindent\textbf{CP}\,\,  Set $\mu=0, \alpha=0, \beta=1$. For any $\rho>0$, $P_c^\prime\succeq0$ can be guaranteed if
\[\frac1{\tau}\geq L_f+(\rho+\sigma)\nrm{A}^2.\]

\noindent\textbf{GDA}\,\,  Set $\mu=1, \alpha=0, \beta=0$. For any $\rho>0$, $P_c^\prime\succeq0$ can be guaranteed if
\[
    \sigma< \frac{\rho}{4},\quad \frac1\tau\geq L_f+\rho\nrm{A}^2\frac{\rho-3\sigma}{\rho-4\sigma}.
\]

\noindent\textbf{OGDA}\,\,  Set $\mu=1, \alpha=1, \beta=1$. For any $\rho>0$, $P_c^\prime\succeq0$ can be guaranteed if
\[
    \maxop{ \tau L_{f_{\rho}} , \frac{\sigma}{\rho} }+\sqrt{\sigma\tau}\nrm{A}\leq \frac{1}{2}.
\]

Combining the analysis in Section \ref{sec:equality}, we obtain the complete ergodic convergence of the unified primal-dual framework.
From the results, we can observe that the properly selected penalty enables the convergence of several algorithms.
For example, the PDHG method based on Lagrangian function has no convergence guarantee generally.
However, if we add a penalty term to the Lagrangian function and apply PDHG on the augmented Lagrangian function, the algorithm is guaranteed to converge without further assumptions.
The penalty term makes the convex objective function into a strongly convex function along at least one direction, which brings the benefits of convergence.
The similar benefit from the penalty term is also reflected in GDA and SOGDA.
They converge if and only if the penalty parameter $\rho>0$.

\section{The non-ergodic convergence}\label{sect:smooth}
In the previous section, we have established the sublinear ergodic convergence rate of our primal-dual algorithm framework \eqref{generalalgo}. In this section, we establish the non-ergodic convergence for the last iterate of our algorithm. For simplicity, we only consider the affine equality constrained problem where $\KK=\{0\}$. Pointwise convergence to a KKT pair of problem \eqref{prob} can be directly obtained from our analysis framework. Furthermore, under the error bound condition (e.g. strong convexity), if we properly choose the algorithmic parameters, then linear convergence to a KKT pair can be achieved. %For the sake of simplicity, we only consider the linear constraints. 
Next, we define the local error bound (LEB) condition of problem \eqref{prob}.
\begin{definition}\label{def:errorbound}
For the affine equality constrained problem \eqref{prob} with $\KK=\{0\}$, denote $z = [x;y]\in\RR^{n+m}$, we define a set-valued operator $T:\RR^{n+m}\rightrightarrows \RR^{n+m}$ as
\[\begin{aligned}
T: z=[x; y]^\top \mapsto [\partial \Phi(x)-A^\top y; Ax-b]^\top.
\end{aligned}\]
Let $\cZ^*$ be the set of all KKT pairs of problem \eqref{prob}, we say $T$ satisfies (LEB) if for every $z^*\in\cZ^*$, there exists $\epsilon>0, M>0$ such that
\[\begin{aligned}
\dist{z, \cZ^*}\leq M \dist{T(z), 0}, \qquad \forall z\ \st \dist{z, z^*}\leq \epsilon.
\end{aligned}\]
\end{definition}
\noindent This type of error bound condition is satisfied for a large class of constrained optimization problems. We state a few important examples as follows. 

\begin{example}
	Consider the affinely constrained strongly convex  problem: %the smooth problem $\min_{x} f(x)$ with constraint $Ax=b$, 
	\begin{align}\label{scs}
		\min_{x} f(x), \quad\st \; Ax=b.
	\end{align}
	If $f$ is $L_f$-smooth and $\mu_f$-strongly convex, then (LEB) is satisfied for some $M>0$. %Furthermore, we provide a more careful analysis in Theorem \ref{thm:afsc}, which asserts that if the step sizes are chosen suitably, then a convergence rate of $\bigO{(\kappa_{f}+\kappa_A^2)\log\frac{1}{\epsilon}}$ can be achieved, where $\kappa_f\!=\!\frac{L_f}{\mu_f}$, $\kappa_A\!=\!\frac{\nrm{A}}{\lambda_0}$. 
	Moreover, this result can also be extended to the case where $f$ is $\mu_f$-strongly convex restricted to the hyperplane $\{x\!:Ax\!=\!b\!\}$.
\end{example}

\begin{example}
	Consider the two-block affinely constrained convex  problem:
	$$\min_{x_1,x_2} f(x_1)\!+\!h(x_2), \st A_1x_1\!+\!A_2x_2=b.$$ The operator $T$ introduced by Definition \ref{def:errorbound} agrees with the operator $T_{KKT}$ introduced in \cite{yuan2020discerning}. By \cite[Theorem 61]{yuan2020discerning}, (LEB) holds if the following assumptions are satisfied
	\begin{itemize}
		\item $A_1$ has full row rank, $A_2$ has full column rank.
		\item $f(x_1)=g(Lx_1)+\iprod{q}{x_1}$ with $g$ being smooth and strongly convex.
		\item $h$ is either a convex piecewise linear-quadratic function, or a \textit{$\ell_{1, q}$-norm regularizer with $q \in[1,2]$}, or a \textit{sparse-group LASSO regularizer}.
	\end{itemize}
    In particular, the BP problem and the L1L1 problem considered in the numerical experiments are both included in this class.
\end{example}
More examples can be found in \cite{ye1997necessary}. 
Next, let us establish the last-iterate non-ergodic convergence of our method. 
\begin{theorem}\label{thm:ogda-limit}
	For problem \eqref{prob} with $\KK=\{0\}$, suppose that Assumptions \ref{assumption1}-\ref{Slater's} hold and the algorithmic parameters are chosen so that $P_c\succ0$. For the iterate sequence $\{z_k\}_{k=0}^{+\infty}$ generated by our algorithm, then there exists $z^*\in\cZ^*$ such that $\lim_{k\to\infty}z^{k} = z^*$.
\end{theorem}
\begin{proof}
    Fix a $z^*=(x^*,y^*)\in\cZ^*$, then by Lemma \ref{lemma:eq}, it holds that
    \[\begin{aligned}
        \Delta_{k}(z^*)-\Delta_{k+1}(z^*)\geq\ &\Phi(x^{k+1})-\Phi(x^*)-\langle Ax^{k+1}-b,y^*\rangle+\delta^{k}\geq \delta^k
    \end{aligned}\]
    due to the optimality condition $Ax^* = b$ and $A^\T y^*\in\partial \Phi(x^*)$. Note that by Corollary \ref{cor:pos-delta} and the analysis of Lemma \ref{lemma:eq}, we have $\Delta_k(z^*)\geq \|z^k-z^*\|_{Q_c}^2$ and  $\delta_k\geq\delta_2^k =  \|u\|_{P_c}^2$, $u = [Ax^{k+1}-b;z^{k+1}-z^k]$, where $Q_c$ is defined in \eqref{def:Qc}. Therefore, because $P_c\succ 0$, we know $Q_c\succ0$ and  there exists $\lambda>0$ such that $\delta^{k}_2\geq \lambda\nrm{z^{k+1}-z^k}_{\Lambda^{-1}}^2$ and $\Delta_k(z^*)\geq \lambda\nrm{z^*-z^k}_{\Lambda^{-1}}^2$, for all $k\geq 0$.
    Thus the sequence $\{\Delta_k(z^*)\}_{k=0}^{+\infty}$ is monotonically decreasing, and hence the sequence $\{z^k\}_{k=0}^{+\infty}$ is bounded and has a convergent subsequence $\{z^{k_n}\}_{n=0}^{+\infty}$. Suppose this subsequence  converges to some point $\bar z=[\bar x;\bar y]$. Note that
    $$\lambda\sum_{k=0}^{+\infty} \|z^{k+1}-z^k\|^2_{\Lambda^{-1}}\leq \sum_{k=0}^{+\infty} \delta^k \leq \Delta_0(z^*)<+\infty.$$
     Consequently, we have $\lim_{n\to\infty} \nrm{z^{k+1}-z^{k}}=0$. Then taking $k=k_n$ in \eqref{generalalgo} and let $n\to\infty$ yields
    \begin{align*}
        \bar x = \prox_{\tau h}\left[\bar x-\tau \nabla_x \Psi(\bar x, \bar y)\right],\qquad
        \bar y = \prox_{\sigma s}\left[\bar y + \eta \nabla_y \Psi(\bar x, \bar y)\right].
    \end{align*}
    That is, $\bar z\in\cZ^*$.  Let us set $z^* = \bar z$, then $\lim_{n\to\infty} \|z^{k_n}-z^*\|^2_{\Lambda^{-1}} = 0$. By Corollary  \ref{cor:pos-delta}, 
    $$\Delta_{k_n}(z^*)\leq \|z^{k_n}-z^*\|^2_{\Lambda^{-1}} + c\|z^{k_n}-z^{k_n-1}\|_{\Lambda^{-1}}\to0\quad\mbox{as}\quad n\to\infty.$$ 
    Finally, due to the fact that the sequence $\{\Delta_k(z^*)\}_{k=0}^{+\infty}$ is monotonically decreasing, a subsequence converging to zero also indicates the convergence of the whole sequence to zero.   Finally, we conclude that $\|z^k\!-\!z^*\|_{\Lambda^{-1}}^2\!\leq\! \lambda^{-1}\Delta_k(z^*)\!\to\!0$ as $k\!\to\!+\infty$.
\qed
\end{proof}

Next, we provide the analysis of the linear convergence for the scheme \eqref{generalalgo}. 
\begin{lemma}\label{cor-desc}
Under the setting of Theorem  \ref{thm:ogda-limit}, for $\forall k\geq0$, there exists $C>0$ such that
\begin{align*}
    &\rho\nrm{Ax^{k+1}\!-\!b}^2\!+\!\tau\dist{\partial \Phi(x^{k+1})\!-\!A^\T y^{k+1},0}^2\!+\!\nrm{z^{k+1}\!-\!z^{k}}_{\Lambda^{-1}}^2\!\leq\! C\delta^k.
\end{align*} 
\end{lemma}
\begin{proof}
    Because $P_c\succ0$, there exists $\lambda_1,\lambda_2>0$ such that $P_c\succeq \diag\left(\lambda_1 \rho, \lambda_2\Lambda^{-1}\right)$. Then from the definition of $\delta_2^k$ in Lemma \ref{lemma:eq}, we have
    \begin{align*}
        \delta_2^k\geq \lambda_1\rho\nrm{Ax^{k+1}-b}^2+\lambda_2\nrm{z^{k+1}-z^{k}}_{\Lambda^{-1}}^2.
    \end{align*}
    It remains to bound $\dist{\partial \Phi(x^{k+1})-A^\T y^{k+1},0}^2$ by $\delta^k$.  From \eqref{Fz}, we have
    \begin{align*}
        \delta_1^k=
        &\frac{c}{2}\|z^{k+1}\!-\!z^{k}\|^2_{\Lambda^{-1}}\!+\!\frac{c}{2}\|z^{k}\!-\!z^{k-1}\|^2_{\Lambda^{-1}}
        \!-\!\iprod{ F(z^{k})\!-\!F(z^{k-1}) }{ z^k\!-\!z^{k+1} }_{\Xi\Theta}\\
        \geq &
        \frac{c}{2}\|z^{k+1}\!-\!z^{k}\|^2_{\Lambda^{-1}}\!+\!\frac{c}{2}\|z^{k}\!-\!z^{k-1}\|^2_{\Lambda^{-1}}
        \!-\!\alpha L_{f_\rho}\nrm{x^k\!-\!x^{k-1}}\nrm{x^{k+1}\!-\!x^k}\\
        &-\alpha \nrm{A}\nrm{y^k\!-\!y^{k-1}}\nrm{x^{k+1}\!-\!x^k}
        \!-\!\abs{\mu\beta} \nrm{A}\nrm{x^k\!-\!x^{k-1}}\nrm{y^{k+1}\!-\!y^k}\\
        \geq &
        \frac{c}{2}\|z^{k+1}\!-\!z^{k}\|^2_{\Lambda^{-1}}\!+\!\frac{c}{2}\|z^{k}\!-\!z^{k-1}\|^2_{\Lambda^{-1}}
        \!-\!\frac{c}{\tau}\nrm{x^k\!-\!x^{k-1}}\nrm{x^{k+1}\!-\!x^k}\\
        &-\frac{c}{\sqrt{\tau\sigma}}\nrm{y^k\!-\!y^{k-1}}\nrm{x^{k+1}\!-\!x^k}
        \!-\!\frac{c}{\sqrt{\tau\sigma}}\nrm{x^k\!-\!x^{k-1}}\nrm{y^{k+1}\!-\!y^k}\\
        \geq &
        \frac{c}{4\tau}\nrm{x^{k}\!-\!x^{k-1}}^2\!+\!\frac{c}{4\sigma}\nrm{y^{k}\!-\!y^{k-1}}^2
        \!-\!\frac{5c}{2\tau}\nrm{x^{k+1}\!-\!x^{k}}^2\!-\!\frac{3c}{2\sigma}\nrm{y^{k+1}\!-\!y^{k}}^2.
    \end{align*}
    Hence, it holds that $c\nrm{z^{k}-z^{k-1}}_{\Lambda^{-1}}^2\leq 4\delta_1^k+10c\nrm{z^{k+1}-z^{k}}_{\Lambda^{-1}}^2$.

    \noindent Recall the update rule of primal variable in \eqref{generalalgo}, we have
    \begin{equation}\label{xupdate}
        x^{k+1} = \prox_{\tau h}\left[x^k-\tau \left((1+\alpha) g_x^k-\alpha g_x^{k-1}\right)\right].
    \end{equation}
    According to optimality condition of \eqref{xupdate}, there exists $v_x^{k+1}\in\partial h(x^{k+1})$ such that 
    \[
        v_x^{k+1}+\frac1\tau(x^{k+1}-x^k)+(1+\alpha)g_x^k-\alpha g_x^{k-1}=0.
    \]
    Rearranging the above equation yields
    \begin{align*}
        &\|v_x^{k+1}+g_x^{k+1}\|\\ 
        =\ &\left\|-\frac1\tau(x^{k+1}-x^k)+g_x^{k+1}-g_x^k-\alpha(g_x^k-g_x^{k-1})\right\|\\
        \leq\ &\! \frac1{\tau}\nrm{x^{k+1}-x^k}+\|g_x^{k+1}-g_x^k\|+\alpha\|g_x^k-g_x^{k-1}\|\\
        \leq\ &\! (\frac1\tau\!+\!L_{f_\rho})\!\nrm{x^{k+\!1}\!\!-\!x^k}\!+\!\nrm{A}\!\nrm{y^{k+\!1}\!\!-\!y^k} \!+\!\alpha L_{\!f_{\!\rho}}\!\nrm{x^{k}\!\!-\!x^{k-\!1}}\!+\!\alpha\! \nrm{A}\!\nrm{y^{k}\!\!-\!y^{k-\!1}}\\
        \leq\ &\! (c\!+\!1)\!\left(
            \frac{\nrm{x^{k+1}\!-\!x^k}}{\tau}\!+\!\frac{\nrm{y^{k+1}\!-\!y^k}}{\sqrt{\sigma\tau}}
        \right) \!+\!c\left(
            \frac{\nrm{x^{k}\!-\!x^{k-1}}}{\tau}\!+\!\frac{\nrm{y^{k}\!-\!y^{k-1}}}{\sqrt{\sigma\tau}}
        \right)
    \end{align*}
    where the last inequality is due to the definition of $c$.
    Therefore, we have
    \begin{align*}\label{gradboundx}
        &\tau \dist{\partial h(x^{k+1})+\nabla f(x^{k+1})-A^\T y^{k+1},0}^2\\
        \stackrel{(i)}{\leq}\ & \tau\|v_x^{k+1}+\nabla f(x^{k+1})-A^\T y^{k+1}\|^2\\
        \stackrel{(ii)}{\leq}\, & 4(c+1)^2\nrm{z^{k+1}-z^{k}}_{\Lambda^{-1}}^2 + 4c^2\nrm{z^{k}-z^{k-1}}_{\Lambda^{-1}}^2\\
        \leq\ & 4c\delta_1^k+(10c^2+4(c+1)^2)\nrm{z^{k+1}-z^{k}}_{\Lambda^{-1}}^2
    \end{align*}
    where (i) is because $v_x^{k+1}\in\partial h(x^{k+1})$, (ii) is due to Cauchy inequality.
    Then it suffices to take $C=\max\left\{\frac{1}{\lambda_1}, 4c, \frac{10c^2+4(c+1)^2+1}{\lambda_2} \right\}$.
\qed
\end{proof}

Next, we prove the linear convergence of our algorithm.
\begin{theorem}
    Suppose Assumptions \ref{assumption1}-\ref{Slater's} are satisfied and assume (LEB) condition holds. If the parameters are chosen so that $P_c\succ0$, then there $\exists\kappa, R>0$ and an integer $K$, s.t.  for all $k\geq K$, it holds that
    \[\begin{aligned}
    \dist{x^k, \cX^*}\leq Re^{-\kappa(k-K)},\quad
    \dist{\partial \Phi(x^{k})-A^\T y^{k},0}\leq Re^{-\kappa(k-K)}.
    \end{aligned}\]
\end{theorem}
\begin{proof}
    Assume that the sequence $\{z^k\}_{k=0}^{+\infty}$ converges to $z^*\in\cZ^*$. Then there exists constant $M_1$, $\epsilon_1$ (that depends on the step sizes and $z^*$) so that 
    \begin{align}\label{thm:leb}
        \distw{\Lambda^{-1}}{z^{k+1}, \cZ^*}^2 \leq M_1\left(
            \rho\nrm{Ax^{k+1}-b}^2+\tau\dist{\partial \Phi(x^{k+1})-A^\T y^{k+1},0}^2
        \right),
    \end{align}
    as long as $\nrm{z^{k+1}-z^*}\leq \epsilon_1$. By Theorem \ref{thm:ogda-limit}, there exists $K$ such that for all $k\geq K$, it holds that $\nrm{z^{k}-z^*}\leq \epsilon_1$. Therefore, there exists $M_2$ such that for all $k\geq K$, $\tilde{z}\in\cZ^*$, it holds that
    \[\begin{aligned}
        \Delta_{k}(\tilde{z})-\Delta_{k+1}(\tilde{z})
        \geq \delta^{k}
        \overset{(i)}{\geq}& M_2\left(\distw{\Lambda^{-1}}{z^{k+1}, \cZ^*}^2+c\nrm{z^{k+1}-z^{k}}_{\Lambda^{-1}}^2\right)\\
        \overset{(ii)}{\geq}& M_2 \inf_{z\in\cZ^*} \Delta_{k+1}(z),
    \end{aligned}\]
    where $M_2=\frac{1}{C\max\{M_1,c\}}$, (i) is the combination of Lemma \ref{cor-desc} and \eqref{thm:leb}, and (ii) follows from Corollary \ref{cor:pos-delta}. Thus, we denote $\Delta^*_k:=\inf_{z\in\cZ^*} \Delta_{k}(z)$, and then $\Delta^*_k\geq (1+M_2)\Delta^*_{k+1}\forall k\geq K$. 
    Hence, it holds that for all $k\geq K$, $\Delta_k^*\leq (1+M_2)^{-(k-K)}\Delta_{K}^*$.
    The desired conclusion follows from Lemma \ref{cor-desc} and the fact that $\Delta_k^*\geq \delta^k$.
\qed
\end{proof}

For the problem class \eqref{scs}, we provide a more careful analysis. In this case, the optimal solution $x^*$ of \eqref{prob} is unique, and
\begin{equation}\label{def:Y*}
    \cZ^*=\{x^*\}\times \cY^*, \qquad
    \cY^*=\left\{y: A^\top y=\nabla f(x^*)\right\}.
\end{equation}
\begin{theorem}\label{thm:afsc}
    Suppose the problem has the form \eqref{scs} and Assumption \ref{assumption3} holds. %$h=0$, $f$ is $L_f$-smooth and $\mu_f$-strongly convex, 
    Let the step sizes of \eqref{generalalgo} be suitably chosen as $\tau=\frac{c_{\tau}}{L_f}, \sigma=c_{\sigma}\frac{L_f}{\nrm{A}^2}, \rho=c_{\rho}\frac{L_f}{\nrm{A}^2}$, where the constants $c_{\tau}, c_{\sigma}, c_{\rho}$ only depend on $\alpha, \beta, \mu$. Then there exists constant $c_{\kappa}$ such that for all $k\geq 0$,
    \[\begin{aligned}
    \nrm{x^k-x^*}\leq \bigO{ \exp(-c_{\kappa}\left(\kappa_f+\kappa_A^2\right) k) },
    \end{aligned}\]
    where $\bigO{\cdot}$ hides constants that depend on $x^0, y^0$ only.
\end{theorem}
\begin{proof}
    We can simply choose $c_{\tau}, c_{\sigma}, c_{\rho}$ such that the constant $C$ defined in Lemma \ref{cor-desc} are of constant order. For any fixed $z^*=(x^*,y^*)\in\cZ^*$, it holds that
    \[\begin{aligned}
    & f(x)-f(x^*)-\iprod{Ax-b}{y^*}\\
    \stackrel{(i)}{=}& f(x)-f(x^*)-\iprod{\nabla f(x^*)}{x-x^*}\\
    \stackrel{(ii)}{\geq}& \max\left\{\mu_{f}\nrm{x-x^*}^2,\frac{1}{2L_{f}}\nrm{\nabla f(x)-\nabla f(x^*)}^2\right\},
    \end{aligned}\]
    where (i) utilizes the KKT conditions $Ax^*=b$ and $\nabla f(x^*)=A^\T y^*$, (ii) is due to the assumption that $f$ is $L_f$-smooth and $\mu_f$-strongly convex.
    For simplicity, we denote $\delta_0^k:=f(x^{k+1})-f(x^*)-\iprod{Ax^{k+1}-b}{y^*}$, then
    \begin{align*}
        &\nrm{x^{k+1}-x^*}\leq \sqrt{\frac{\delta_0^k}{\mu_{f}}},\qquad
        \nrm{\nabla f(x^{k+1})-\nabla f(x^*)}\leq \sqrt{2L_{f}\delta_0^k}.
    \end{align*}
    By Lemma \ref{cor-desc} we have $\nrm{\nabla f(x^{k+1})-A^\top y^{k+1}}\leq \sqrt{\frac{C\delta^k}{\tau}}$. Recall the definition of $\cY^*$ in \eqref{def:Y*} and $\lambda_0\!>\!0$ is the minimum nonzero singular value of $A$, we obtain
    \begin{align*}
    \lambda_0 \dist{y^{k+1}, \cY^*}
    &\leq \nrm{A^\top y^{k+1}-\nabla f(x^*)}\\
    &\leq \nrm{A^\top y^{k+1}-\nabla f(x^{k+1})}+\nrm{\nabla f(x^{k+1})-\nabla f(x^*)}\\
    &\leq \sqrt{\frac{C\delta^k}{\tau}}+\sqrt{2L_{f}\delta_0^k}.
    \end{align*}
    Hence
    \begin{align*}
        \distw{\Lambda^{-1}}{z^{k+1}, \cZ^*}^2\leq 
        \frac{\delta_0^{k}}{\tau\mu_{f}}
        +\frac{3}{\lambda_0^2\sigma}\left(\frac{C\delta^k}{\tau}+L_{f}\delta_0^k\right).
    \end{align*}
    On the other hand, we have $\nrm{z^{k+1}-z^{k}}_{\Lambda^{-1}}^2\leq C\delta^k$ by Lemma \ref{cor-desc}. Thus, for
    \[\begin{aligned}
        M=\maxop{\frac{3C}{\lambda_0^2\sigma\tau}+cC, \frac{1}{\tau\mu_{f}}+\frac{3L_f}{\lambda_0^2\sigma}}
    \end{aligned}\]
    and any $\tilde{z}\in\cZ^*$, it holds that
    \begin{align*}
        \Delta_{k}(\tilde{z})-\Delta_{k+1}(\tilde{z})
        \geq& \delta_0^k+\delta^{k}\\
        {\geq}& \frac 1M \left(\distw{\Lambda^{-1}}{z^{k+1}, \cZ^*}^2+c\nrm{z^{k+1}-z^{k}}_{\Lambda^{-1}}^2\right)\\
        {\geq}& \frac 1M \inf_{z\in\cZ^*} \Delta_{k+1}(z).
    \end{align*}
    Then we have $\Delta^*_{k+1}\leq \frac{M}{M+1}\Delta^*_{k},$
    which yields a convergence rate of $\left(1-\frac{1}{M+1}\right)^k$. Here $M=\bigO{\kappa_{f}+\kappa_A^2}$ clearly.
\qed
\end{proof}
For example, we can take $\rho=\sigma=\frac{L_f}{4\nrm{A}^2}$, $\tau=\frac{1}{8L_f}$ for SOGDA, and take $\rho=\sigma=\frac{L_f}{2\nrm{A}^2}$, $\tau=\frac{1}{2L_f}$ for LALM.

\section{A byproduct: proximal OGDA for nonsmooth saddle point problem}  Note that, our analysis framework can be extended beyond bilinear problem or augmented Lagrangian function. As a byproduct of our analysis, we derive the convergence result of the proximal OGDA algorithm, which is a generalization of well-known OGDA algorithm for smooth problem \cite{wei2020linearOGDA,mokhtari2020unifiedOGDA}.
 Consider the problem
\begin{equation}\label{gen-saddle}
    \min_{x}\max_{y}\ \cL(x,y):=h(x)+\Psi(x,y)-s(y).
\end{equation}
We propose the proximal OGDA algorithm as % (understanding $x^{-1}=x^0, y^{-1}=y^0$)
\begin{align}\label{OGDA}
    \begin{split}
        x^{k+1} &= \prox_{\tau h}\left[x^k-\tau \left(2 g_x^k-g_x^{k-1}\right)\right],\\
        y^{k+1} &= \prox_{\sigma s}\left[y^k+\sigma\left(2 g_y^{k}-g_y^{k-1}\right)\right].
    \end{split}
\end{align}
Especially, when $h(\cdot)=\mathbb{I}_{\cX}(\cdot)$ and $s(\cdot)=\mathbb{I}_{\cY}(\cdot)$, we recover the OGDA on smooth minimax problem with closed convex domain.
The convergence analysis of the proximal OGDA \eqref{OGDA} is almost a direct implication of our analysis for \eqref{generalalgo}, but much simpler. Consider the potential function
\begin{align*}
    \Delta_k(z) =& \frac1{2}\|z^k-z\|^2_{\Lambda^{-1}}+\frac{1}{4}\|z^{k}-z^{k-1}\|^2_{\Lambda^{-1}}+\langle F(z^k)-F(z^{k-1}), z-z^{k}\rangle.
\end{align*}

We assume that the gradient $\nabla_x\Psi(x,y)$ is $L_{xx}$-Lipschitz continuous with respect to $x$ and $L_{xy}$-Lipschitz continuous with respect to $y$,
the gradient $\nabla_y\Psi(x,y)$ is $L_{yx}$-Lipschitz continuous with respect to $x$ and $L_{yy}$-Lipschitz continuous with respect to $y$, i.e.,
\begin{align*} 
    \nrm{\nabla_x\Psi(x,y)-\nabla_x\Psi(x^\prime,y)}&\leq L_{xx}\nrm{x-x^\prime},\ \forall x,x^\prime\in\mathcal{X}, \forall y\in\mathcal{Y},\\
    \nrm{\nabla_x\Psi(x,y)-\nabla_x\Psi(x,y^\prime)}&\leq L_{xy}\nrm{y-y^\prime},\ \forall x\in\mathcal{X}, \forall y,y^\prime\in\mathcal{Y},\\
    \nrm{\nabla_y\Psi(x,y)-\nabla_y\Psi(x^\prime,y)}&\leq L_{yx}\nrm{x-x^\prime},\ \forall x,x^\prime\in\mathcal{X}, \forall y\in\mathcal{Y},\\
    \nrm{\nabla_y\Psi(x,y)-\nabla_y\Psi(x,y^\prime)}&\leq L_{yy}\nrm{y-y^\prime},\ \forall x\in\mathcal{X}, \forall y,y^\prime\in\mathcal{Y}.
\end{align*}

\begin{theorem}
    The sequence $\{z^n\}_{n=0}^{+\infty}$ is generated by \eqref{OGDA} with the step sizes satisfying $\tau\leq \frac{1}{2L_{xx}}, \sigma\leq \frac{1}{2L_{yy}}$ and $\big(\frac{1}{\tau}-2L_{xx}\big)\big(\frac{1}{\sigma}-2L_{yy}\big)> 4\max\{L_{xy},L_{yx}\}^2$. 
    Then the sequence $\{z^n\}_{n=0}^{+\infty}$ converges to a saddle point of problem \eqref{gen-saddle}. Furthermore, let  $(\bar x_N, \bar y_N)  = \left(\frac{1}{N}\sum_{k=1}^{N}x^k, \frac{1}{N}\sum_{k=1}^{N}y^k\right)$, then for any $R_x, R_y>0$, it holds that
    \begin{align*}
        \max_{y\in\cY\cap \BB(y_0, R_y)} \cL(\ox_N, y)-\min_{x\in\cX\cap \BB(x_0, R_x)}\cL(x, \oy_N)\leq \frac{1}{N}\left(\frac{R_x^2}{\tau}+\frac{R_y^2}{\sigma}\right).
    \end{align*}
\end{theorem}
\begin{proof}
    By Lemma \ref{lemma:one-step}, it holds that for any $z\in\cZ=\cX\times\cY$ that 
    \begin{eqnarray*}
    	&&\Delta_k(z)-\Delta_{k+1}(z)\geq R(z^{k+1})-R(z)+\left\langle F(z^{k+1}), z^{k+1}-z\right\rangle\\
    	&&\qquad\qquad +\frac{1}{4}\nrm{z^{k+1}\!-\!z^{k}}^2_{\Lambda^{-1}}\!+\!\frac{1}{4}\nrm{z^{k}\!-\!z^{k-1}}^2_{\Lambda^{-1}}\!-\!\iprod{  F(z^{k})\!-\!F(z^{k-1}) }{ z^k\!-\!z^{k+1} }.
    \end{eqnarray*} 
    By definition, we have
    \[\begin{aligned}
        &\iprod{  F(z^{k})-F(z^{k-1}) }{ z^k-z^{k+1} }\\
        \leq &\nrm{x^{k}-x^{k+1}}\nrm{\nabla_x \Psi(z^{k})-\nabla_x \Psi(z^{k-1})}+ \nrm{y^{k}-y^{k+1}}\nrm{\nabla_y \Psi(z^{k})-\nabla_y \Psi(z^{k-1})}\\
        \leq &L_{xx}\nrm{x^{k}-x^{k+1}}\nrm{x^{k}-x^{k-1}}+L_{yy}\nrm{y^{k}-y^{k+1}}\nrm{y^{k}-y^{k-1}}\\
        &+L_{xy}\nrm{x^{k}-x^{k+1}}\nrm{y^{k}-y^{k-1}}+L_{yx}\nrm{y^{k}-y^{k+1}}\nrm{x^{k}-x^{k-1}}.
    \end{aligned}\]
    Then by $\left(\frac{1}{2\tau}-L_{xx}\right)\left(\frac{1}{2\sigma}-L_{yy}\right)\geq \max\{L_{xy},L_{yx}\}^2$, it follows that
    \[\begin{aligned}
    \delta^k:=\frac{1}{4}\nrm{z^{k+1}-z^{k}}^2_{\Lambda^{-1}}+\frac{1}{4}\nrm{z^{k}-z^{k-1}}^2_{\Lambda^{-1}}-\iprod{  F(z^{k})-F(z^{k-1}) }{ z^k-z^{k+1} }\geq 0.
    \end{aligned}\]
    Similarly, we have $\Delta_k(z)\geq \frac{1}{4}\nrm{z^k-z}_{\Lambda^{-1}}^2$. Hence, for any $z\in\cZ=\cX\times\cY$,
    \[\begin{aligned}
        R(z^{k+1})-R(z)+\left\langle F(z^{k+1}), z^{k+1}-z\right\rangle
        \leq \Delta_k(z)-\Delta_{k+1}(z).
    \end{aligned}\]
    The desired inequality follows directly by taking average.

    Furthermore, when $\left(\frac{1}{2\tau}-L_{xx}\right)\left(\frac{1}{2\sigma}-L_{yy}\right)> \max\{L_{xy},L_{yx}\}^2$, there actually exists a $\gamma>0$ such that for all $k\geq0$, $\delta^k\geq \gamma\nrm{z^{k+1}-z^{k}}^2_{\Lambda^{-1}}$, and hence for any saddle point $z^*$,
    \[\begin{aligned}
        \Delta_k(z^*)-\Delta_{k+1}(z^*)
        &\geq \delta_0^k+R(z^{k+1})-R(z^*)+\left\langle F(z^{k+1}), z^{k+1}-z^*\right\rangle\\
        &\geq \gamma \nrm{z^{k+1}-z^{k}}^2_{\Lambda^{-1}}.
    \end{aligned}\]
    Then we see $\left(\Delta_k(z^*)\right)$ is monotonically decreasing, and $\lim_k \nrm{z^{k+1}-z^k}=0$. By the fact $\Delta_k(z^*)\geq \frac{1}{4}\nrm{z^k-z^*}_{\Lambda^{-1}}^2$, $(z^n)$ has a subsequence $(z^{k_n})$ which converges to a $z^{\star}\in\cZ$. Then in 
    \begin{align*}
        z^{k+1}=\prox_{\Lambda R}\left(z^{k}-2\Lambda F(z^{k})+\Lambda F(z^{k-1})\right),
    \end{align*}
    we can plug in $k=k_n$ and let $n\to\infty$, by continuity it holds $z^{\star}=\prox_{\Lambda R}\left(z^{\star}-\Lambda F(z^{\star})\right)$. Therefore, $z^{\star}\in\cZ^*$ and hence the above argument shows that $\left(\Delta_k(z^{\star})\right)$ is also monotonically decreasing. Thus it must tend to 0, and $\lim_k z^k=z^{\star}$.
\qed
\end{proof}

\section{Numerical experiments}\label{sec:experiment}
In this section, we demonstrate the effectiveness of our proposed algorithm on several test problems.
\subsection{Applications to linear programming}
Consider the following linear programming problem:
\[
    \min_x\  r^\T x, \quad \st\ Ax\leq b,\ Cx=d,\ l\leq x\leq u,
\]
where $A\in\mathbbm{R}^{m_1\times n}, C\in\mathbbm{R}^{m_2\times n}$ are given matrices, $r,l,u\in\mathbbm{R}^{n}, b\in\mathbbm{R}^{m_1}, d\in\mathbbm{R}^{m_2}$ are given vectors. 
When $\rho_1=\rho_2=0$, the Lagrangian function is
\[
    \mathcal{L}(x,y,z) = r^\T x -y^\T(Cx-d)-z^\T(Ax-b).
\]
Let $f(x)=\mathbbm{1}_{[l,u]}(x), h(x)=0, \Psi(x,y,z) = \mathcal{L}(x,y,z)$ and $s(y,z)=\mathbbm{1}_{[-\infty,0]}(z)$ in \eqref{generalalgo}.
When $\rho_1, \rho_2>0$, the augmented Lagrangian function is 
\begin{equation*}
    \mathcal{L}_\rho(x,y,z) = r^\T x -y^\T(Cx-d)+\rho_1\|Cx-d\|^2+\frac{\rho_2}{2}\left\|\left[Ax-b-\frac{z}{\rho_2}\right]_+\right\|^2-\frac{\|z\|^2}{2\rho_2}.
\end{equation*}
Let $f(x)=\mathbbm{1}_{[l,u]}(x), h(x)=0, \Psi(x,y,z) = \mathcal{L}_\rho(x,y,z)$ and ${s(y,z)=0}$,
then the primal-dual methods \ref{generalalgo} can be easily implemented according to Lemma \ref{s0cone}.
In contrast, the subproblem of ADMM has no explicit solution. 
Thus, the inner iteration is needed to solve the subproblem approximately, which is computationally expensive.

Since the optimal solution is not necessarily unique, we utilize the objective function gap and the primal infeasibility to measure the optimality of the iterates:
\[
    \text{ObjGap}=\frac{\abs{r^\T x-r^\T x^*}}{\abs{r^\T x^*}},\qquad
    \text{Pinf}=\frac{\|[Ax-b]_+\|_2}{\|b\|_2}+\frac{\|Cx-d\|_2}{\|d\|_2}.
\]
We test several primal-dual methods on the two simple instances from netlib dataset.
Before applying the algorithms, we presolve the instances using gurobi \cite{gurobi} to simplify the problems.
The tested algorithms include the Chambolle-Pock method and the OGDA method based on the Lagrangian and augmented Lagrangian function, as well as SOGDA based on the augmented Lagrangian function.
We do not include the result of SOGDA based on the Lagrangian function since it has no convergence guarantee and the experiments confirm the point.
The results of ADMM are also not presented here because of the need to solve the subproblems.

For fairness of comparison, we choose $\tau, \sigma$ from $\{1e-5,1e-4,1e-3,1e-2,1e-1,1,1e1,1e2,1e3\}$.
For the algorithms based on the augmented Lagrangian function, $\rho_1=\rho_2=\hat{\rho}$ are also chosen from the set.
The best results under the parameter range are shown in Figure \ref{fig:lp}.

We can find that our proposed SOGDA-AL is efficient and competitive when compared to other methods.
Moreover, algorithms with an augmented Lagrangian function converge faster than those based on the Lagrangian function, especially for OGDA-AL in qap8 and CP-AL in both the two instances.

\begin{figure*}[!htb]
    \centering
        \subfigure[qap8]{\includegraphics[scale = 0.42]{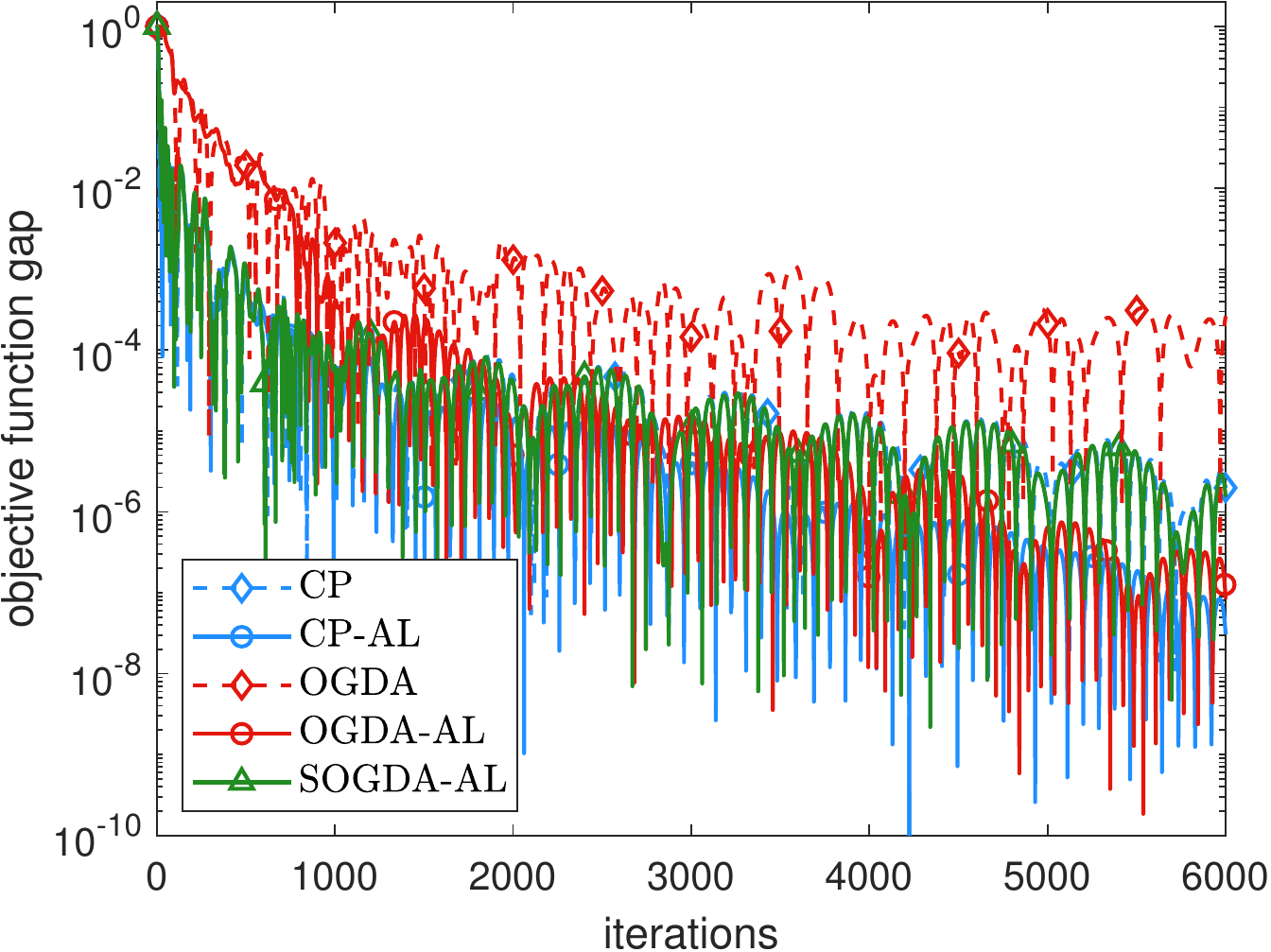}}
        \subfigure[qap8]{\includegraphics[scale = 0.42]{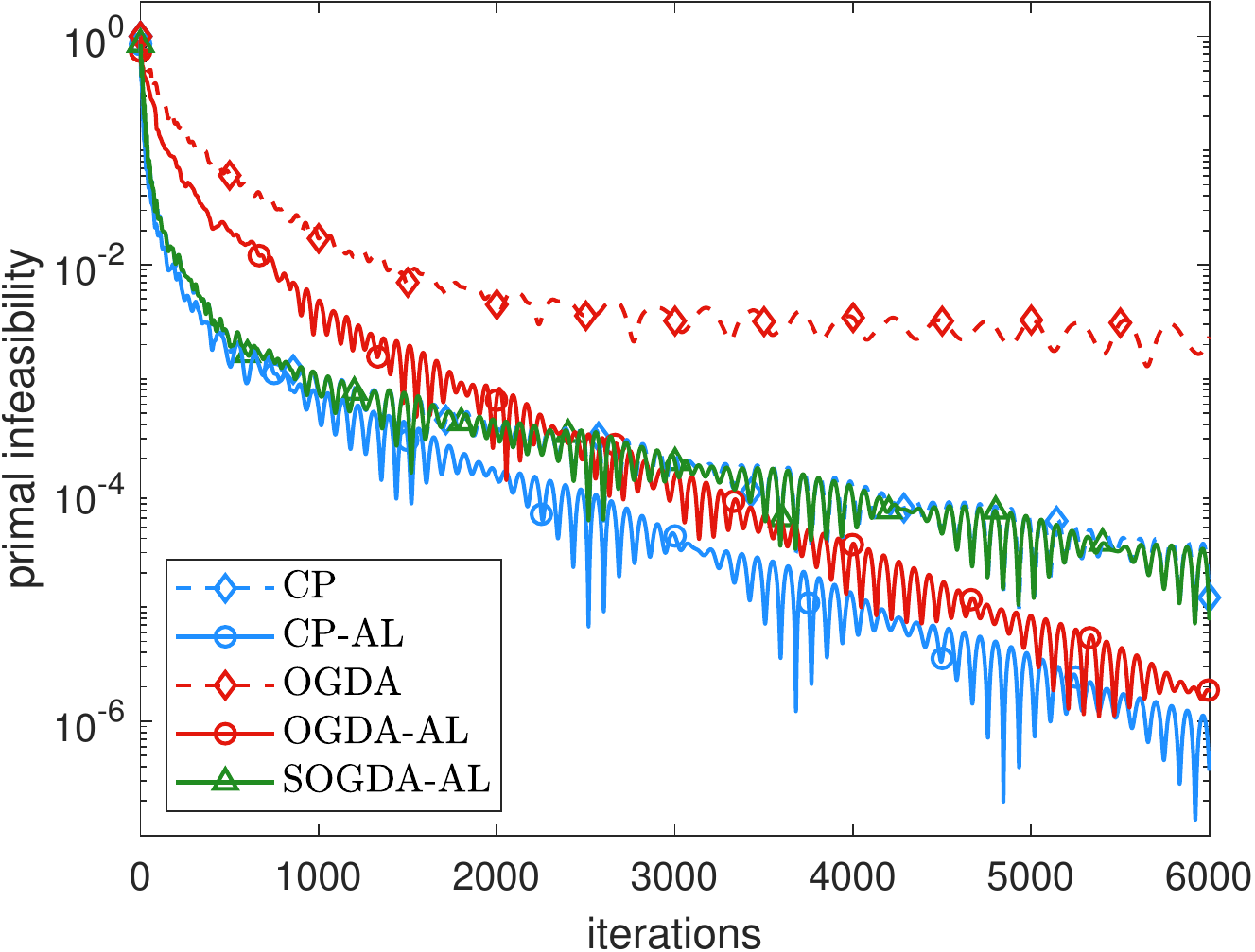}}\\
        \subfigure[sc50a]{\includegraphics[scale = 0.42]{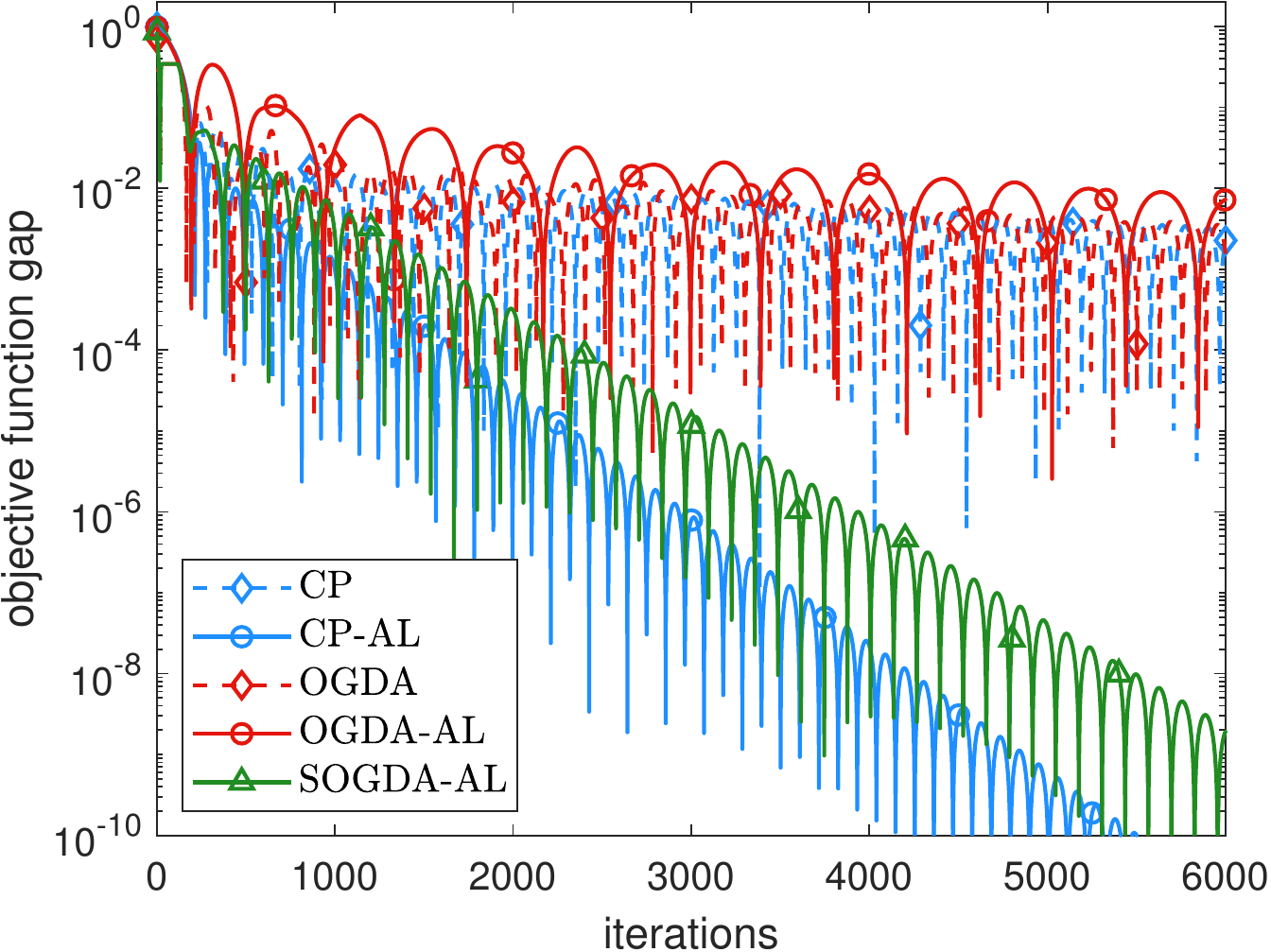}}
        \subfigure[sc50a]{\includegraphics[scale = 0.42]{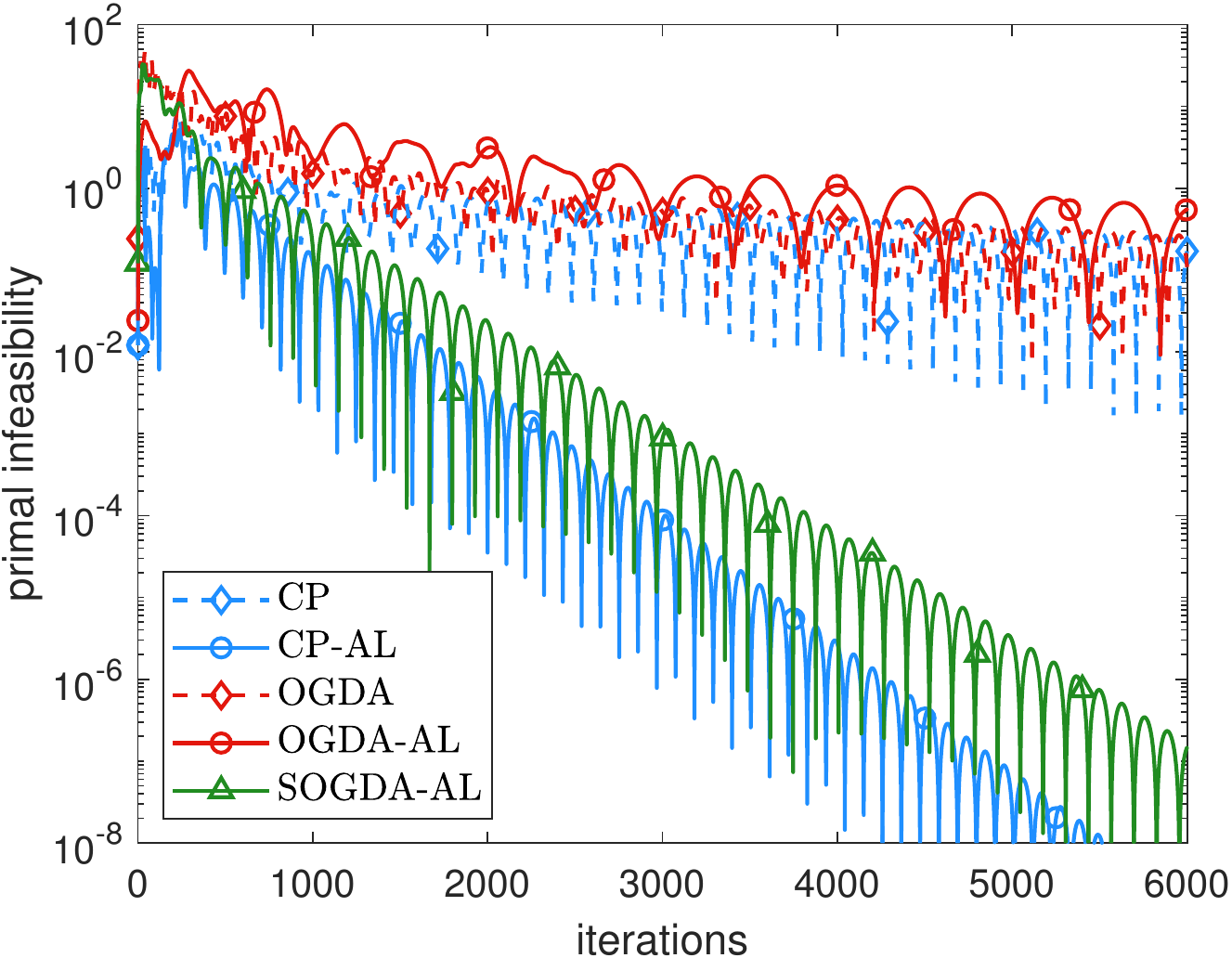}}
    \caption{Linear programming.}
        \label{fig:lp}
        % \vspace{0.2in}
\end{figure*}

\subsection{Applications to $\ell_1$ minimization}
\subsubsection{Basis pursuit}\label{bpsec}
We consider the basis pursuit problem, 
\begin{equation}\label{bp}
        \min_x\ \|x\|_1,\quad \st\ Ax=b,
\end{equation}
where $A\in\mathbb{R}^{m\times n}$ is of full row rank and $b\in\mathbb{R}^m$.
The augmented Lagrangian function can be written as 
\[
    \mathcal{L}_\rho(x,y) = \|x\|_1 - y^\T(Ax-b) + \frac\rho2\|Ax-b\|_2^2.
\]
Let $f(x)=0$, $h(x)=\|x\|_1$ and $\Psi(x,y)=-y^\T(Ax-b)+\frac\rho2\|Ax-b\|^2$ and $s(y)=0$ in \eqref{generalalgo},
we can easily get the primal-dual methods for solving \eqref{bp}.
We define the following metrics to describe the gap between the iterate and the optimum:
\[
    \text{RelErr} = \frac{\|x-x^*\|_2}{\max(\|x^*\|_2,1)},\qquad
    \text{Pinf} = \frac{\|Ax-b\|_2}{\|b\|_2}.
\]

\subsubsection{L1L1}
Consider the following problem
\begin{equation}\label{l1l1}
        \min_x\ \zeta\|x\|_1+\|Ax-b\|_1,
\end{equation}
where the settings of $A$ and $b$ are the same as the basis pursuit problem and $\zeta>0$ is a constant.
It is different from LASSO in that the square of $\ell_2$ norm is replaced by $\ell_1$ norm.
We call the problem \eqref{l1l1} the L1L1 problem here.
In order to write the problem in the form of an affinely constrained problem,
we introduce $r := b-Ax$ and \eqref{l1l1} becomes
\[
    \min_{x,r}\ \zeta\|x\|_1+\|r\|_1,\quad \st\ Ax-b+r = 0.
\]
Then the augmented Lagrangian function of the problem is
\[
    \mathcal{L}_\rho(x,r,y) = \zeta\|x\|_1 + \|r\|_1 - y^\T(Ax-b+r) + \frac\rho2\|Ax-b+r \|_2^2.
\]
substitute $f(x,r)=0$, $h(x,r)=\zeta\|x\|_1+\|r\|_1$, $\Psi(x,r,y)=- y^\T(Ax-b+r) + \frac\rho2\|Ax-b+r \|_2^2$ and $s(y)=0$ into \eqref{generalalgo} to get primal-dual algorithms for solving \eqref{l1l1}.
We define the following metrics to measure the gap between the iterate and the optimum of the problem:
\begin{align*}
    \text{RelErr} = \frac{\|x-x^*\|_2}{\max(\|x^*\|_2,1)},\qquad
    \text{Pinf} = \frac{\|Ax-b+r\|_2}{\|b\|_2}.
\end{align*}

\subsubsection{Numerical comparison}
Our test problems\cite{milzarek2014semismooth} are constructed as follows.
Firstly, we create a sparse solution $x^* \in \mathbb{R}^{n}$ with $k$ nonzero entries, where $n = 512^2 = 262144$ and $k = [n/40] = 5553$. 
The $k$ different indices are sampled uniformly from $\{1,2,\cdots,n\}$.
The magnitude of each nonzero element is determined by $x^*_{i}=\eta_{1}(i) 10^{d \eta_{2}(i) / 20}$, 
where $d$ is a dynamic range, $\eta_{1}(i)$ and $\eta_{2}(i)$ are uniformly randomly chosen from $\{-1, 1\}$ and $[0, 1]$, respectively. 
The linear operator $A$ is defined as $m = n/8 = 32768$ random cosine measurements, i.e., $A x=(\operatorname{dct}(x))_{J}$, 
where $\operatorname{dct}$ is the discrete cosine transform and the set $J$ is a subset of $\{1, 2, \cdots, n\}$ with $m$ elements.

We choose the step size $\tau$ among $\{0.5,1,1.5,2,2.5\}$ and the ratio $\tau/\sigma$ among $\{1,1.5,2,2.5,3\}$.
For the algorithms based on the augmented Lagrangian function, the penalty factor $\rho$ is chosen among $\{0.5,1,1.5,2,2.5,3\}$.
ADMM is implemented in the software YALL1 \cite{yang2011alternating}.
It is worth noticing that the subproblem of ADMM for the dual problem can be solved exactly since $AA^T=I$.
The relative error and primal infeasibility with respect to the total number of iterations for different problems are shown in Figure \ref{fig:bp_05_25} and \ref{fig:l1l1_05_25}.

The augmented term potentially accelerates the primal-dual method, especially for OGDA and high accuracy solution.
SOGDA-AL outperforms other algorithms in most of the problems.
By contrast, although ADMM solves the subproblem exactly and adopts the strategy of updating the penalty parameters, it has no obvious advantage over the other algorithms.
In addition, the linear convergence rate of the primal-dual methods can be observed, as we have proved.

\begin{figure*}[!htb]
    \centering
      \subfigure[20dB]{\includegraphics[scale = 0.42]{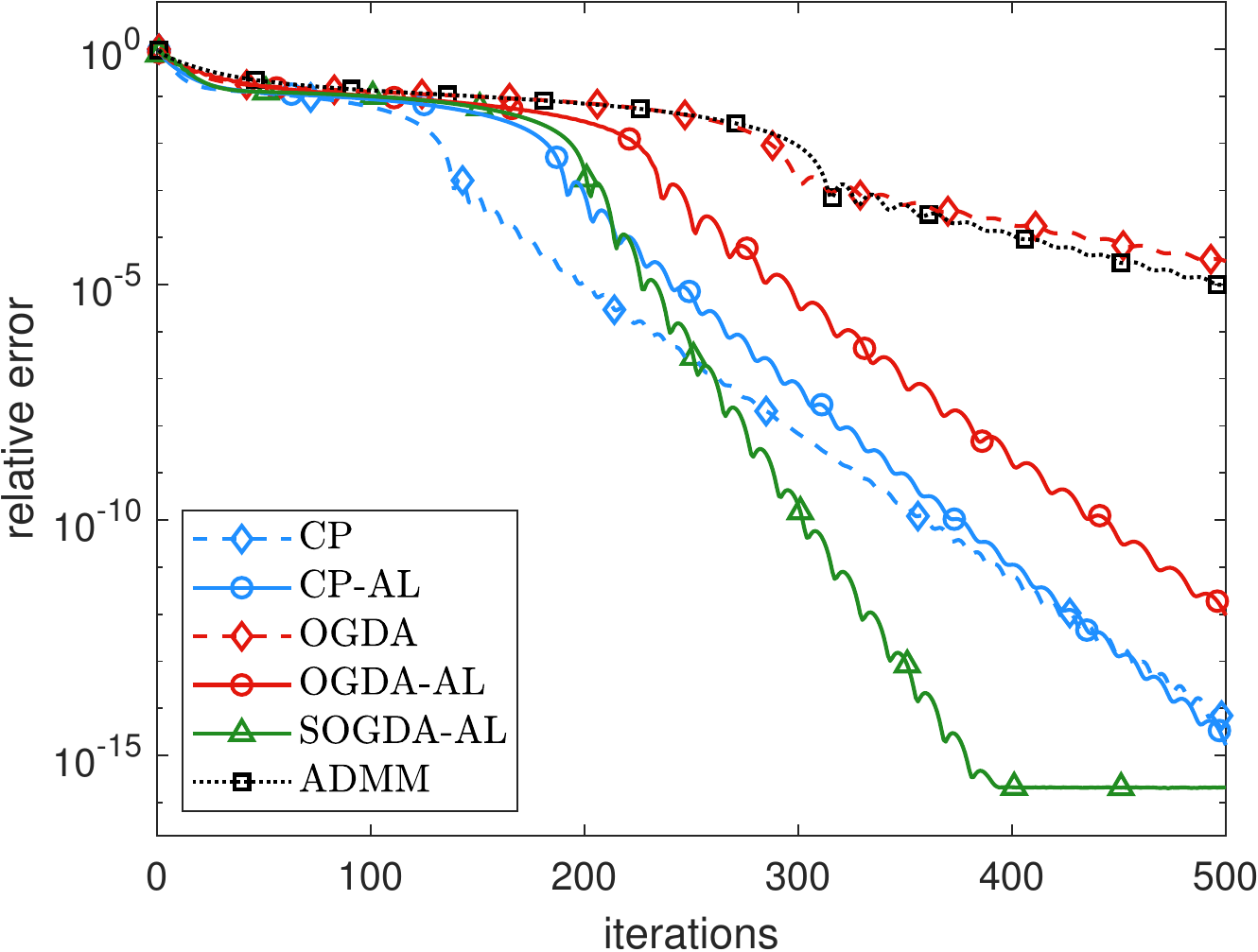}}
      \subfigure[20dB]{\includegraphics[scale = 0.42]{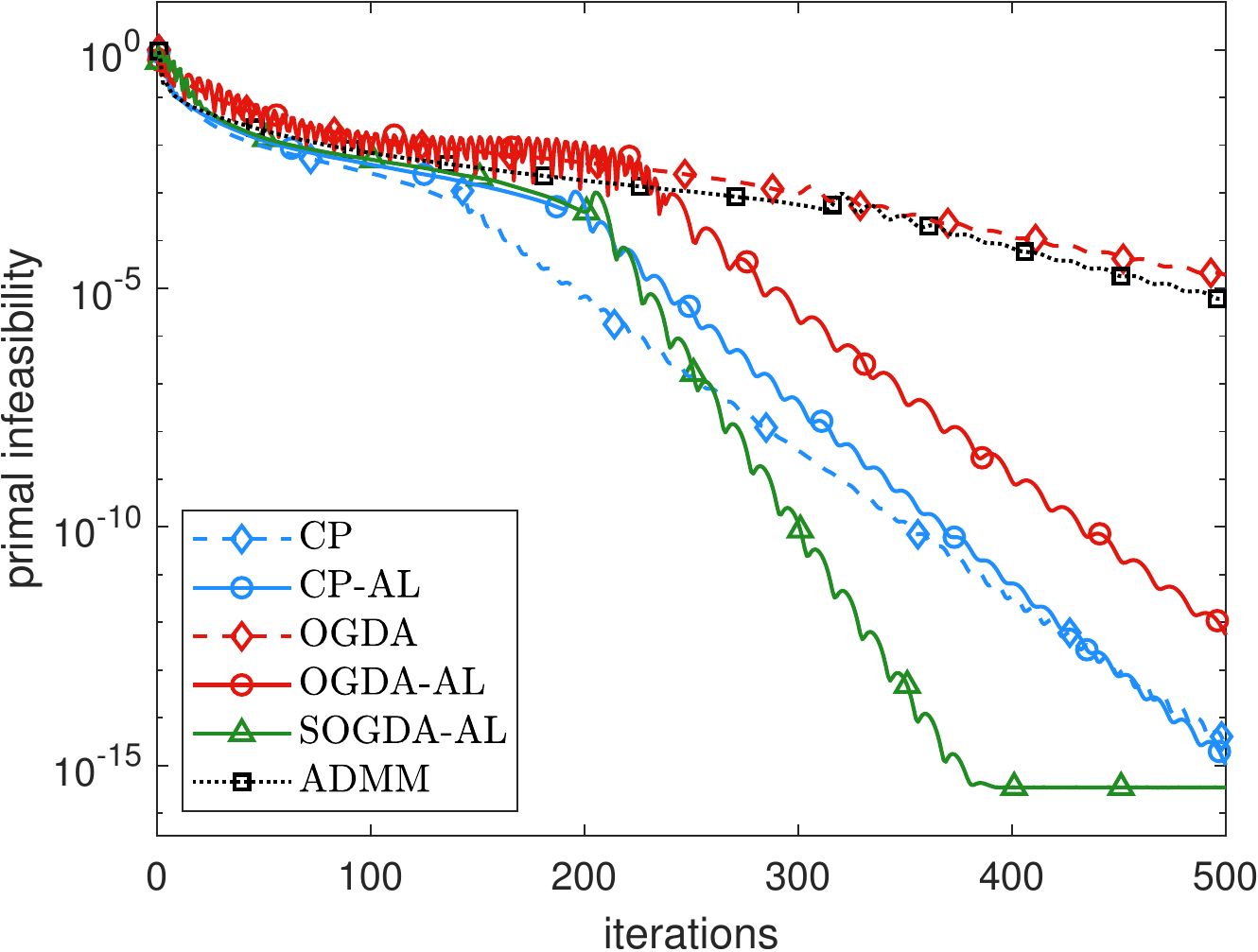}} \\
      \subfigure[40dB]{\includegraphics[scale = 0.42]{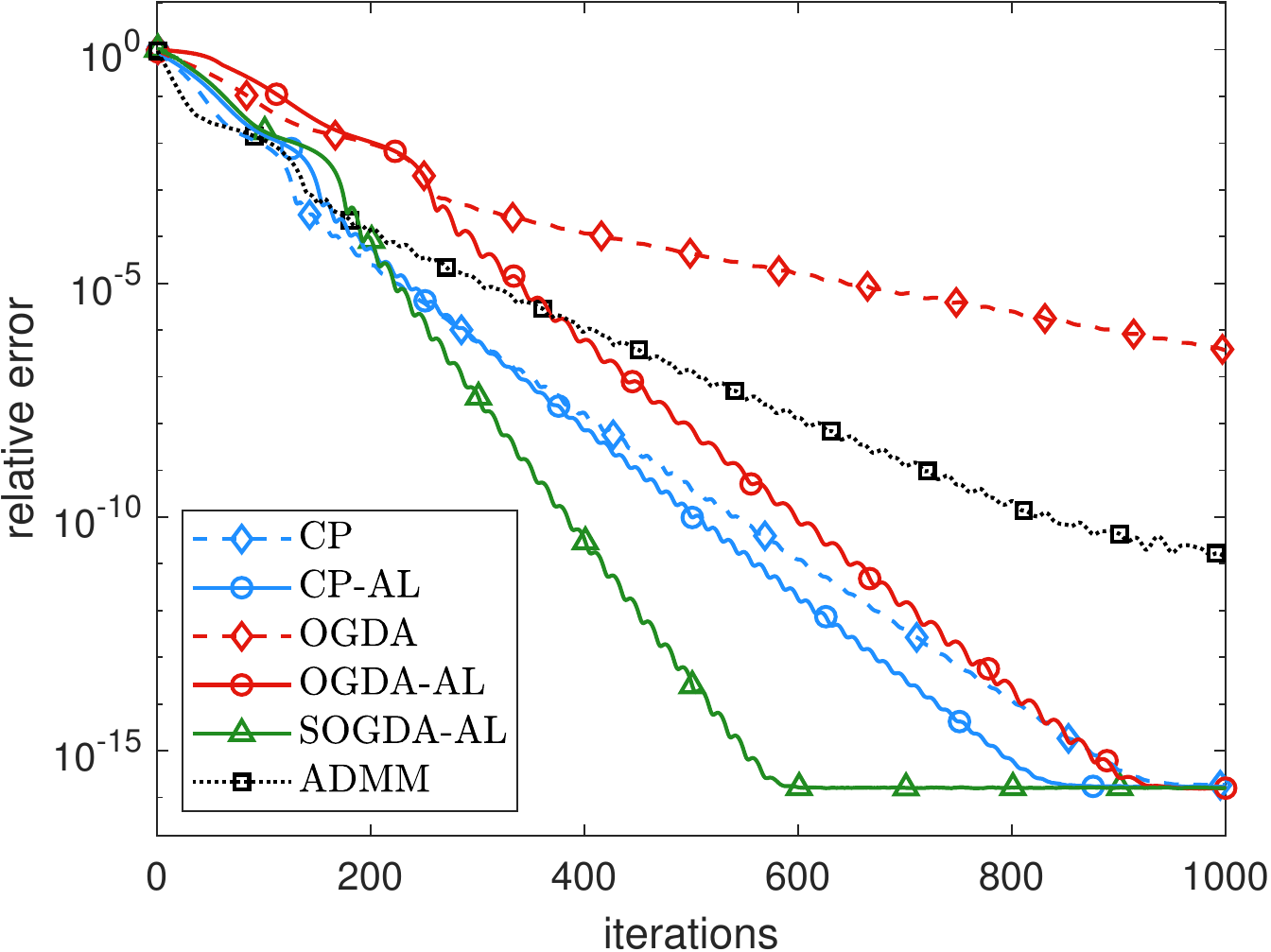}}
      \subfigure[40dB]{\includegraphics[scale = 0.42]{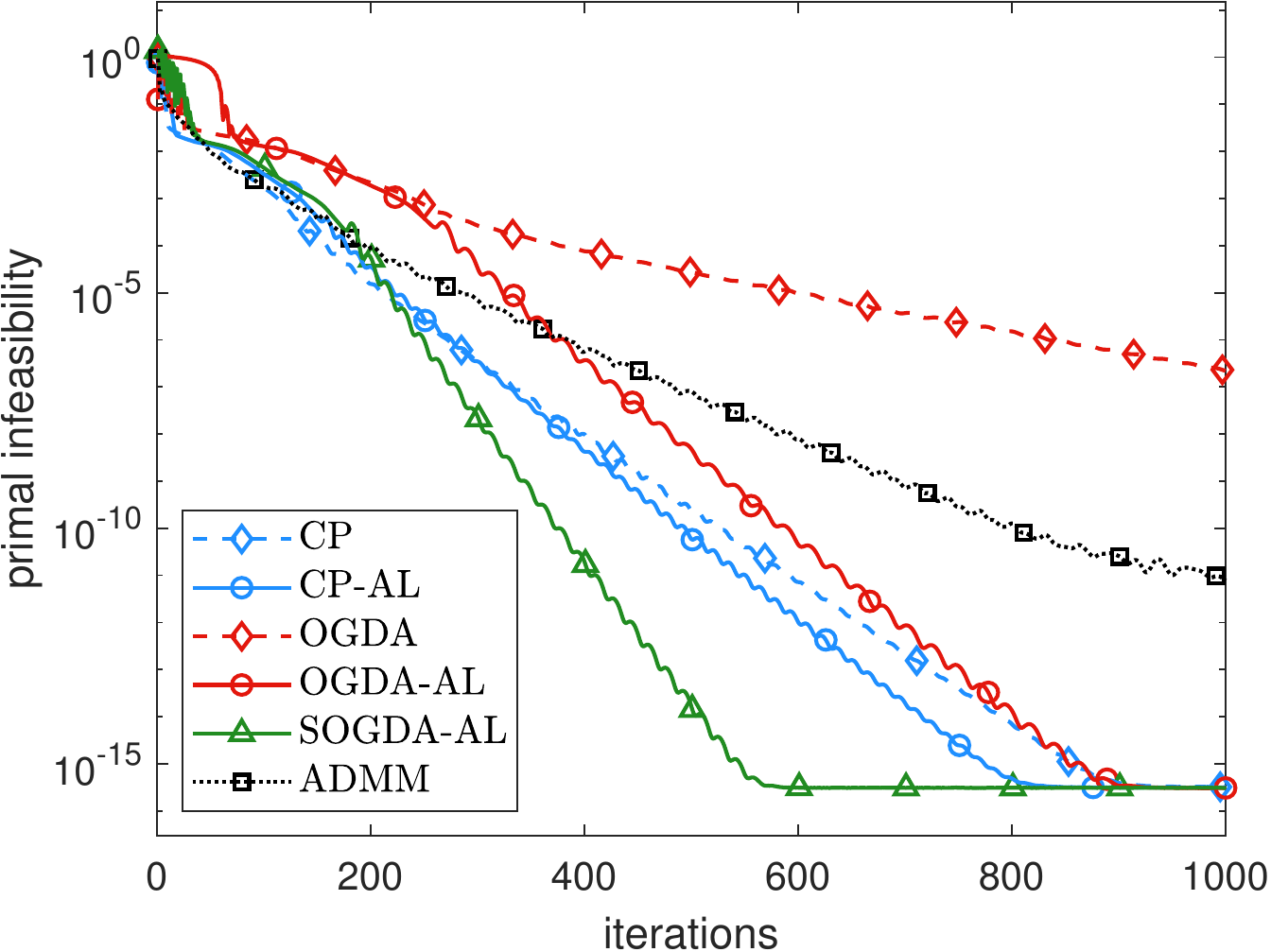}}
    \caption{Basis pursuit.}
      \label{fig:bp_05_25}
\end{figure*}

\begin{figure*}[!htb]
\centering
    \subfigure[20dB]{\includegraphics[scale = 0.42]{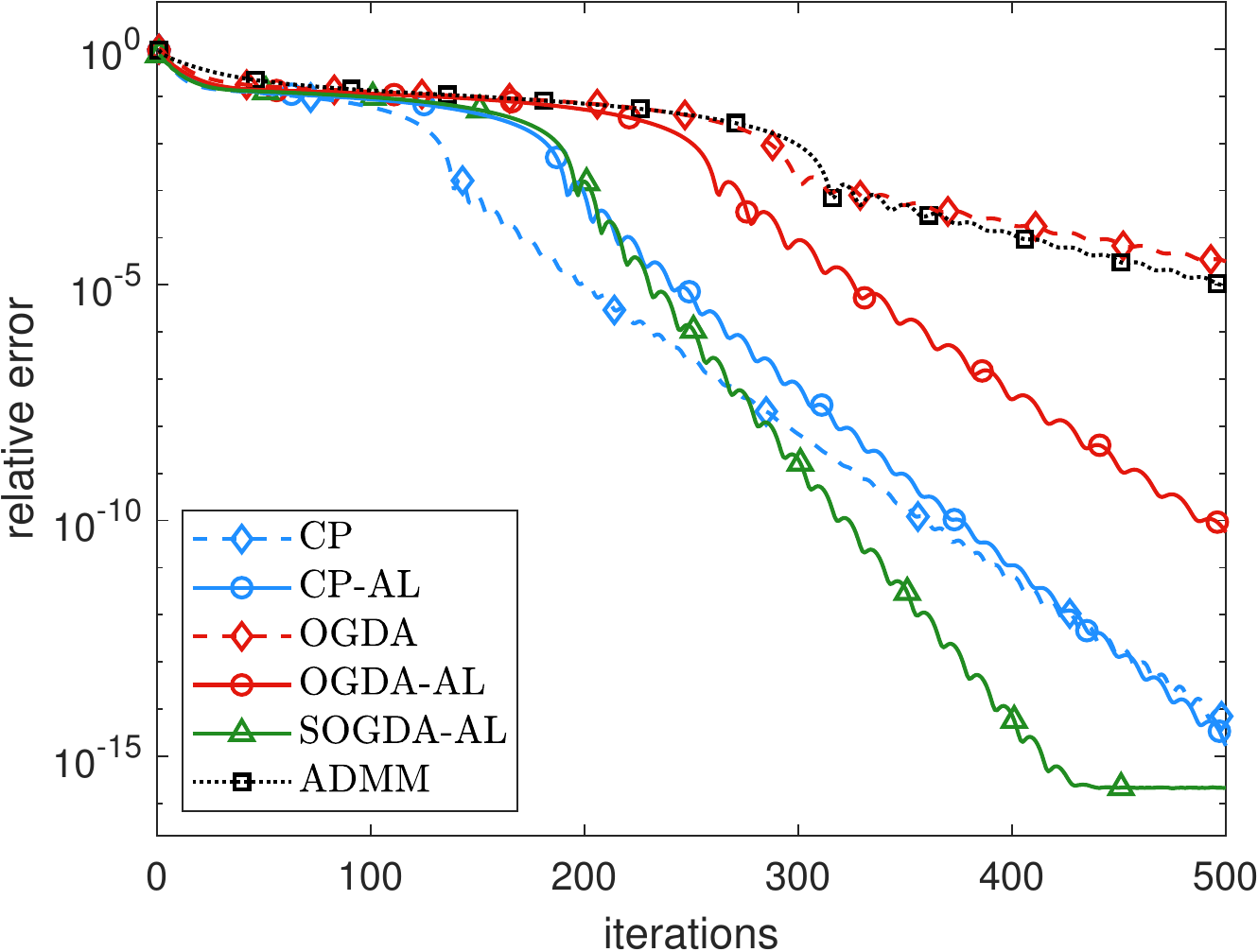}}
    \subfigure[20dB]{\includegraphics[scale = 0.42]{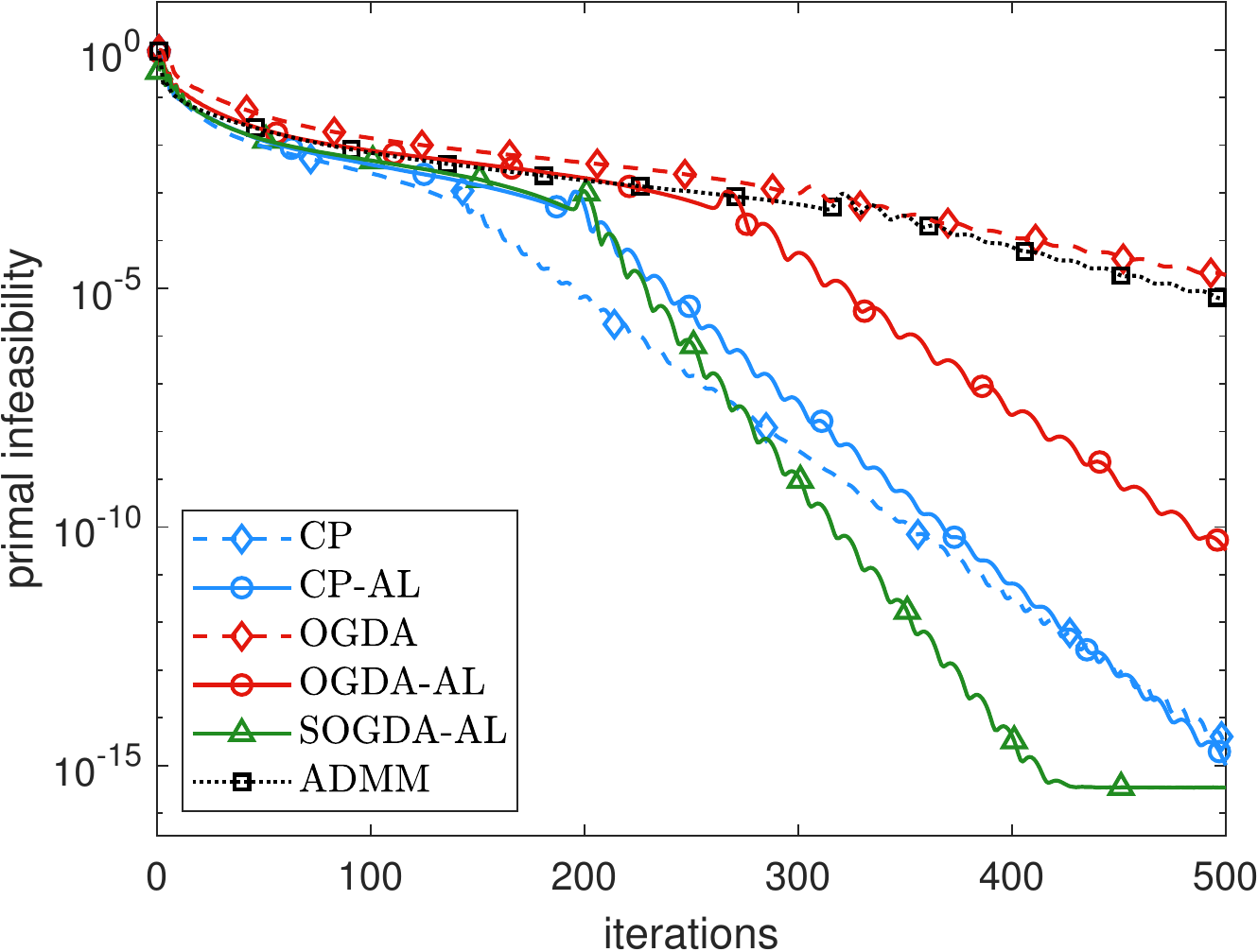}} \\
    \subfigure[40dB]{\includegraphics[scale = 0.42]{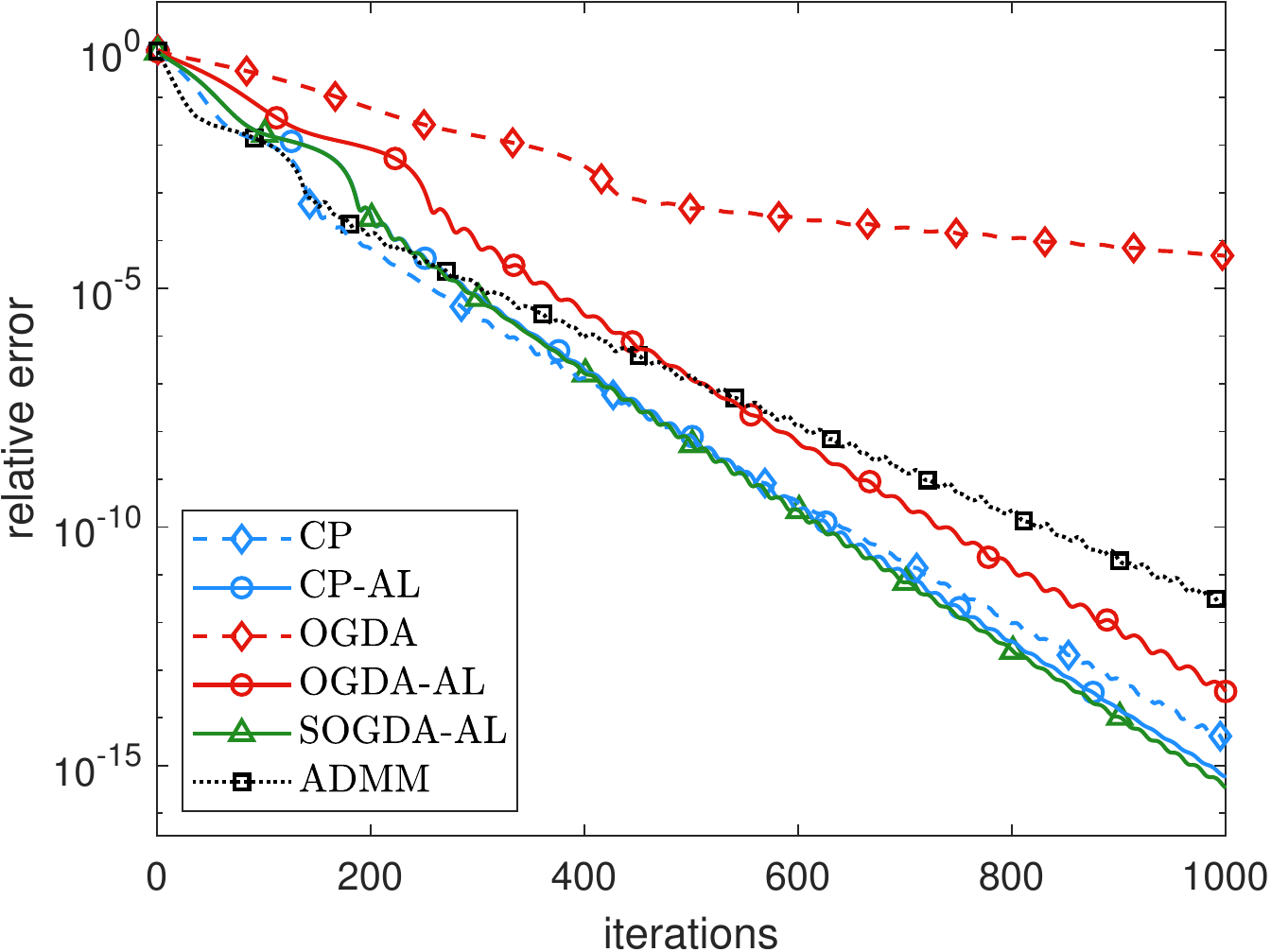}}
    \subfigure[40dB]{\includegraphics[scale = 0.42]{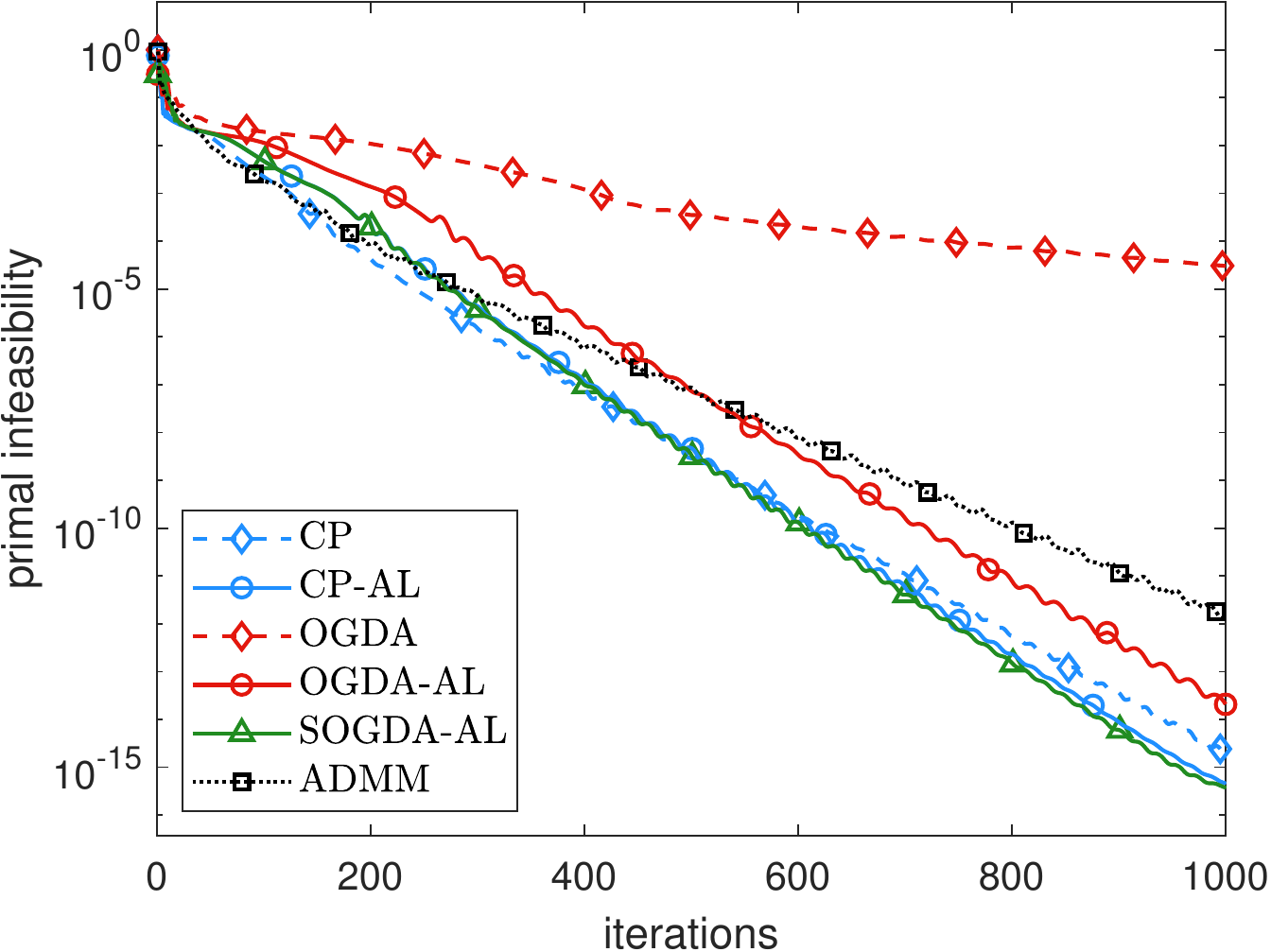}}
\caption{L1L1.}
    \label{fig:l1l1_05_25}
\end{figure*}

\subsection{Applications to multi-block problems}
\subsubsection{Multi-block basis pursuit}
Suppose that the data in \eqref{bp} can be partitioned into $N$ blocks:
$x = [x_1, x_2, \cdots, x_N]$, $A = [A_1, A_2, \cdots, A_N]$, 
then the problem can be rewritten as follows:
\begin{equation}\label{mbbp}
    \min_{x_1,x_2,\cdots,x_N}\ \sum_{i=1}^N\|x_i\|_1,\quad \st\ \sum_{i=1}^N A_i x_i=b,
\end{equation}
where $x_i\in\mathbbm{R}^{n_i}$, $A_i\in\mathbbm{R}^{m\times n_i}$, $\sum_{i=1}^N n_i = n$.
The augmented Lagrangian function becomes
\[
    \mathcal{L}_\rho(x_1,x_2,\cdots,x_N,y) = \sum_{i=1}^N\|x_i\|_1 - y^\T\left(\sum_{i=1}^N A_i x_i-b\right) + \frac\rho2\left\|\sum_{i=1}^N A_i x_i-b\right\|_2^2.
\]
The primal-dual methods can be obtained by setting 
$f(x)=0$, $h(x)=\sum_{i=1}^N\|x_i\|_1$, $\Psi(x,y)=- y^\T\left(\sum_{i=1}^N A_i x_i-b\right) + \frac\rho2\left\|\sum_{i=1}^N A_i x_i-b\right\|_2^2$ and $s(y)=0$ in \eqref{generalalgo}.
In fact, the derived multi-block algorithm is equivalent to the one-block algorithm obtained in Section \ref{bpsec}
because the update of the primal variable is separable and applied in a Jacobian fashion.
Therefore, we only need to test the primal-dual methods in the case of $N=1$ and the results of the multi-block cases are all the same.

As for the multi-block ADMM, the subproblem of minimizing $\mathcal{L}_{\!\rho}(x_{1},\cdots,x_N,y)$ with respect to $x_i$ has no explicit solution.
In order to overcome the obstacle, we introduce $u_i=x_i$ to get an equivalent form:
\[
    \min_{\substack{x_1,x_2,\cdots,x_N\\u_1,u_2,\cdots,u_N}}\ \sum_{i=1}^N\|x_i\|_1,\quad \st\ \sum_{i=1}^N A_i u_i=b, \quad x_i = u_i,\ i=1,\cdots,N.
\]
The augmented Lagrangian function becomes 
\begin{align*}
    \mathcal{L}^\prime_\rho(x_1,x_2,\cdots,x_N,y,z) =& \sum_{i=1}^N\|x_i\|_1 - y^\T\left(\sum_{i=1}^N A_i x_i-b\right) + \frac{\rho_1}2\left\|\sum_{i=1}^N A_i x_i-b\right\|^2\\
    &-z^\T(x-u)+\frac{\rho_2}2\|x-u\|^2.
\end{align*}
Then all the subproblems of multi-block ADMM can be solved exactly. 
In practice, we update the variables in the order of $x_1,x_2,\cdots,x_N,y,z$.
We define the following two metrics to measure the optimality:
\[
    \text{RelErr} = \frac{\|x-x^*\|_2}{\max(\|x^*\|_2,1)},\qquad
    \text{Pinf} = \frac{\|Ax-b\|_2+\|x-u\|_2}{\|b\|_2}.
\]

We test primal-dual algorithms and multi-block ADMM in a randomly generated example which is constructed as follows.
First, we randomly generate a sparse solution $x^*\in\mathbbm{R}^n$ with $k$ nonzero entries drawn from the standard Gaussian distribution.
The matrix $A$ is also generated by standard Gaussian distribution and the vector $b$ is set to be $Ax^*$. 
Then $x^*$ and $A$ is partitioned evenly into $N$ blocks. 
In our experiments, we set $n=1000, m=300, k=60$ and $N=1,2,5,10$.
We choose the primal step size $\tau$ in $\{0.01,0.02,0.03,0.04,0.05\}$ and the ratio $\tau/\sigma$ in $\{0.5,1,1.5,2,2.5,3\}$.
The penalty factor $\rho$ is chosen from $\{0.01,0.02,0.03,0.04,0.05\}$ for the primal-dual methods based on the augmented Lagrangian function.
The numerical results are shown in Figure \ref{fig:bp_gaussian}.

As observed in Figure \ref{fig:bp_gaussian}, SOGDA-AL has some advantages among primal-dual algorithms.
Algorithms based on the augmented Lagrangian function have faster convergence to high-accuracy solutions.
Multi-block ADMM performs well in the case of $N=1$ while converges more and more slowly as the number of blocks increases.
\begin{figure*}[!htb]
    \centering
        \subfigure[primal-dual methods]{\includegraphics[scale = 0.42]{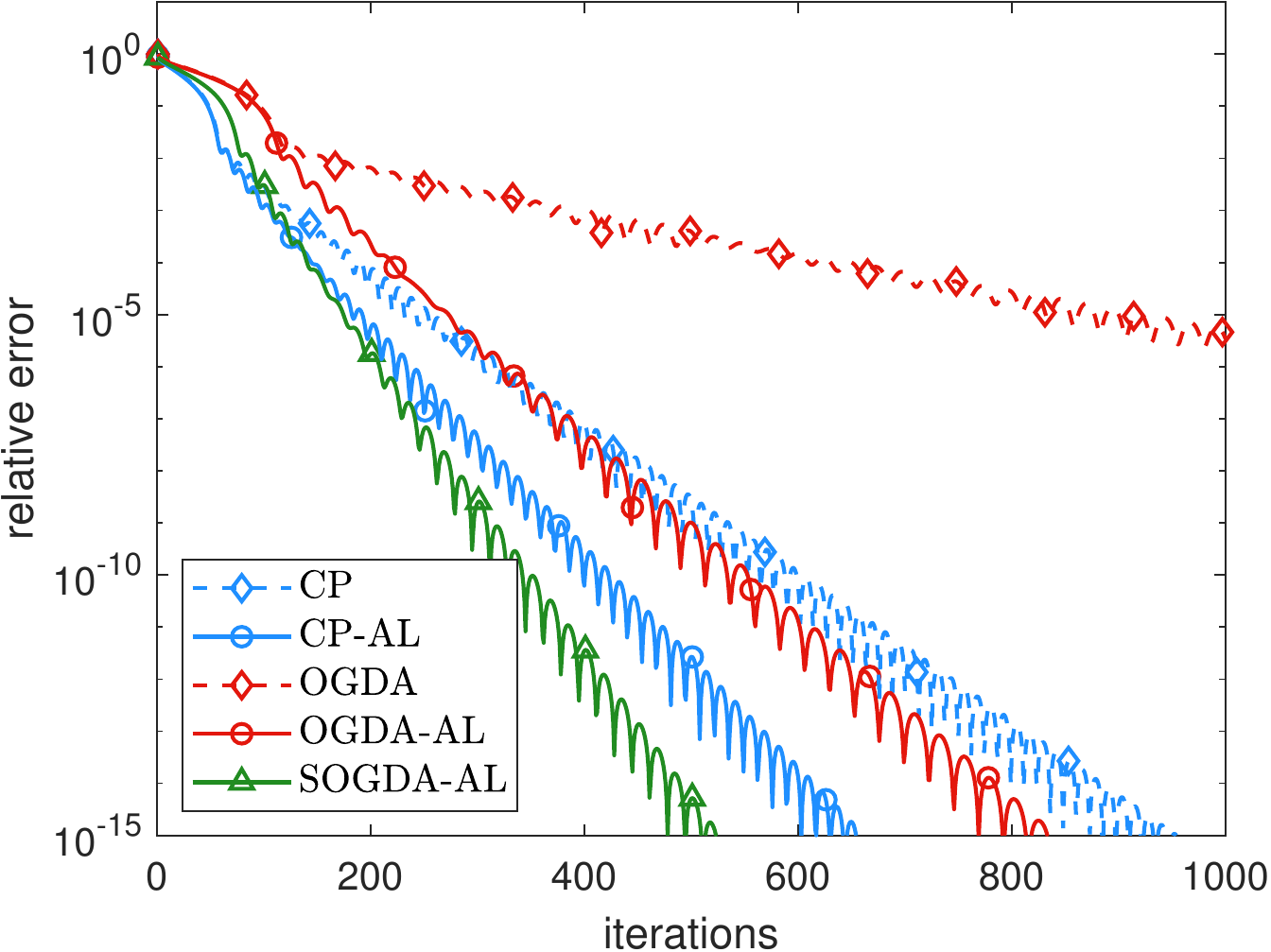}}
        \subfigure[primal-dual methods]{\includegraphics[scale = 0.42]{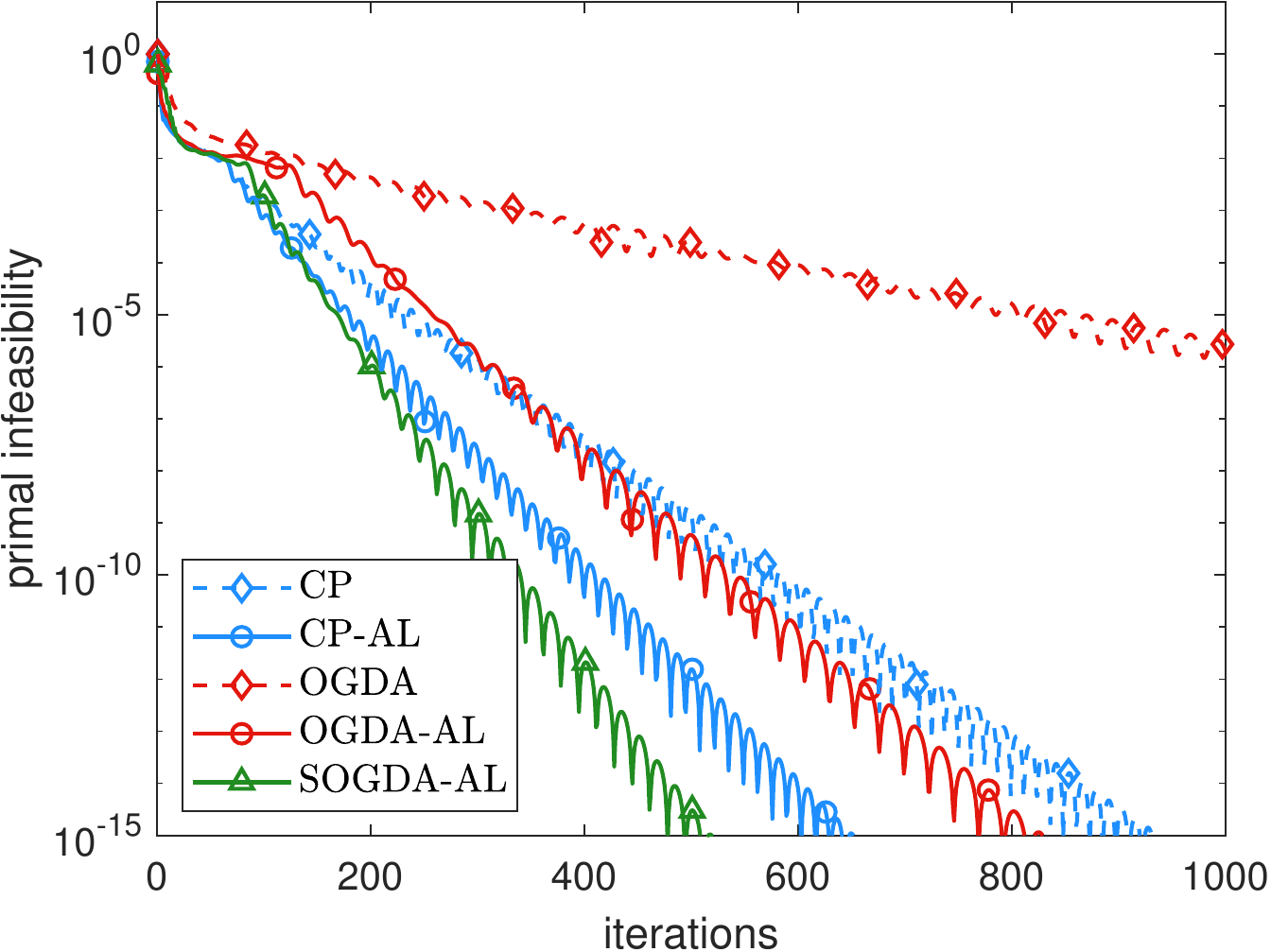}}\\
        \subfigure[multi-block ADMM]{\includegraphics[scale = 0.42]{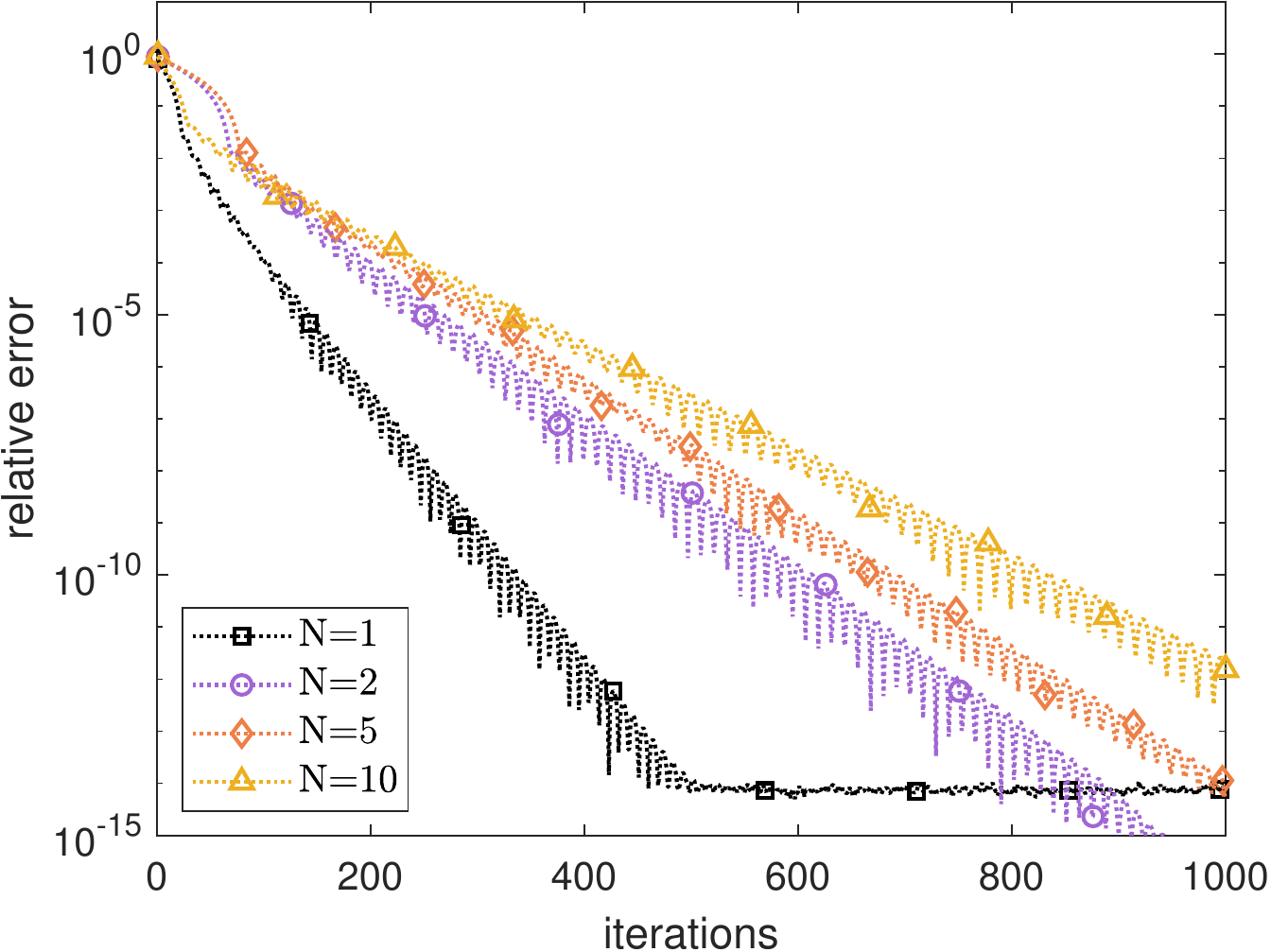}}
        \subfigure[multi-block ADMM]{\includegraphics[scale = 0.42]{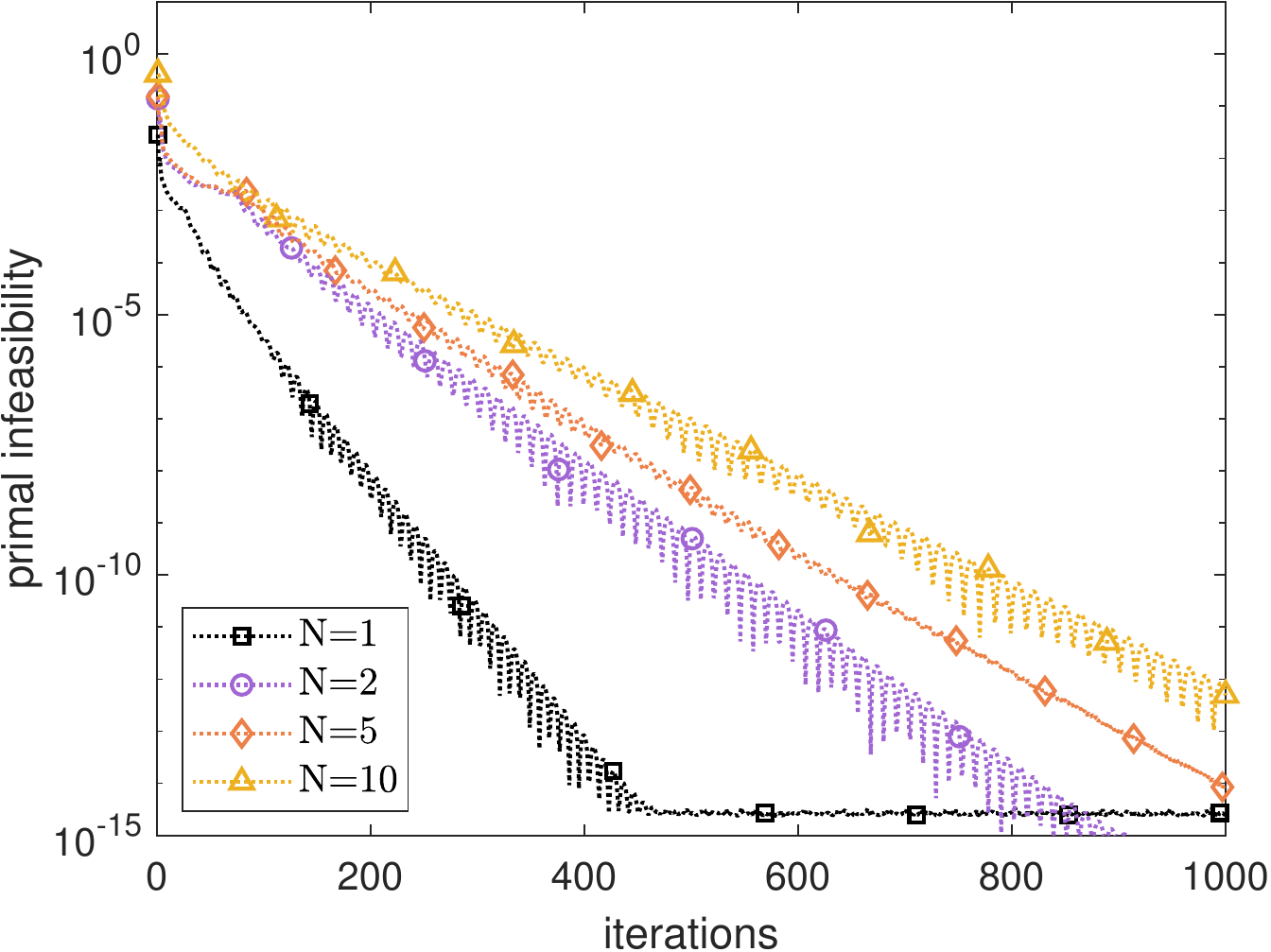}}
    \caption{Multi-block BP.}
        \label{fig:bp_gaussian}
        % \vspace{0.2in}
\end{figure*}

\subsubsection{Non-convergent examples for the direct extension of ADMM}
In this section, we test primal-dual methods on non-convergent examples of multi-block ADMM present in \cite{chen2016direct}.
According to the convergence analysis established in Section \ref{sec:conv}, the primal-dual algorithms converge to the optimal point on these examples,
which directly demonstrates the advantage of primal-dual methods over multi-block ADMM.

The first example is to solve the linear equation:
\begin{equation}\label{eg1}
\left(\begin{array}{l}
    1 \\
    1 \\
    1
    \end{array}\right) x_{1}+\left(\begin{array}{l}
    1 \\
    1 \\
    2
    \end{array}\right) x_{2}+\left(\begin{array}{l}
    1 \\
    2 \\
    2
    \end{array}\right) x_{3}=\left(\begin{array}{l}
    0 \\
    0 \\
    0
    \end{array}\right).
\end{equation}
where $A = [A_1,A_2,A_3] = \left(\begin{array}{lll}1 & 1 & 1 \\1 & 1 & 2 \\1 & 2 & 2\end{array}\right)$, $b = \left(\begin{array}{l}0 \\0 \\0\end{array}\right)$.
Since $A$ is full rank, the unique solution of the equation is $x_1=0, x_2=0$ and $x_3=0$.
While applying multi-block ADMM on the problem, the penalty factor $\rho$ just scales the dual variable $y_k$ by a constant, hence any choice of $\rho$ is equivalent.

The second example is solving
\begin{equation}\label{eg2}
    \begin{split}
        &\min\ \frac12 x_{1}^{2},\\
        &\st\ \left(\begin{array}{l}1 \\ 1 \\ 1\end{array}\right) x_{1}+\left(\begin{array}{l}1 \\ 1 \\ 1\end{array}\right) x_{2}+\left(\begin{array}{l}1 \\ 1 \\ 2\end{array}\right) x_{3}+\left(\begin{array}{l}1 \\ 2 \\ 2\end{array}\right) x_{4}=\left(\begin{array}{l}0 \\ 0 \\ 0\end{array}\right),
    \end{split}
\end{equation}
As presented in \cite{chen2016direct}, the multi-block ADMM diverges for any $\rho\geq0$.

In the experiments, we randomly select a nonzero point as the initial point.
For primal-dual methods, we choose $\rho, \tau, \sigma$ from $\{1e-5,1e-4,1e-3,1e-2,1e-1,1,1e1,1e2,1e3\}$.
For the multi-block ADMM, we choose the penalty factor $\rho=1$ for both examples.
The numerical results are shown in Figure \ref{fig:diverge_admm}.

As expected, multi-block ADMM diverges in both examples
while the primal-dual algorithms show good convergence.
Algorithms based on the augmented Lagrangian function converge faster than those based on the Lagrangian function, especially for the Chambolle-Pock method. 
Furthermore, SOGDA-AL shows competitive performance again.
\begin{figure*}[!htb]
    \centering
        \subfigure[Example1]{\includegraphics[scale = 0.42]{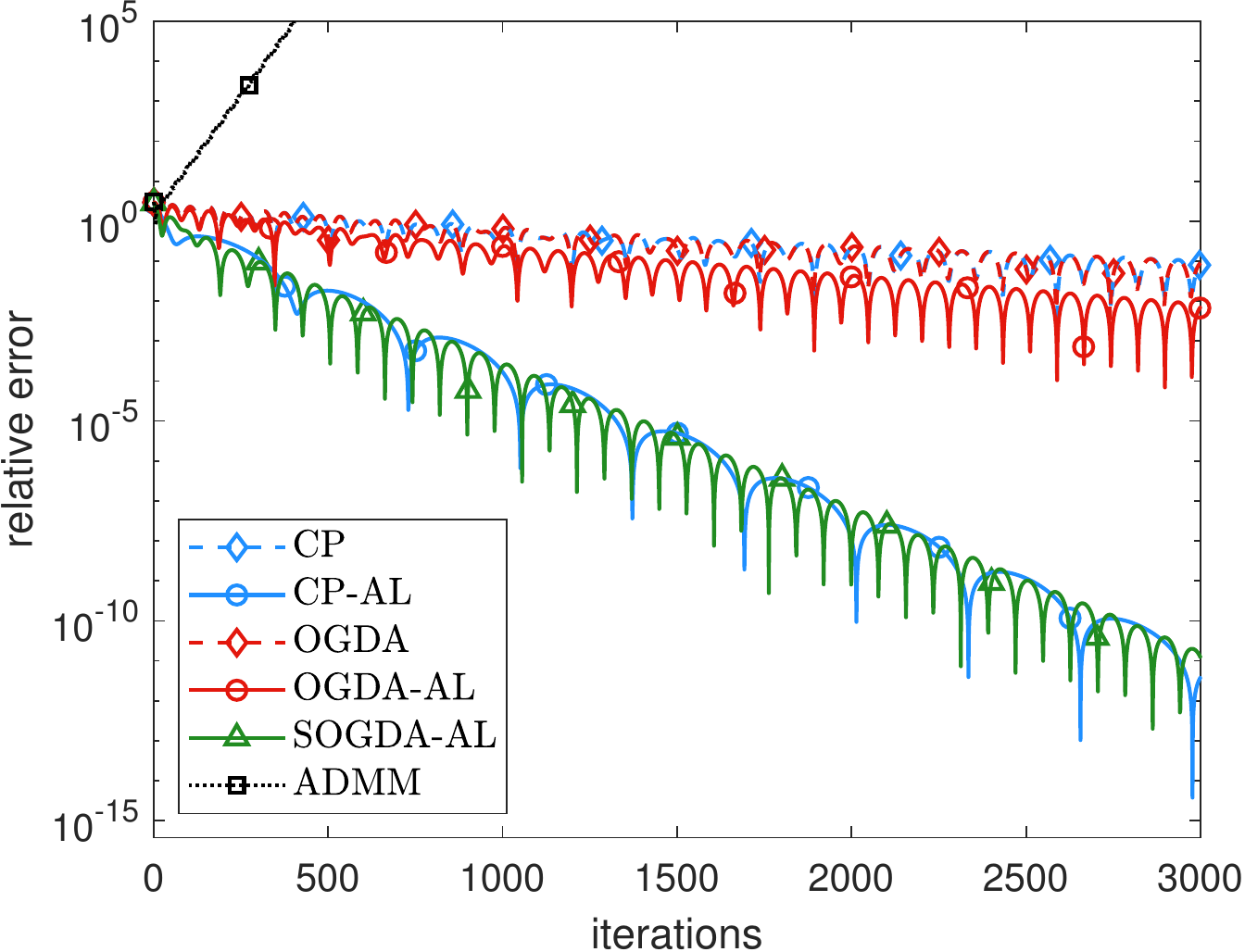}}
        \subfigure[Example1]{\includegraphics[scale = 0.42]{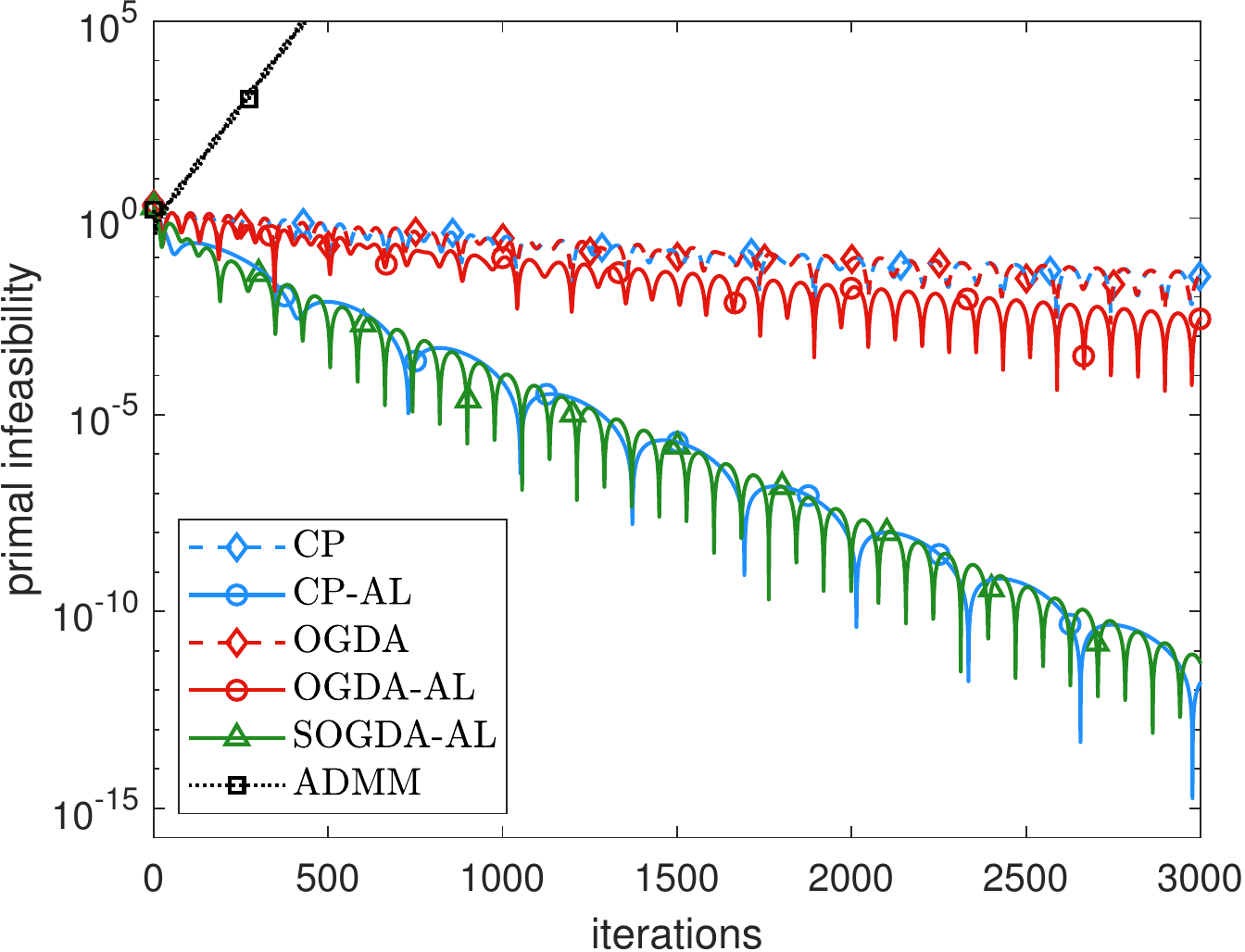}}\\
        \subfigure[Example2]{\includegraphics[scale = 0.42]{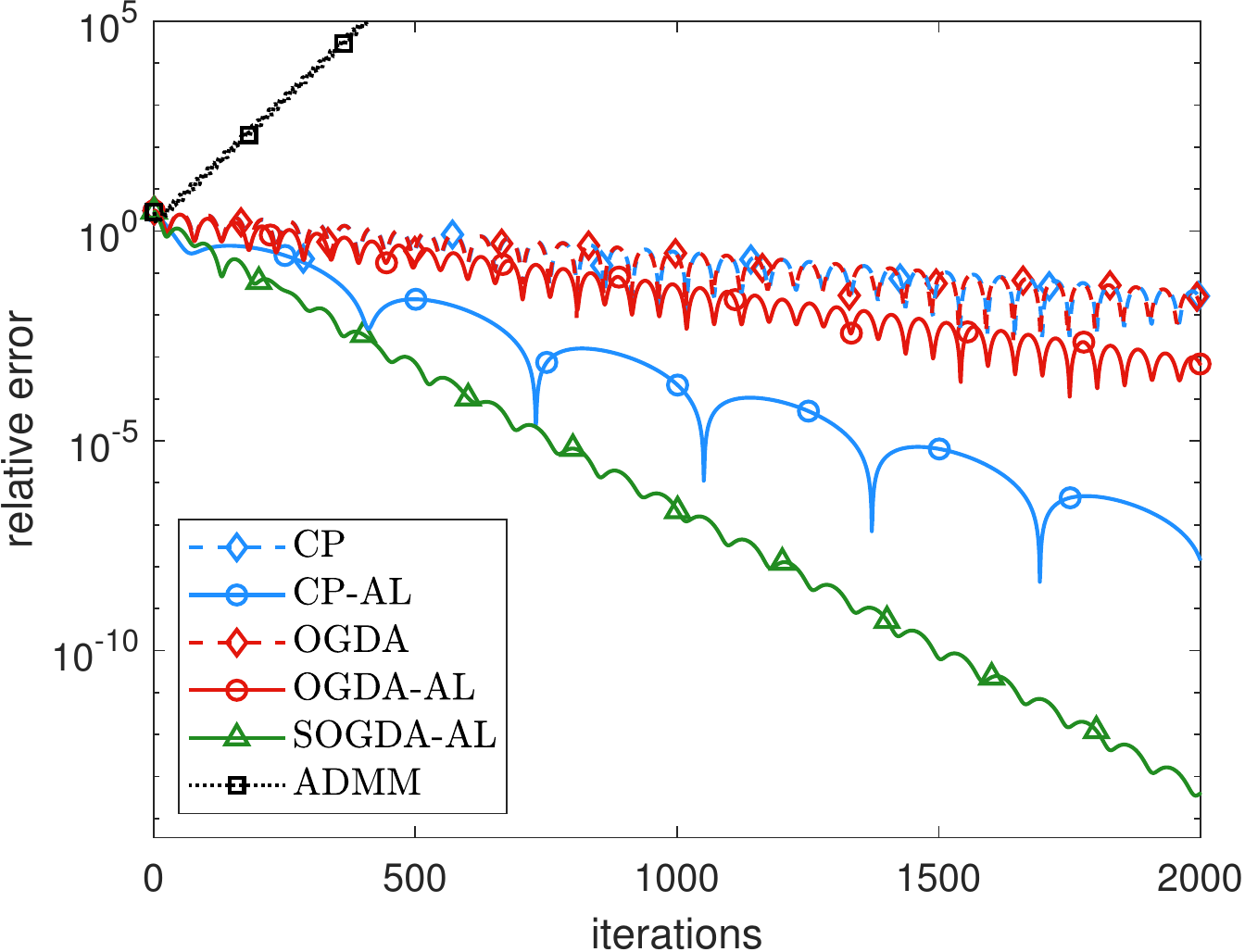}}
        \subfigure[Example2]{\includegraphics[scale = 0.42]{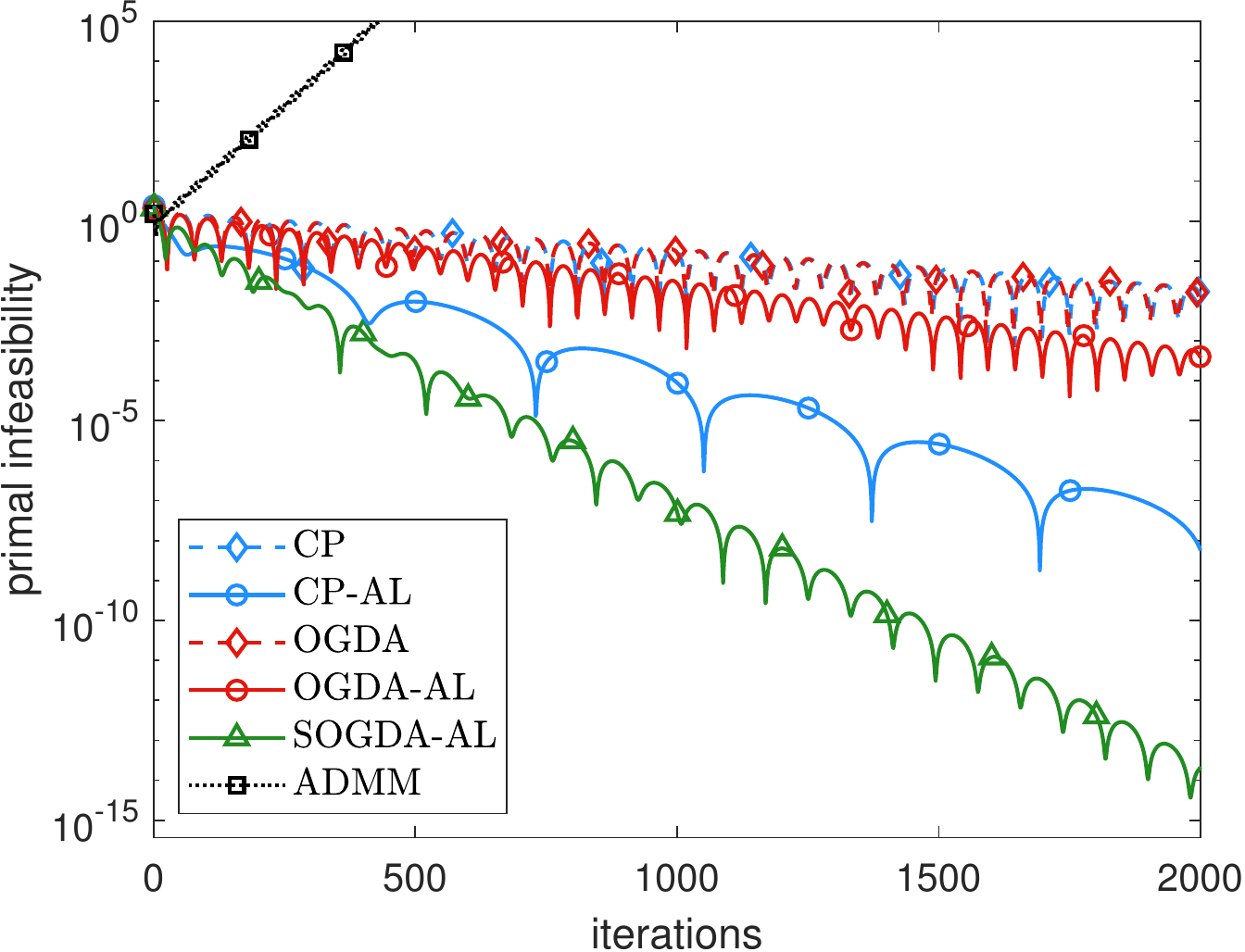}}
    \caption{Non-convergent examples for the direct extension of ADMM.}
        \label{fig:diverge_admm}
\end{figure*}

\section{Conclusions}
\label{sec:conclusions}
In this paper, we focus on the primal-dual methods for the conic constrained problems.
Several popular algorithms are unified into the proposed framework and
their $\mathcal{O}(1/N)$ ergodic convergence and linear convergence are analyzed under the entire framework.
In contrast to the existing theoretical results, 
we utilize the suboptimality and primal infeasibility to measure the optimality of the iterates, instead of the duality gap which is invalid without the boundedness assumption on the dual variable.
In addition, both the theory and experiments show that the penalty term added to the Lagrangian saddle point problem is helpful for convergence.
When compared to approximate dual ascent method such as ADMM, the primal-dual methods are easy to implement and enjoy better convergence guarantees, especially in the multi-block cases.
Preliminary numerical experiments verify the above points.

\section*{Appendix}
\begin{theorem}[Moreau's decomposition theorem \cite{moreau1962decomposition}]\label{Moreau}
    Let $\KK\subset \RR^n$ be a closed convex cone and $\KKc$ be its polar cone. For $x,y,z\in\RR^n$, the following statements are equivalent:
    \begin{itemize}
        \item $z=x+y,x\in\KK, y\in\KKc$ and $\langle x,y\rangle=0$,
        \item $x=\mathcal{P}_{\KK}(z)$ and $y=\mathcal{P}_{\KKc}(z)$.
    \end{itemize}
    \end{theorem}

The proof of basic lemmas is presented as follows. It is worth noting that when $\KK=\RR^{n}_{\leq 0}$, we have $\cpos{\cdot}=\pos{\cdot}$ being the standard entry-wise positive part, and $\cneg{\cdot}=\neg{\cdot}$ also.
\begin{proof}[Proof of Lemma \ref{lemma:algebra}]
    By the definition of $\cpos{\cdot}$ and $\cneg{\cdot}$, we have
    \begin{align*}
        \nrm{ w-w' }^2
        =&\nrm{ \cpos{w}-\cpos{w'}-\cneg{w}+\cneg{w'} }^2\\
        =&\nrm{ \cpos{w}-\cpos{w'} }^2 + \nrm{ \cneg{w}-\cneg{w'} }^2\\
        &+2\iprod{\cpos{w}}{\cneg{w'}}+2\iprod{\cpos{w'}}{\cneg{w}},
    \end{align*}
    where the second equality is due to the fact that $\iprod{\cpos{w}}{\cneg{w}}=\iprod{\cpos{w'}}{\cneg{w'}}=0$. Hence, the case $a>b$ is proven by noting that $\iprod{\cpos{w}}{\cneg{w'}}\geq 0$ and $\iprod{\cpos{w'}}{\cneg{w}}\geq0$. As of the case $a\geq b$, we have
    \begin{align*}
        &a^2\nrm{ w-w' }^2-(2a-b)\left[2a\iprod{ \cpos{w} }{ \cneg{w'} }+b\nrm{ \cpos{w}-\cpos{w'} }^2\right]\\
        =&(a-b)^2\nrm{ \cpos{w}-\cpos{w'} }^2+a^2\nrm{ \cneg{w}-\cneg{w'} }^2\\
        &-2a(a-b)\iprod{\cpos{w}}{\cneg{w'}}+2a^2\iprod{\cpos{w'}}{\cneg{w}}\\
        =&\nrm{(a-b)\left(\cpos{w}-\cpos{w'}\right)+a\left(\cneg{w}-\cneg{w'}\right)}^2\\
        &+2a(2a-b)\iprod{\cpos{w'}}{\cneg{w}}\\
        \geq& 0,
    \end{align*}
    which completes the proof of the first inequality. 

    For the second inequality, we consider
    \begin{align*}
        &\nrm{u-u'}^2+\nrm{v-v'}^2-\nrm{r-r'}^2\\
        =&\nrm{w+v-w'-v'}^2+\nrm{v-v'}^2-\nrm{\cpos{w}+v-\cpos{w'}-v'}^2\\
        =&\nrm{w-w'}^2+\nrm{v-v'}^2-\iprod{v-v'}{\cneg{w}-\cneg{w'}}-\nrm{\cpos{w}-\cpos{w'}}^2\\
        \geq&\nrm{\cneg{w}-\cneg{w'}}^2+\nrm{v-v'}^2-\iprod{v-v'}{\cneg{w}-\cneg{w'}}\\
        \geq&0.
    \end{align*}
\end{proof}

\begin{lemma}\label{xleq0}
    When $\KK=\{x|x\leq0\}$ and $\rho>0$, for any $i\in[m]$, the dual update rule of \eqref{generalalgo} is 
    \begin{equation}\label{lemma2}
        y_i^{k+1}=
        \begin{cases}
            0, & \omega_i>0\ \text{and}\ \kappa \nu_i\geq-\omega_i,\\
            \omega_i, & \omega_i\leq0\ \text{and}\ \nu_i\geq \omega_i,\\
            \frac{\omega_i+\kappa \nu_i}{\kappa+1}, &\text{otherwise}.
        \end{cases}
    \end{equation}
\end{lemma}
\begin{proof}
    According to the dual update rule in \eqref{generalalgo}, $y^{k+1}$ is the solution of the following equation:
    \[
        y = -\negp{\omega-\kappa\negp{\nu-y}}.
    \]
    We can give the solution of the equation through category discussion:
    \begin{enumerate}
    \item If $\nu_i\geq y_i$, $\negp{\nu_i-y_i}=0$ and hence $y_i = -\negp{\omega_i}$. 
    \item If $\nu_i<y_i$, the equation becomes $y_i = -\negp{\omega_i+\kappa(\nu_i-y_i)}$, then we consider the following two cases:
    \begin{itemize}
        \item When $\omega_i+\kappa \nu_i\geq 0$, it holds that $\omega_i+\kappa(\nu_i-y_i)\geq0$ due to the fact that $\kappa\geq0$ and $y_i\leq0$. Thus, we obtain $y_i=0$.
        \item When $\omega_i+\kappa \nu_i< 0$, the equation has a solution if and only if $\omega_i+\kappa(\nu_i-y_i)\leq0$, that is $y_i = \omega_i+\kappa(\nu_i-y_i)$. 
        Thus, the solution is $y_i=\frac{\omega_i+\kappa \nu_i}{\kappa+1}$.
    \end{itemize}
    \end{enumerate}
    By further simplification, we get \eqref{lemma2}.
    \qed
\end{proof}

\begin{lemma}\label{normlemma}
    For any $a\in\mathbbm{R}^n$ and $b\in-\KK$, it holds that $\|\cpos{a+b}\|\geq\|\cpos{a}\|$.
\end{lemma}
\begin{proof}
Write $c=\cpos{a+b}\in\KK^{\circ}, d=\cneg{a+b}\in-\KK$, then $a+b=c-d$. By definition of $\mathcal{P}_{\KK}$, we have
$$
\nrm{a-\mathcal{P}_{\KK}(a)}=\min_{x\in\KK}\|a-x\|\leq \nrm{a-(-b-d)}=\|c\|.
$$
Hence, $\nrm{\cpos{a}}=\nrm{a+\cneg{a}}=\nrm{a-\mathcal{P}_{\KK}(a)}\leq\|c\|=\nrm{\cpos{a+b}}$.
\end{proof}

% BibTeX users please use one of
% \bibliographystyle{spbasic}      % basic style, author-year citations
\bibliographystyle{spmpsci}      % mathematics and physical sciences
\bibliography{references}   % name your BibTeX data base

\end{document}